\documentclass[a4paper,10pt]{article}
\usepackage{amssymb,amscd,amsfonts,amsbsy}
\usepackage{enumerate}
\usepackage{amsmath,amsthm,amssymb,amsfonts}
\usepackage{mathrsfs}
\usepackage{epsf,epsfig}
\usepackage{graphics}        
\usepackage{graphicx}

\input amssymb.sty
\input euscript.sty

\numberwithin{equation}{section}

\newcommand{\meanint}{{\int{\mkern-19mu}-}}
\newcommand{\dmeanint}{{\int{\mkern-16mu}-}}

\newtheorem{proposition}{Proposition}[section]
\newtheorem{theorem}[proposition]{Theorem}
\newtheorem{lemma}[proposition]{Lemma}
\newtheorem{corollary}[proposition]{Corollary}

\newtheorem{remark}[proposition]{Remark}

\newtheorem{definition}[proposition]{Definition}

\begin{document}

\title{Extending Sobolev Functions with Partially Vanishing Traces
from Locally $(\varepsilon,\delta)$-Domains
and Applications to Mixed Boundary Problems}

\author{Kevin Brewster, Dorina Mitrea\thanks{This author has been supported in part 
by a Simons Foundation grant.}\,\,\,\,
Irina Mitrea\thanks{This author has been supported in part 
by NSF grant DMS-1201736.}\,\,\,\,
and Marius Mitrea\thanks{2010{\it Mathematics Subject Classification:} 
Primary: 46E35, 42B35, 35J58, 35G45; Secondary 46B70, 31C15.
}}

\date{\today}

\maketitle

\begin{abstract}
We prove that given any $k\in{\mathbb{N}}$, for each open set $\Omega\subseteq{\mathbb{R}}^n$ and any closed subset $D$ of $\overline{\Omega}$ 
such that $\Omega$ is locally an $(\varepsilon,\delta)$-domain near 
$\partial\Omega\setminus D$ there exists a linear and bounded extension 
operator ${\mathfrak{E}}_{k,D}$ mapping, for each $p\in[1,\infty]$, 
the space $W^{k,p}_D(\Omega)$ into $W^{k,p}_D({\mathbb{R}}^n)$. Here, 
with ${\mathcal{O}}$ denoting either $\Omega$ or ${\mathbb{R}}^n$,
the space $W^{k,p}_D({\mathcal{O}})$ is defined as the completion in the classical 
Sobolev space $W^{k,p}({\mathcal{O}})$ of (restrictions to ${\mathcal{O}}$ of) functions
from ${\mathscr{C}}^\infty_c({\mathbb{R}}^n)$ whose supports are disjoint from $D$. 
In turn, this result is used to develop a functional analytic theory for 
the class $W^{k,p}_D(\Omega)$ (including intrinsic characterizations, boundary 
traces and extensions results, interpolation theorems, among other things) 
which is then employed in the treatment of mixed boundary value
problems formulated in locally $(\varepsilon,\delta)$-domains.
\end{abstract}

\section{Introduction}
\label{Sect:1}
\setcounter{equation}{0}

Extension results for Sobolev spaces defined on open subsets of the Euclidean 
ambient are important tools in many branches of mathematics, including 
harmonic analysis, potential theory, and partial differential equations. 
Recall that, given $k\in{\mathbb{N}}$ and $p\in[1,\infty]$, an open 
set $\Omega\subseteq{\mathbb{R}}^n$ is called a $W^{k,p}$-extension domain 
if there exists a bounded linear operator
\begin{eqnarray}\label{TFh-UN}
E:W^{k,p}(\Omega)\longrightarrow W^{k,p}({\mathbb{R}}^n)
\end{eqnarray}
with the property that $(Eu)\big|_{\Omega}=u$ for each $u\in W^{k,p}(\Omega)$
(for background definitions the reader is referred to \S\,\ref{Sect:2}).  
Such a condition necessarily imposes restrictions on the underlying set $\Omega$.  
For example, not all functions from the Sobolev space $W^{k,p}(\Omega)$ 
with $k\in{\mathbb{N}}$ and $p\in(n/k,\infty)$ may be extended 
to $W^{k,p}({\mathbb{R}}^n)$, $n\geq 2$, in the case in which 
\begin{eqnarray}\label{Yaytf}
\Omega_a=\big\{(x_1,\dots,x_n)\in{\mathbb{R}}^n:\,0<x_1,\dots,x_{n-1}<1
\,\,\mbox{ and }\,\,0<x_n<x_{n-1}^a\big\}\,\,\mbox{ with }\,\,a>kp-1.
\end{eqnarray}
Indeed, the fact that $p>n/k$ ensures that
$W^{k,p}({\mathbb{R}}^n)\hookrightarrow{\mathscr{C}}^0({\mathbb{R}}^n)$ and
yet if $b>0$ is small enough so that $a>p(k+b)-1$, then  
the function $u(x):=x_{n-1}^{-b}$ for each $x=(x_1,\dots,x_n)\in\Omega_a$ belongs 
to $W^{k,p}(\Omega_a)$, but obviously has no continuous extension to ${\mathbb{R}}^n$. 
The obstruction in this case is the presence of outward cusps on $\partial\Omega_a$
(caused by the fact that $a>kp-1$, $p>n/k$, and $n\geq 2$ necessarily entails $a>1$).

On the positive side, a classical result in harmonic analysis asserts that any 
Lipschitz domain is a $W^{k,p}$-extension domain for all $k\in{\mathbb{N}}$ 
and $p\in[1,\infty]$. The first breakthrough came in the work of A.P.~Calder\'on 
in \cite{Cal61} where, for each given Lipschitz domain $\Omega\subseteq{\mathbb{R}}^n$
and each given $k\in{\mathbb{N}}$, a linear extension operator $E_k$ is constructed
which maps $W^{k,p}(\Omega)$ boundedly into $W^{k,p}({\mathbb{R}}^n)$ for each 
$p\in(1,\infty)$, and which has the additional property that
\begin{eqnarray}\label{Yaytf.RRD}
{\rm supp}\big(E_k u\big)\subseteq\overline{\Omega}\,\,\,
\mbox{ for every }\,\,u\in{\mathscr{C}}_c^\infty(\Omega).
\end{eqnarray}
In the same geometrical setting but via a different approach, E.M.~Stein 
has produced (see the exposition in \cite[Theorem~5, p.\,181]{Ste70})
an extension operator which, as opposed to Calder\'on's, is universal  
in the sense that it does not depends on the order of smoothness (and, of course,
the integrability exponent), and which is also bounded in the limiting cases
$p=1$ and $p=\infty$. Nonetheless, Stein's operator no longer enjoys 
property \eqref{Yaytf.RRD} as it scatters the support of the function on 
which it acts across the boundary. 

Both original proofs of Calder\'on's and Stein's theorems make essential 
use of the fact that Lipschitz domains satisfy a uniform cone property. The 
latter property actually characterizes Lipschitzianity, so new ideas must be 
involved if the goal is to go establish extension results beyond this class of domains.
Via a conceptually novel approach, which builds on the seminal work of 
H.~Whitney on his extension theorem for Lipschitz functions in \cite{Whit34}, 
P.W.~Jones succeeded (cf. \cite[Theorem~1, p.\,73]{Jon81}) in generalizing the 
results of Calder\'on and Stein to a much larger class of sets, which he called
$(\varepsilon,\delta)$-domains. Jones also proved that a finitely connected open 
set $\Omega\subseteq{\mathbb{R}}^2$ is a $W^{k,p}$-extension domain for 
all $k\in{\mathbb{N}}$ and $p\in[1,\infty]$ if and only if $\Omega$ 
is $(\varepsilon,\delta)$-domains for some values $\varepsilon,\delta>0$ 
(cf. \cite[Theorem~3, p.\,74]{Jon81}).
Since Jones' class of domains is going to be of basic importance for the goals
we have in mind, below we record its actual definition. 

\begin{definition}\label{Def-EPDE}
Assume that $\varepsilon\in(0,\infty)$ and $\delta\in(0,\infty]$.
A nonempty, open, proper subset $\Omega$ of ${\mathbb{R}}^n$ is called an
{\tt $(\varepsilon,\delta)$-domain} if for any $x,y\in\Omega$ with $|x-y|<\delta$ 
there exists a rectifiable curve $\gamma:[0,1]\to\Omega$ such that
$\gamma(0)=x$, $\gamma(1)=y$, and 
\begin{eqnarray}\label{jcdj}
{\rm length}(\gamma)\leq\frac{1}{\varepsilon}|x-y|\quad\mbox{and}\quad
\frac{\varepsilon|z-x|\,|z-y|}{|x-y|}
\leq {\rm dist}\,(z,\partial\Omega),\quad\forall\,z\in\gamma([0,1]).
\end{eqnarray}
\end{definition}

Informally, the first condition in \eqref{jcdj} says that $\Omega$ is locally
connected in some quantitative sense, while the second condition in 
\eqref{jcdj} says that there exists some type of ``tube'' $T$, with 
$\gamma([0,1])\subset T\subset\Omega$ and the width of $T$ at a point $z$ 
on the curve is of the order $\min\{|z-x|,|z-y|\}$. 

Examples of $(\varepsilon,\delta)$-domains include bi-Lipschitz images
of Lipschitz domains, open sets whose boundaries are given locally as graphs 
of functions in the Zygmund class $\Lambda_1$, or of functions with gradients in 
the John-Nirenberg space ${\rm BMO}$, as well as the classical van Koch snowflake 
domain of conformal mapping theory. The boundary of an $(\varepsilon,\delta)$-domain 
can be highly nonrectifiable and, in general, no regularity condition on 
$\partial\Omega$ can be inferred from the $(\varepsilon,\delta)$ property
described in Definition~\ref{Def-EPDE}. 
The fact that, in general, $(\varepsilon,\delta)$-domains are not sets of 
finite perimeter can be seen from the fact that the classical van Koch snowflake 
domain does not have finite perimeter. In fact, for each $d\in[n-1,n)$ there 
exists an open set $\Omega\subseteq{\mathbb{R}}^n$ such that $\Omega$ is 
an $(\varepsilon,\infty)$-domain for some $\varepsilon=\varepsilon(d)\in(0,\infty)$
and $\partial\Omega$ has Hausdorff dimension $d$. 
This being said, it has been shown in \cite[Lemma~2.3, p.\,77]{Jon81} that 
\begin{eqnarray}\label{PJE-wcX}
\mbox{any $(\varepsilon,\delta)$-domain
$\Omega\subseteq{\mathbb{R}}^n$ satisfies ${\mathscr{L}}^n(\partial\Omega)=0$},
\end{eqnarray}
where ${\mathscr{L}}^n$ denotes the Lebesgue measure in ${\mathbb{R}}^n$.

Jone's $(\varepsilon,\delta)$-domains interface tightly with the category of 
uniform domains considered a little earlier by O.~Martio and J.~Sarvas in \cite{MS79}.
Recall that a nonempty, proper, open subset $\Omega$ of ${\mathbb{R}}^n$ is said to 
be a {\tt uniform domain} provided there exists a constant $c=c(\Omega)\in[1,\infty)$ 
with the property that each pair of points $x_1,x_2\in\Omega$ can be 
joined by a rectifiable curve $\gamma$ in $\Omega$ for which
\begin{eqnarray}\label{TYDY-854}
{\rm length}\,(\gamma)\leq c\,|x_1-x_2|\,\,\mbox{ and }\,\,
\min_{j=1,2}|x_j-x|\leq c\,{\rm dist}\,(x,\partial\Omega)
\,\,\mbox{ for each }\,\,x\in\gamma. 
\end{eqnarray}
Also, call $\Omega$ a {\tt locally uniform domain} if there exist
$c,r\in(0,\infty)$ with the property that \eqref{TYDY-854} holds whenever 
$|x_1-x_2|<r$. 
Then a nonempty, proper, open subset of the Euclidean space 
is an $(\varepsilon,\delta)$-domain for some $\varepsilon,\delta\in(0,\infty)$
if and only if it is a locally uniform domain. 
Moreover, if $\Omega$ is a uniform domain, then $\Omega$ satisfies an 
interior corkscrew condition as well as a Harnack chain condition, in the 
sense of D.~Jerison and C.E.~Kenig (cf. \cite{JeKe82}). Conversely, if $\Omega$ 
satisfies an interior corkscrew condition, a Harnack chain condition, and has 
the property that $\partial\Omega$ is bounded, then $\Omega$ is a uniform domain. 
The interested reader is referred to \cite[Propositions~A.2-A.3]{Han05} for 
more details in this regard. As a consequence, here we only wish to note that 
the class of $(\varepsilon,\delta)$-domains with a compact boundary coincides
with the category of {\tt one-sided NTA domains} (i.e., domains satisfying 
an interior corkscrew condition and a Harnack chain condition), from \cite{JeKe82}.

\vskip 0.10in

Returning to the issue of extension results for Sobolev spaces, 
the following is \cite[Theorem~1, p.\,73]{Jon81}.

\begin{theorem}\label{TGKV885}
Let $\Omega$ be a finitely connected $(\varepsilon,\delta)$-domain in ${\mathbb{R}}^n$
and fix $k\in{\mathbb{N}}$. Then there exists a linear operator $\Lambda_k$ mapping 
$W^{k,p}(\Omega)$ boundedly into $W^{k,p}({\mathbb{R}}^n)$ for each $p\in[1,\infty]$,
and such that $(\Lambda_k u)\big|_{\Omega}=u$ for each $u\in W^{k,p}(\Omega)$.
\end{theorem}

Since its introduction in the early 80's, Jones' extension operator $\Lambda_k$
has been the focal point of a considerable amount of work. For example, 
R.\,DeVore and R.\,Sharpley have successfully adapted Jones' ideas
as to construct in \cite{DeVSh} an extension operator for Besov spaces on 
$(\varepsilon,\delta)$-domains in ${\mathbb{R}}^n$. Furthermore, B.L.\,Fain 
in \cite{Fain}, A.\,Seeger in \cite{Se}, and P.A.\,Shvartsman in \cite{Shv} 
have generalized Jones' theorem to anisotropic Sobolev spaces, while S.-K.\,Chua 
has proved weighted versions of Jones' theorem, involving Muckenhoupt weights 
in \cite{Chua}, and doubling weights satisfying a Poincar\'e inequality 
in \cite{Chua2}. Here we also wish to mention the work of N.\,Garofalo and 
D.M.\,Nhieu in \cite{GN98} where the authors have established extension theorems 
for Sobolev functions in Carnot-Carath\'eodory spaces in a suitable analogue of
the class of $(\varepsilon,\delta)$-domains for this setting. 
Another significant development appeared in \cite{Rog06} where L.G.~Rogers
combines the techniques of P.~Jones and E.~Stein to produce an extension 
operator for Sobolev functions in $(\varepsilon,\delta)$-domains which, 
as opposed to Jones', is universal. However, Rogers' hybrid operator scatters 
supports of functions across the boundary even more severely than 
Jones's extension operator (which already fails to satisfy property \eqref{Yaytf.RRD}).  

A version of Jones'extension operator in $(\varepsilon,\delta)$-domains which 
is more in line with the original design from \cite{Jon81} is due to M.\,Christ 
(cf. \cite{Christ}), who has shown that a mild alteration renders Jones' operator
semi-universal (i.e., it simultaneously extends functions with preservation of 
class up to any desired, a priori given, threshold of smoothness). 
Moreover, M.\,Christ works (cf. also \cite{Mi4}) with a more general scale, 
which he denotes by ${\mathfrak{N}}^p_\alpha(\Omega)$, $1<p\leq\infty$, $\alpha>0$,
originally introduced by R.A.\,DeVore and R.C.\,Sharpley in \cite{DS1}. Indeed,  
${\mathfrak{N}}^p_\alpha(\Omega)$ turns out to be a genuine Sobolev space in the 
case when $\alpha$ is an integer (cf. \cite{DS1}), and a Triebel-Lizorkin space 
otherwise (cf. \cite{KMM}). The semi-universality character of M.\,Christ's extension 
is going to have some degree of significance for our work. This being said, the 
main issues we are presently concerned with (see below) have, to the best our 
knowledge, never been addressed before. 

\vskip 0.10in

To gain a broader perspective let us now revisit the concept of $W^{k,p}$-extension 
domain and introduce an extra nuance. Specifically, fix $k\in{\mathbb{N}}$
and $p\in[1,\infty]$ then, given an arbitrary nonempty
open set $\Omega\subseteq{\mathbb{R}}^n$ and a closed linear subspace $V$ of 
$W^{k,p}(\Omega)$, call $\Omega$ a $V$-{\it extension domain} provided
there exits a linear and bounded operator $E:V\to W^{k,p}({\mathbb{R}}^n)$
such that $(Eu)\big|_{\Omega}=u$ for each $u\in V$. In this terminology,  
any nonempty open set $\Omega\subseteq{\mathbb{R}}^n$ is a 
$\mathring{W}^{k,p}(\Omega)$-extension domain (taking $E$ to be the operator 
of extension by zero outside $\Omega$) while, at the other end of the spectrum,
Theorem~\ref{TGKV885} comes to assert that any finitely connected
$(\varepsilon,\delta)$-domain $\Omega$ in ${\mathbb{R}}^n$ is 
a $W^{k,p}(\Omega)$-extension domain. As an intermediate case, note that
given $k\in{\mathbb{N}}$ and $p\in(n/k,\infty)$, even though the set $\Omega_a$ 
from \eqref{Yaytf} fails to be a $W^{k,p}(\Omega_a)$-extension domain
it is a $V$-extension domain for any space of the form 
\begin{eqnarray}\label{Uah9GGa8732}
V:=\big\{u\in W^{k,p}(\Omega_a):\,u\equiv 0\,\,\mbox{ on }\,\,O\cap\Omega_a\big\}
\end{eqnarray}
where $O$ is some neighborhood of the cuspidal edge 
$C:=\big\{(x_1,...,x_{n-2},0,0):\,0\leq x_1,...,x_{n-2}\leq 1\big\}$ of $\Omega_a$.
This is clear from Calder\'on's extension theorem and the fact that 
$\partial\Omega_a$ is a Lipschitz surface away from $C$. 
In light of this discussion, it is very much apparent that 
the geometry of a nonempty open set $\Omega$ influences the nature of the
linear closed spaces $V\subseteq W^{k,p}(\Omega)$ for which 
$\Omega$ is a $V$-extension domain.

By way of analogy, suppose now that $\Omega$ is an open, nonempty subset of
${\mathbb{R}}^n$ which satisfies the $(\varepsilon,\delta)$ condition from
Definition~\ref{Def-EPDE} only near a relatively open portion $N$ of its 
topological boundary $\partial\Omega$. The question now becomes: {\it for what closed 
linear subspaces $V$ of $W^{k,p}(\Omega)$ is $\Omega$ a $V$-extension domain}? 
Our earlier analysis suggests considering the closure in $W^{k,p}(\Omega)$ of 
restrictions to $\Omega$ of functions from ${\mathscr{C}}^\infty_c({\mathbb{R}}^n)$ 
whose support is disjoint from the rough portion of the boundary, i.e., $\partial\Omega\setminus N$.

More generally, in the case when $\Omega$ is a nonempty open set in ${\mathbb{R}}^n$ 
and $D$ (playing the role of $\partial\Omega\setminus N$) is an arbitrary closed 
subset of $\overline{\Omega}$, introduce
\begin{eqnarray}\label{TD-TRY753}
W^{k,p}_D(\Omega):=\,\mbox{the closure of }\,\, 
\big\{\varphi\big|_{\Omega}:\,\varphi\in{\mathscr{C}}^\infty_c({\mathbb{R}}^n)
\,\,\mbox{ with }\,\,D\cap{\rm supp}\,\varphi=\emptyset\big\}\,\,\mbox{ in }\,\,
W^{k,p}(\Omega).
\end{eqnarray}
Then our first main result in this paper is the following extension result 
(for a more general formulation see Theorem~\ref{cjeg.5AD}; also,
the class of sets which are locally $(\varepsilon,\delta)$-domains
near a portion of their boundary is introduced in Definition~\ref{OM-Oah-Y88}). 

\begin{theorem}\label{TGKV885.2}
Let $\Omega\subseteq{\mathbb{R}}^n$ and $D\subseteq\overline{\Omega}$
be such that $D$ is closed and $\Omega$ is locally an $(\varepsilon,\delta)$-domain 
near $\partial\Omega\setminus D$. Then for any $k\in{\mathbb{N}}$ there exists 
an operator ${\mathfrak{E}}_{k,D}$ such that for any $p\in[1,\infty]$ one has
\begin{eqnarray}\label{Tgb-11VVV.a6yH}
&& {\mathfrak{E}}_{k,D}:W^{k,p}_D(\Omega)\longrightarrow W^{k,p}_D({\mathbb{R}}^n)
\quad\mbox{linearly and boundedly, and}
\\[4pt]
&& \big({\mathfrak{E}}_{k,D}\,u\big)\big|_{\Omega}=u,\quad
\mbox{${\mathscr{L}}^n$-a.e. on $\Omega$ for every $u\in W^{k,p}_D(\Omega)$}.
\label{Pa-PL2VB.a6yH}
\end{eqnarray}
\end{theorem}

Theorem~\ref{TGKV885.2} bridges between the ``trivial" extension of functions
from $\mathring{W}^{k,p}(\Omega)$ by zero outside an arbitrary open set 
$\Omega$, to which it reduces in the case when $D:=\partial\Omega$, and Jones' 
extension operator $\Lambda_k$, which our ${\mathfrak{E}}_{k,D}$ becomes in the
case when $D:=\emptyset$. Indeed, such a choice forces $\Omega$ to be a genuine $(\varepsilon,\delta)$-domain and, as explained in the proof of Corollary~\ref{cjeg.5}, 
our operator ${\mathfrak{E}}_{k,D}$ automatically reduces to Jones' extension operator $\Lambda_k$ for this class of domains (irrespective of the nature of $D$). 
As such, Theorem~\ref{TGKV885.2} brings to light a 
novel basic feature of Jones' extension operator, i.e., the property that
for any $(\varepsilon,\delta)$-domain $\Omega$ and any closed set $D\subseteq\Omega$,
the operator $\Lambda_k$ maps the subspace $W^{k,p}_D(\Omega)$ of $W^{k,p}(\Omega)$
into the subspace $W^{k,p}_D({\mathbb{R}}^n)$ of $W^{k,p}({\mathbb{R}}^n)$
(with $D:=\emptyset$ yielding Theorem~\ref{TGKV885}, at least when $p\not=\infty$). 
It should be noted that Theorem~\ref{TGKV885} would readily imply this more refined 
version of itself if Jones' extension operator were to be support preserving, i.e., 
if ${\rm supp}\big(\Lambda_k u\big)={\rm supp}\,u$ for each $u\in W^{k,p}(\Omega)$, 
but this is far from being the case. In fact, as already remarked by P.W.~Jones 
in \cite{Jon81}, his extension operator $\Lambda_k$ lacks even the weaker 
property that ${\rm supp}\big(\Lambda_k u\big)\subseteq\overline{\Omega}$ 
for every $u\in{\mathscr{C}}_c^\infty(\Omega)$ which, as pointed out earlier, 
Calder\'on's extension operator enjoys (in the setting of Lipschitz domains, 
of course). 

The main ingredients in the proof of Theorem~\ref{TGKV885.2} are 
Jones' extension result stated in Theorem~\ref{TGKV885}, augmented with the 
property that
\begin{eqnarray}\label{fgghvk}
\overline{\Omega}\cap\,{\rm supp}\,(\Lambda_k u)
={\rm supp}\,u,\qquad\forall\,u\in W^{k,p}(\Omega),
\end{eqnarray}
and the existence of a quantitative partition of unity which is geometrically 
compatible with our notion of locally $(\varepsilon,\delta)$-domain near a 
portion of its boundary. The key non-expansive support property \eqref{fgghvk}
is proved in Theorem~\ref{cj-TRf} via a careful inspection of the format of 
Jones' extension operator, recounted in \eqref{PJE-a1}. One may regard 
\eqref{fgghvk} as a vestigial form of property \eqref{Yaytf.RRD}, in the 
more general context of $(\varepsilon,\delta)$-domains.
In fact, it is possible to prove a version of Theorem~\ref{TGKV885.2} in which 
the intervening extension operator is semi-universal, at least if $1<p\leq\infty$. 
We do so in Theorem~\ref{cjeg.5AD.UUV}, whose proof relies on M.\,Christ's alteration 
of Jones' extension operator, recorded in Theorem~\ref{cjeg.UUV}. This extra  
feature is important in the context of interpolation with change of smoothness, 
although we shall not pursue this in the present paper. 

\vskip 0.10in

Given that, in the special cases when $D=\partial\Omega$ and $D=\emptyset$
the space $W^{k,p}_D(\Omega)$ becomes, respectively, $\mathring{W}^{k,p}(\Omega)$ and
$W^{k,p}(\Omega)$ (for $p\not=\infty$), we shall refer to $W^{k,p}_D(\Omega)$
as a {\it Sobolev space with a partially vanishing trace} (on the set $D$). 
We stress that, in general, the set $D\subseteq{\overline{\Omega}}$ is not 
necessarily assumed to be contained in the boundary of $\Omega$, and that the 
terminology ``vanishing trace" warrants further clarification. 
The reader is referred to Theorem~\ref{YTab-YHb} for a formal statement 
in which the vanishing of the higher-order restriction of $u$ to $D$ is
formulated in an appropriate capacitary sense. See also Theorem~\ref{YTah-YYHa9}
where the aforementioned restriction is interpreted in the sense of 
${\mathcal{H}}^d$, the $d$-dimensional Hausdorff measure, 
in the case when $D$ is a closed subset of $\overline{\Omega}$ 
which is $d$-Ahlfors regular for some $d\in(0,n)$ (a piece of terminology explained
in \eqref{Fq-A.4}). Finally, in Theorem~\ref{Kance.745} we are able to 
describe $W^{k,p}_D(\Omega)$ as the space consisting of those $u\in W^{k,p}(\Omega)$ 
whose intrinsic restriction to $D$, as functions defined in $\Omega$, vanishes
${\mathcal{H}}^d$-a.e. on $D$. This is done under the assumption that 
$\Omega$ is an $(\varepsilon,\delta)$-domain and $D$ is a closed subset 
of $\overline{\Omega}$ which is $d$-Ahlfors regular for some $d\in(0,n)$.

The key ingredients in the proofs of the structure theorems for the spaces
$W^{k,p}_D(\Omega)$ from Theorem~\ref{YTab-YHb} and Theorem~\ref{YTah-YYHa9} 
are the extension result from Theorem~\ref{TGKV885.2}, a deep result of L.I.~Hedberg 
and T.H.~Wolff from \cite{HW83} stating that any closed set in ${\mathbb{R}}^n$ has 
the so-called $(k,p)$-synthesis property, for any $p\in(1,\infty)$ and any
$k\in{\mathbb{N}}$, along with the trace/extension theory on $d$-Ahlfors regular 
subsets of ${\mathbb{R}}^n$ developed by A.~Jonsson and H.~Wallin in \cite{JoWa84}.
The intrinsic characterization of the spaces $W^{k,p}_D(\Omega)$ from 
Theorem~\ref{Kance.745} (reviewed in the earlier paragraph) requires refining 
the Jonsson-Wallin theory in several important regards. This is accomplished
in Theorem~\ref{YTah-YYHa8}, Theorem~\ref{Ohav-7UJ.2}, and Theorem~\ref{NIceTRace}
which, in turn, are used to study the issue of preservation of Sobolev class
under extension by zero, in Theorem~\ref{Tfgg-75dS}, and under gluing 
functions with matching traces, in Theorem~\ref{SSSg-g5dS}.

Collectively, these results amount to a robust functional analytic theory for the
category of Sobolev spaces with partially vanishing traces introduced in 
\eqref{TD-TRY753}. Along the way, we also answer a recent question posed to us 
by D.~Arnold (cf. Theorem~\ref{Pcsd5} and Corollary~\ref{Uafav976f4} for 
precise statements, in various degrees of generality), and provide a solution 
to a question raised by J.\,Ne\v{c}as in 1967. Specifically, Problem~4.1 on 
p.\,91 of \cite{Nec} asks whether for any Lipschitz domain $\Omega$ in 
${\mathbb{R}}^n$, any $k\in{\mathbb{N}}$ and $p\in(1,\infty)$, one has
\begin{eqnarray}\label{Necas.YYY}
\mathring{W}^{k,p}(\Omega)=
\Big\{u\in W^{k,p}(\Omega):\,\frac{\partial^j u}{\partial\nu^j}=0\,\,
\mbox{ ${\mathcal{H}}^{n-1}$-a.e. on }\partial\Omega\,\,\mbox{ for }\,\,0\leq j\leq k-1
\Big\},
\end{eqnarray}
where $\frac{\partial^j}{\partial\nu^j}$ denotes the $j$-th iterated directional 
derivative with respect to the outward unit normal $\nu$ to $\Omega$ (suitably defined). 
In Theorem~\ref{NIcBBB} we prove that this is the case even in the considerably 
more general setting when $\Omega$ is an $(\varepsilon,\delta)$-domain 
in ${\mathbb{R}}^n$ with the 
property that $\partial\Omega$ is $(n-1)$-Ahlfors regular
and such that its measure theoretic boundary, $\partial_\ast\Omega$, has full 
${\mathcal{H}}^{n-1}$-measure in $\partial\Omega$. The latter condition, 
i.e. that ${\mathcal{H}}^{n-1}(\partial\Omega\setminus\partial_*\Omega)=0$, 
merely ensures that the geometric measure theoretic outward unit normal $\nu$ 
to $\Omega$ is defined ${\mathcal{H}}^{n-1}$-a.e. on $\partial\Omega$. 
In this vein, it is worth noting that, by Rademacher's a.e. differentiability 
theorem for Lipschitz functions, this is always the case in the category of 
Lipschitz domains.

While the functional analytic study of the spaces $W^{k,p}_D(\Omega)$ undertaken 
in the first part of the paper is of independent interest, the principal 
motivation for such an endeavor remains its impact on the study of partial 
differential equations. In particular, the specific nature of the spaces
$W^{k,p}_D(\Omega)$ from \eqref{TD-TRY753} naturally makes the body of 
results established here particularly well-suited for the treatment of boundary 
value problems of mixed type in very general classes of Euclidean domains. 
Mixed boundary value problems arises naturally in connection to a series of 
important problems in mathematical physics and engineering, dealing with 
conductivity, heat transfer, wave phenomena, electrostatics, metallurgical melting, 
stamp problems in elasticity and hydrodynamics, among many other applications. 
Specific references can be found in \cite{AzKr}, \cite{Da}, \cite{Fabr}, \cite{Gr},
\cite{Grog}, \cite{Lag}, \cite{MazRo}, \cite{MiMi}, \cite{Pr}, \cite{ReSh}, \cite{Sav}, 
\cite{Sim}, \cite{Sham}, \cite{Sned}, to cite just a fraction of a vast literature
on this topic. 

In the last section of our paper we formulate and solve such mixed boundary problems 
for strongly elliptic higher-order systems in bounded open subsets $\Omega$ of ${\mathbb{R}}^n$ which are locally $(\varepsilon,\delta)$-domains near 
$\partial\Omega\setminus D$ with $D$ closed subset of $\partial\Omega$ 
which is $d$-Ahlfors regular for some $d\in(n-2,n)$. In such a scenario, 
$D$ is the portion of the boundary on which a homogeneous higher-order Dirichlet
condition is imposed, while a homogeneous Neumann condition is assigned on 
$\partial\Omega\setminus D$. In this connection, we wish to note that the class
of domains just described is much more general than those previously considered 
in the literature. The following is a slightly sanitized version of our main 
well-posedness result proved in Theorem~\ref{yaUNDJ}. 

\begin{theorem}\label{yaUHVav7f39}
Let $\Omega$ be a bounded, connected, open, nonempty, subset of ${\mathbb{R}}^n$, 
$n\geq 2$, and suppose that $D$ is a nonempty closed subset of $\partial\Omega$ 
which is $d$-Ahlfors regular for some $d\in(n-2,n)$. In addition, assume that 
$\Omega$ is locally an $(\varepsilon,\delta)$-domain near $\partial\Omega\setminus D$,  
and consider a strongly elliptic, divergence-form system ${\mathcal{L}}$ 
of order $2m$, whose tensor coefficient $A$ consists of bounded measurable functions
in $\Omega$. 

Then there exists $p_\ast\in(2,\infty)$ with the following significance. If 
\begin{eqnarray}\label{S-EllPka.ii33}
\frac{p_\ast}{p_\ast-1}<p<p_\ast
\end{eqnarray}
then the mixed boundary value problem 
\begin{eqnarray}\label{Tgb-5bb.ii33}
\left\{
\begin{array}{l}
{\mathcal{L}}u=f\lfloor_{\,\Omega}\,\,\mbox{ in }\,\,{\mathcal{D}}'(\Omega),
\\[6pt]
u\in W^{m,p}_D(\Omega),
\\[6pt]
\partial^A_\nu(u,f)=0\,\,\mbox{ on }\,\,\partial\Omega\setminus D,
\end{array}
\right.
\end{eqnarray}
is uniquely solvable for each functional $f\in\big(W^{m,p'}_D(\Omega)\big)^*$,
where $1/p+1/p'=1$.
\end{theorem}

\noindent Above, $f\lfloor_{\,\Omega}$ denotes the distribution in $\Omega$ 
canonically associated with the functional $f\in\big(W^{m,p'}_D(\Omega)\big)^*$,
and the homogeneous Neumann boundary condition $\partial^A_\nu(u,f)=0$ 
on $\partial\Omega\setminus D$ is understood in a variational sense, 
made clear in Definition~\ref{Ian-Yab679}. We also remark that the membership 
of $u$ to $W^{m,p}_D(\Omega)$ automatically implies (by virtue of results mentioned
earlier) that the higher-order restriction of $u$ to $D$ vanishes 
at ${\mathcal{H}}^d$-a.e. point on $D$. Thus a homogeneous Dirichlet boundary 
condition on $D$ is implicit, making it clear that problem \eqref{Tgb-5bb.ii33} 
has a mixed character. 

Specializing \eqref{Tgb-5bb.ii33} to the particular case when $D=\partial\Omega$ 
yields a well-posedness result for the inhomogeneous Dirichlet problem. For the
sake of this introduction, we choose to formulate this corollary in a way which
emphasizes the traditional Dirichlet boundary condition in the higher-order 
setting (i.e., using iterated normal derivatives). Compared with 
Theorem~\ref{yaUHVav7f39}, this requires upgrading the underlying geometrical 
assumptions in order to make this type of boundary condition meaningful. 
Specifically, the following is a particular case of Theorem~\ref{yTganDJ} from 
the body of the paper.

\begin{theorem}\label{yTganDJ.IOa}
Let $\Omega$ be a bounded $(\varepsilon,\delta)$-domain in ${\mathbb{R}}^n$ 
whose boundary is $(n-1)$-Ahlfors regular. In addition, assume that
${\mathcal{H}}^{n-1}(\partial\Omega\setminus\partial_\ast\Omega)=0$ and  
denote by $\nu$ the geometric measure theoretic outward unit normal to $\Omega$.  
Also, consider a strongly elliptic, 
divergence-form system ${\mathcal{L}}$ of order $2m$, with bounded measurable
coefficients in $\Omega$. Then there exists $p_\ast\in(2,\infty)$, depending only 
on $\Omega$ as well as the bounds on the coefficients and the ellipticity constant 
of ${\mathcal{L}}$, with the property that the classical inhomogeneous Dirichlet 
boundary value problem 
\begin{eqnarray}\label{Tgb-5Uhgba}
\left\{
\begin{array}{l}
{\mathcal{L}}u=f\in W^{-m,p}(\Omega),
\\[6pt]
u\in W^{m,p}(\Omega),
\\[6pt]
\displaystyle
\frac{\partial^j u}{\partial\nu^j}=0
\,\mbox{ on }\,\partial\Omega\,\mbox{ for }\,0\leq j\leq m-1,
\end{array}
\right.
\end{eqnarray}
is well-posed whenever $\frac{p_\ast}{p_\ast-1}<p<p_\ast$.
\end{theorem}

By means of counterexamples (based on classical constructions due 
to N.~Meyers \cite{Mey}, E.~De Giorgi \cite{DG}, and V.~Maz'ya \cite{Maz-Ctr}), 
in the last part of \S\,\ref{Sect:7} we prove that the restriction of $p$ to a 
small interval near $2$ is actually necessary for the well-posedness of the 
inhomogeneous Dirichlet boundary value problem \eqref{Tgb-5Uhgba}. 
A key ingredient in the proof of Theorems~\ref{yaUHVav7f39}-\ref{yTganDJ.IOa}
is an interpolation result (cf. Theorem~\ref{HanmNB8} for a more complete statement) 
to the effect that, for each fixed $k\in{\mathbb{N}}$,
\begin{eqnarray}\label{GYknnbn}
\begin{array}{c}
\mbox{the scale }\,\,
\big\{W^{k,p}_D(\Omega)\big\}_{\max\{1,n-d\}<p<\infty}\,\,\mbox{ is stable} 
\\[6pt]
\mbox{both under the complex and the real method},
\end{array}
\end{eqnarray}
whenever $\Omega\subseteq{\mathbb{R}}^n$ and $D\subseteq\overline{\Omega}$ are 
such that $D$ is closed and $d$-Ahlfors regular for some $d\in(0,n)$, and $\Omega$ 
is locally an $(\varepsilon,\delta)$-domain near $\partial\Omega\setminus D$.

\vskip 0.10in

The layout of the paper is as follows. In \S\,\ref{Sect:2} we collect background
definitions and results, clarify terminology, and record a detailed statement of 
Jones' extension result, expanding on the succinct presentation from 
Theorem~\ref{TGKV885}. In turn, this is used in Theorem~\ref{cj-TRf}
to establish the fact that Jones' extension operator $\Lambda_k$ does not 
enlarge the support of a function $u\in W^{k,p}(\Omega)$ in $\overline{\Omega}$ (assuming, of course, that $\Omega$ is an $(\varepsilon,\delta)$-domain). 
This aspect plays a basic role in the proof of Theorem~\ref{TGKV885.2} 
(reformulated more generally as Theorem~\ref{cjeg.5AD}) in \S\,\ref{Sect:3}. 
In Theorem~\ref{cjeg.UUV} we record M.\,Christ's version of Jones' extension
theorem, specialized to Sobolev spaces, and subsequently note that Christ's 
semi-universal extension operator also enjoys the non-expansive support property 
alluded to above. In particular, this allows us to construct a semi-universal 
extension operator for Sobolev spaces with partially vanishing traces in  
Theorem~\ref{cjeg.5AD.UUV}.

The task of elucidating the structure of Sobolev spaces with partially 
vanishing traces from \eqref{TD-TRY753} is taken up in 
\S\,\ref{Sect:4}. Here, characterizations involving the vanishing of traces 
on $D$ is proved in Theorem~\ref{YTab-YHb} in a capacitary quasieverywhere sense, 
and in Theorem~\ref{YTah-YYHa9} where such a vanishing condition is formulated 
using the Hausdorff measure in place of Bessel capacities. The key ingredient 
in the proof of the latter result is the intrinsic characterization of the 
null-space of the the higher-order boundary trace operator (in the sense considered by 
A.~Jonsson and H.~Wallin in \cite{JoWa84}) in Theorem~\ref{YTah-YYHa8}
when this trace operator is acting from Sobolev spaces defined in ${\mathbb{R}}^n$.
Subsequently, in Theorem~\ref{Ohav-7UJ.2} we refine the Jonsson-Wallin theory 
from Theorem~\ref{GCC-67} in a manner which allows considering extension/restriction
operators preserving certain types of vanishing conditions. 

The main result in \S\,\ref{Sect:5} is Theorem~\ref{NIceTRace}, containing 
a trace/extension theory on locally $(\varepsilon,\delta)$-domains 
onto/from Ahlfors regular subsets in ${\mathbb{R}}^n$. One of the basic consequences
of this theory is the intrinsic description of Sobolev spaces with partially 
vanishing traces from Theorem~\ref{Kance.745}. We then proceed to deduce 
several important properties of this scale of spaces, including the hereditary 
property from Theorem~\ref{Pcsd5}, the issue of preservation of Sobolev class
under extension by zero in Theorem~\ref{Tfgg-75dS}, and under gluing Sobolev 
functions with matching traces in Theorem~\ref{SSSg-g5dS}. The last result 
in this section is Theorem~\ref{NIcBBB}, which establishes the characterization 
of the null-space of the higher-order Dirichlet trace operator from \eqref{Necas.YYY}. 

The main goal in \S\,\ref{Sect:6} is proving interpolation results in the spirit of 
\eqref{GYknnbn}. See Theorem~\ref{HanmNB8} in this regard. Finally, \S\,\ref{Sect:6}
is devoted to applications to boundary value problems in a very general 
geometric measure theoretic setting. More specifically, Theorem~\ref{yaUNDJ}
deals with the higher-order mixed boundary value problem 
in locally $(\varepsilon,\delta)$-domains, Theorem~\ref{yTganDJ} treats 
the higher-order inhomogeneous Dirichlet problem in arbitrary bounded open
sets with $d$-Ahlfors regular boundaries, Theorem~\ref{yTganDJ.Yam}
addresses the fully inhomogeneous higher-order Poisson problem in 
bounded $(\varepsilon,\delta)$-domains with $d$-Ahlfors regular boundaries, 
while the higher-order Neumann problem is considered in Theorem~\ref{yaUNDJ.taF} 
in the context of bounded $(\varepsilon,\delta)$-domains.

\section{Sobolev spaces and Jones' extension operator}
\label{Sect:2}
\setcounter{equation}{0}

We begin by discussing some background definitions and results. 
Fix a space dimension $n\in{\mathbb{N}}$, $n\geq 2$, and denote by ${\mathscr{L}}^n$ 
the $n$-dimensional Lebesgue measure in ${\mathbb{R}}^n$. Given a Lebesgue
measurable set ${\mathcal{O}}$ in ${\mathbb{R}}^n$, we 
let $L^p({\mathcal{O}},{\mathscr{L}}^n)$, $0<p\leq\infty$, stand for the 
scale of (equivalent classes of) Lebesgue-measurable functions which are 
$p$-th power ${\mathscr{L}}^n$-integrable in ${\mathcal{O}}$. Also, given 
an open set $\Omega\subseteq{\mathbb{R}}^n$, for each $p\in(0,\infty]$ denote by 
$L^p_{loc}(\Omega,{\mathscr{L}}^n)$ the space of Lebesgue-measurable functions 
$u$ in $\Omega$ with the property that $u\big|_{K}\in L^p(K,{\mathscr{L}}^n)$
for every compact subset $K$ of $\Omega$. 

With ${\mathbb{N}}$ denoting the collection of all 
(strictly) positive integers, we shall abbreviate ${\mathbb{N}}_0:={\mathbb{N}}\cup\{0\}$.
In particular, ${\mathbb{N}}_0^n$ may be regarded as the set of all 
multi-indices $\{\alpha=(\alpha_1,...,\alpha_n):\,
\alpha_i\in{\mathbb{N}}_0,\,\,1\leq i\leq n\}$. As usual, for each 
multi-index $\alpha=(\alpha_1,...,\alpha_n)\in{\mathbb{N}}_0^n$ 
we denote by $|\alpha|:=\alpha_1+\cdots+\alpha_n$ its length, and 
define $\alpha!:=\alpha_1!\cdots\alpha_n!$ (with the usual convention that $0!:=1$).
Also, write $\partial^\alpha:=\partial_{x_1}^{\alpha_1}\cdots\partial_{x_n}^{\alpha_n}$
and, given $\alpha=(\alpha_1,...,\alpha_n),
\beta=(\beta_1,...,\beta_n)\in{\mathbb{N}}_0^n$, by $\beta\leq\alpha$ 
it is understood that $\beta_j\leq\alpha_j$ for each $j\in\{1,...,n\}$.
For an arbitrary set $E\subseteq{\mathbb{R}}^n$ we shall denote by 
$E^\circ$, $\overline{E}$, ${\rm diam}\,E$, and $E^c$, respectively the interior, 
closure, diameter, distance to and complement of $E$ in ${\mathbb{R}}^n$.
In addition, ${\rm dist}\,(F,E)$, denotes the distance from $F$ to $E$.  
As usual, $B(x,r):=\{y\in{\mathbb{R}}^n:\,|x-y|<r\}$ for each 
$x\in{\mathbb{R}}$ and $r\in(0,\infty]$, where $|\cdot|$ denotes the standard 
Euclidean norm. Next, given a measurable space 
$\big({\mathscr{X}},{\mathfrak{M}},\mu\big)$, for any subset $E$ of 
the ambient ${\mathscr{X}}$ which belongs to the sigma-algebra ${\mathfrak{M}}$ 
we denote by $\mu\lfloor E$ the restriction of the measure $\mu$ to $E$.
Throughout, $\#E$ and ${\bf 1}_E$ stand, respectively, for the 
cardinality and the characteristic function of a given set $E$.
Given an open subset ${\mathcal{O}}$ of ${\mathbb{R}}^n$
and $k\in{\mathbb{N}}_0\cup\{\infty\}$, we shall denote 
by ${\mathscr{C}}^k({\mathcal{O}})$ the collection of all $k$-times 
continuously differentiable functions in ${\mathcal{O}}$, 
and by ${\mathscr{C}}^\infty_c({\mathcal{O}})$ the collection of all indefinitely 
differentiable functions which are compactly supported in ${\mathcal{O}}$.
Finally, define ${\mathscr{C}}^\infty(\overline{\mathcal{O}})
:=\big\{\varphi\big|_{\mathcal{O}}:\,\varphi
\in{\mathscr{C}}^\infty({\mathbb{R}}^n)\big\}$.

Assume next that $\Omega\subseteq{\mathbb{R}}^n$ is an arbitrary nonempty
open set. Then for each integer $k\in{\mathbb{N}}_0$ and integrability exponent 
$p\in[1,\infty]$, the $L^p$-based Sobolev space of order $k$ in $\Omega$ is defined 
intrinsically by 
\begin{eqnarray}\label{PJE-2}
W^{k,p}(\Omega):=\big\{u\in L^1_{loc}(\Omega,{\mathscr{L}}^n):\,
\partial^\alpha u\in L^p(\Omega,{\mathscr{L}}^n) 
\,\mbox{ for each }\,\alpha\in{\mathbb{N}}_0^n\,\,\,\mbox{ with }\,\,|\alpha|\leq k\big\},
\end{eqnarray}
where all derivatives are taken in the sense of distributions in the open set $\Omega$. 
As is well-known (cf., e.g., \cite[Theorem~3.3, p.\,60]{AdFo}),
$W^{k,p}(\Omega)$ becomes a Banach space when equipped with the natural norm 
\begin{eqnarray}\label{PJE-3}
\|u\|_{W^{k,p}(\Omega)}:=\sum_{|\alpha|\leq k}
\|\partial^\alpha u\|_{L^p(\Omega,{\mathscr{L}}^n)} 
\,\,\,\mbox{ for each }\,\,u\in W^{k,p}(\Omega).
\end{eqnarray}
As is well-known, $W^{k,p}(\Omega)$ is separable if $1\leq p<\infty$ and 
reflexive if $1<p<\infty$.
Since taking distributional derivatives commutes with the operation of 
restricting distributions to open subsets of their domain, and does not 
increase the Lebesgue norm, it is clear that the restriction operator 
\begin{eqnarray}\label{PJE-3UYGav}
W^{k,p}(\Omega)\ni u\longmapsto u\big|_{{\mathcal{O}}}\in W^{k,p}({\mathcal{O}})
\end{eqnarray}
is well-defined, linear and bounded, whenever $\Omega$ is an open subset 
of ${\mathbb{R}}^n$, ${\mathcal{O}}$ is a nonempty subset of $\Omega$, 
$k\in{\mathbb{N}}$, and $p\in[1,\infty]$.

For future purposes let us also define
\begin{eqnarray}\label{azTagbM}
\mathring{W}^{k,p}(\Omega)
:=\,\mbox{the closure of }\,{\mathscr{C}}^\infty_c(\Omega)\mbox{ in }
\Big(W^{k,p}(\Omega)\,,\,\|\cdot\|_{W^{k,p}(\Omega)}\Big)
\end{eqnarray}
and, assuming that $p,p'\in(1,\infty)$ are such that $1/p+1/p'=1$, set 
\begin{eqnarray}\label{azTagbM-W}
W^{-k,p}(\Omega):=\big(\mathring{W}^{k,p'}(\Omega)\bigr)^*.
\end{eqnarray}
As is well-known, for each $p\in(1,\infty)$, an alternative description of the
the above Sobolev space of negative order is
\begin{eqnarray}\label{azTagbM-W2}
W^{-k,p}(\Omega)=\Big\{u\in{\mathcal{D}}'(\Omega):\,
u=\sum_{|\alpha|\leq k}\partial^\alpha v_\alpha,\,\,\mbox{ where each }
v_\alpha\in L^p(\Omega,{\mathscr{L}}^n)\Big\},
\end{eqnarray}
where ${\mathcal{D}}'(\Omega)$ is the space of distributions in $\Omega$.
For this, as well as other related matters, the interested reader is referred
to, e.g., \cite[pp.\,62-65]{AdFo}. Here we only wish to note that
if $k\in{\mathbb{N}}$, $p\in[1,\infty]$ and $\Omega$ is an arbitrary
nonempty open subset of ${\mathbb{R}}^n$ then  
\begin{eqnarray}\label{YJBKL-uvv.5}
\mathring{W}^{k,p}(\Omega)\ni u\longmapsto\widetilde{u}
\in W^{k,p}({\mathbb{R}}^n)\,\,\,\mbox{ isometrically},
\end{eqnarray}
where tilde denotes the extension by zero outside $\Omega$, 
to the entire ${\mathbb{R}}^n$. Indeed, this rests on the observation 
that $\widetilde{\partial^\alpha u}=\partial^\alpha\widetilde{u}$ 
for every $u\in\mathring{W}^{k,p}(\Omega)$ and any $\alpha\in{\mathbb{N}}_0^n$
with $|\alpha|\leq k-1$ (which, in turn, is readily seen from \eqref{azTagbM}
and a simple limiting argument).

Moving on to the topic of extensions of functions from Sobolev spaces, we recall 
that the idea underpinning the construction of Jones' extension operator from 
\cite[Theorem~1, p.\,73]{Jon81} is to glue together, via a scale-sensitive partition 
of unity associated with a Whitney decomposition of the interior of the complement 
of the domain, certain polynomials which best fit the given function on the 
corresponding reflected cube across the boundary. While all this is made precise
in Theorem~\ref{cjeg} below, for the time being we briefly digress in order to 
clarify terminology and review some standard results.

According to Whitney's decomposition lemma (cf. \cite[Theorem~1, p.\,167]{Ste70}), 
one can associate to any open, nonempty, proper subset ${\mathcal{O}}$ 
of ${\mathbb{R}}^n$ a family ${\mathcal{W}}({\mathcal{O}})=\{Q_j\}_{j\in{\mathbb{N}}}$
of countably many closed dyadic cubes from ${\mathbb{R}}^n$ such that
\begin{eqnarray}\label{Patah-8n.A}
&& {\mathcal{O}}=\bigcup_{j\in{\mathbb{N}}}Q_j,
\\[6pt]
&& \sqrt{n}\ell(Q_j)\leq{\rm dist}\,(Q_j,\partial{\mathcal{O}})
\leq 4\sqrt{n}\,\ell(Q_j),\quad\mbox{for all }\,\,j\in{\mathbb{N}},
\label{Patah-8n.B}
\\[6pt]
&& Q_j^\circ\cap Q_k^\circ=\emptyset,\quad\mbox{for all }\,\,j,k\in{\mathbb{N}}
\,\,\,\mbox{ with }\,\,j\not=k,
\label{Patah-8n.C}
\\[6pt]
&& \tfrac{1}{4}\,\ell(Q_j)\leq\ell(Q_k)
\leq 4\ell(Q_j),\quad\mbox{for all }\,\,j,k\in{\mathbb{N}}
\,\,\,\mbox{ with }\,\,Q_j\cap Q_k\not=\emptyset.
\label{Patah-8n.D}
\end{eqnarray}
Above, $\ell(Q)$ denotes the side-length of the cube $Q$, and
$Q^\circ$ stands for the interior of $Q$. Also, given a 
positive number $\lambda$ and a cube $Q$, we denote by $\lambda Q$ the 
cube with the same center $x_Q$ as $Q$, and side-length $\lambda\ell(Q)$. 
With this convention, it is then straightforward to check that \eqref{Patah-8n.B}
implies
\begin{eqnarray}\label{Patah-8n.E}
\mbox{if $\lambda\in(0,3)$ then }\,\,
\lambda Q_j\subseteq{\mathcal{O}}\,\,\,\mbox{ for all }\,\,j\in{\mathbb{N}}.
\end{eqnarray}
In fact, for each $\lambda\in(0,3)$ there exists $c_\lambda\in(0,1)$ such that
\begin{eqnarray}\label{Patah-8n.F}
c_\lambda\leq\frac{{\rm dist}\,\big(\lambda Q_j,\partial{\mathcal{O}}\big)}{\ell(Q_j)}
\leq c_\lambda^{-1},\,\,\,\mbox{ for all }\,\,j\in{\mathbb{N}}.
\end{eqnarray}
Finally, given an arbitrary nonempty open set $\Omega\subseteq{\mathbb{R}}^n$, 
define 
\begin{eqnarray}\label{YUagv-75v}
\begin{array}{c}
{\rm rad}\,(\Omega):=\inf\limits_{m}\inf\limits_{x\in\Omega_m}
\sup\limits_{y\in\Omega_m}|x-y|,
\\[10pt]
\mbox{ where $\{\Omega_m\}_m$ are the connected components of }\,\,\Omega.
\end{array}
\end{eqnarray}
Unraveling definitions yields the following geometric characterization 
in the class of connected open sets in ${\mathbb{R}}^n$:
\begin{eqnarray}\label{YUagv-75v.4rf}
\begin{array}{c}
{\rm rad}\,(\Omega)=\inf\,\big\{r>0:\,\exists\,x\in\Omega\,\,\mbox{ such that }\,\,
\Omega\subseteq B(x,r)\bigr\},
\\[6pt]
\mbox{for any connected open set $\Omega\subseteq{\mathbb{R}}^n$},
\end{array}
\end{eqnarray}
i.e., ${\rm rad}\,(\Omega)$ is, loosely speaking, the smallest of all 
radii of balls centered in $\Omega$ which envelop this connected set. In particular, 
it is clear that ${\rm rad}\,(\Omega)>0$ for every finitely connected open 
set $\Omega$ in ${\mathbb{R}}^n$. 

Here is the result advertised earlier. It is a detailed statement of 
\cite[Theorem~1, p.\,73]{Jon81} (which should also be compared with the 
statement of \cite[Theorem~B, p.\,1029]{Chua}).

\begin{theorem}[P.W.~Jones's Extension Theorem]\label{cjeg}
Let $\Omega$ be an $(\varepsilon,\delta)$-domain in ${\mathbb{R}}^n$ with 
${\rm rad}\,(\Omega)>0$, and fix $k\in{\mathbb{N}}$. Also, pick a Whitney 
decomposition ${\mathcal{W}}(\Omega)$ of $\Omega$ along with a Whitney decomposition
${\mathcal{W}}\big((\Omega^c)^\circ\big)$ of $(\Omega^c)^\circ$, and 
consider the collection of all small cube in the latter, i.e., define 
\begin{eqnarray}\label{PJE-sd1}
{\mathcal{W}}_s\big((\Omega^c)^\circ\big):=
\big\{Q\in{\mathcal{W}}\big((\Omega^c)^\circ\big):\,\ell(Q)\leq\varepsilon\delta/(16 n)
\big\}.
\end{eqnarray}
For any function $u\in L^1_{loc}(\Omega,{\mathscr{L}}^n)$ and any dyadic 
cube $Q\in{\mathcal{W}}(\Omega)$ let $P_{Q}(u)$ denote the unique polynomial 
of degree $k-1$ which best fits $u$ on $Q$ in the sense that 
\begin{eqnarray}\label{PJE-7jmB}
\int_{Q}\partial^{\alpha}\big(u-P_{Q}(u)\big)\,d{\mathcal{L}}^n=0
\,\,\mbox{ for each }\,\,\alpha\in{\mathbb{N}}_0^n\,\,\mbox{ with }\,\,|\alpha|\leq k-1.
\end{eqnarray}
To each
$Q\in{\mathcal{W}}_s\big((\Omega^c)^\circ\big)$ 
assign a cube $Q^*\in{\mathcal{W}}(\Omega)$ satisfying 
(cf. \cite[Lemma~2.4, p.\,77]{Jon81})
\begin{eqnarray}\label{PJE-b1}
\ell(Q)\leq\ell(Q^*)\leq 4\ell(Q)\,\,\mbox{ and }\,\,
{\rm dist}\,(Q,Q^*)\leq C_{n,\varepsilon}\,\ell(Q),\qquad
\forall\,Q\in{\mathcal{W}}_s\big((\Omega^c)^\circ\big),
\end{eqnarray}

Finally, to each $u\in L^1_{loc}(\Omega,{\mathscr{L}}^n)$ associate the 
function $\Lambda_k u$ defined ${\mathscr{L}}^n$-a.e. in ${\mathbb{R}}^n$ by 
\begin{eqnarray}\label{PJE-a1}
\Lambda_k u:=\left\{
\begin{array}{ll}
u & \mbox{ in }\,\,\Omega,
\\[8pt]
\displaystyle
\sum\limits_{Q\in{\mathcal{W}}_s
\big((\Omega^c)^\circ\big)}P_{Q^*}(u)\,\varphi_Q
& \mbox{ in }\,\,\big(\Omega^c\big)^{\circ},
\end{array}
\right.
\end{eqnarray}
where the family $\big\{\varphi_Q\big\}_{Q\in{\mathcal{W}}_s
\big((\Omega^c)^\circ\big)}$ consists of functions satisfying 
\begin{eqnarray}\label{PJE-H76}
\varphi_Q\in{\mathscr{C}}^\infty_c({\mathbb{R}}^n),\quad
{\rm supp}\,\varphi_Q\subseteq\tfrac{17}{16}Q,\quad
0\leq\varphi_Q\leq 1,\quad
\big|\partial^\alpha\varphi_Q\big|\leq C_\alpha\ell(Q)^{-|\alpha|},\quad
\forall\,\alpha\in{\mathbb{N}}_0^n,
\end{eqnarray}
for every $Q\in{\mathcal{W}}_s
\big((\Omega^c)^\circ\big)$, as well as
\begin{eqnarray}\label{PJE-H77}
\sum_{Q\in{\mathcal{W}}_s
\big((\Omega^c)^\circ\big)}\varphi_Q\equiv 1\qquad\mbox{ on }\,\,\,
\bigcup_{Q\in{\mathcal{W}}_s
\big((\Omega^c)^\circ\big)}Q.
\end{eqnarray}

Then for every $p\in[1,\infty]$ the operator $\Lambda_k$ satisfies
\begin{eqnarray}\label{fgu4J.P}
&& \Lambda_k:W^{k,p}(\Omega)\longrightarrow W^{k,p}({\mathbb{R}}^n)
\quad\mbox{ linearly and boundedly},
\\[4pt]
&& \Lambda_k u\big|_{\Omega}=u,\quad
\mbox{${\mathscr{L}}^n$-a.e. on $\Omega$ for every $u\in W^{k,p}(\Omega)$},
\label{Pa-PL2}
\end{eqnarray}
and with operator norm which may be controlled 
solely in terms of $\varepsilon,\delta,n,p,k,{\rm rad}\,(\Omega)$ in the 
following fashion:
\begin{eqnarray}\label{fgu4J.P75E}
\begin{array}{c}
\forall\,p\in[1,\infty],\,\,\forall\,k\in{\mathbb{N}},\,\,\forall\,M\in(0,\infty),
\,\,\exists\,C(n,k,p,M)\in(0,\infty)\,\,\mbox{ such that}
\\[8pt]
\big\|\Lambda_k\big\|_{W^{k,p}(\Omega)\to W^{k,p}({\mathbb{R}}^n)}\leq C(n,k,p,M)
\,\,\mbox{ whenever}
\\[10pt]
\mbox{$\Omega$ is an $(\varepsilon,\delta)$-domain in ${\mathbb{R}}^n$ 
such that }\,\,\max\,\big\{\varepsilon^{-1},\delta^{-1},{\rm rad}\,(\Omega)^{-1}\big\}
\leq M.
\end{array}
\end{eqnarray}
\end{theorem}

We continue by describing an alteration of Jones' extension operator, following 
the work of M.\,Christ in \cite{Christ}. In preparation, we first record the 
statement of \cite[Proposition~2.5, p.\,67]{Christ} (cf. also \cite{BIN}). 

\begin{lemma}\label{TYD-ytfc}
Fix a number $\Upsilon\in{\mathbb{N}}$ and let $Q_o:=\big(-\frac12,\frac12\big)^n$ 
be the standard unit cube centered at the origin in ${\mathbb{R}}^n$. Then there 
exists a linear projection operator assigning 
\begin{eqnarray}\label{iagarEDV}
L^1\big(Q_o,{\mathscr{L}}^n\big)\ni u\longmapsto\widehat{P}_o(u)\in\big
\{P\big|_{Q_o}:\,P\,\,\mbox{ polynomial of degree $\leq\Upsilon$ in }\,\,
{\mathbb{R}}^n\big\}
\end{eqnarray}
with the property that, given any $M\in{\mathbb{N}}$ such that $M\leq\Upsilon$, 
then for any $p\in[1,\infty]$ and any $r\in(0,1]$,
\begin{eqnarray}\label{iagarEDV.2}
\sum_{|\alpha|\leq M-1}\big\|\partial^\alpha\big(u-\widehat{P}_o(u)\big)
\big\|_{L^p(rQ_o,\,{\mathscr{L}}^n)}\leq C(\Upsilon,r)\sum_{|\beta|=M}
\|\partial^\beta u\|_{L^p(rQ_o,\,{\mathscr{L}}^n)},
\end{eqnarray}
for every $u\in W^{M,p}(Q_o)$, and such that
\begin{eqnarray}\label{iagarEDV.3}
\big\|\partial^\alpha\widehat{P}_o(u)\big\|_{L^p(rQ_o,\,{\mathscr{L}}^n)}
\leq C(\Upsilon,r)\sum_{|\beta|=|\alpha|}
\|\partial^\beta u\|_{L^p(rQ_o,{\mathscr{L}}^n)},\quad\forall\,\alpha
\in{\mathbb{N}}_0^n\,\,\mbox{ with }\,\,|\alpha|\leq\Upsilon-1,
\end{eqnarray}
for every $u\in W^{\Upsilon-1,p}(Q_o)$.
\end{lemma}

\noindent In general, given any cube $Q$ in ${\mathbb{R}}^n$, define the operator 
$\widehat{P}_Q$ by translating and dilating $Q$ as to coincide with the unit 
cube $Q_o$, apply $\widehat{P}_o$ from Lemma~\ref{TYD-ytfc}, and then reverse 
the dilation and translation. 

Theorem~\ref{cjeg.UUV} below is a restatement of 
\cite[Theorem~1.1, p.\,64]{Christ} (complemented with 
\cite[Remark, p.\,66]{Christ}) specialized to the case when the smoothness
index is a positive integer, in which scenario the DeVore-Sharpley spaces
(in terms of which Christ's theorem is originally stated) coincide with 
Sobolev spaces (see \cite[Theorem~6.2, p.\,37]{DS1} in this regard). 
Compared to $\Lambda_k$ from Theorem~\ref{cjeg}, the novel feature of the operator 
$\widehat{\Lambda}$ constructed here is the fact that the latter {\it simultaneously} 
extends functions from $W^{k,p}(\Omega)$ to ${\mathbb{R}}^n$, with preservation of 
class, for all $k$'s up to an a priori upper bound $\Upsilon$. It is therefore
natural to refer to such an extension operator as being {\tt semi-universal}. 
There is a price to pay, however, namely the exclusion of the case when $p=1$. 

\begin{theorem}[M.\,Christ's version of Jones' Extension Theorem]\label{cjeg.UUV}
Let $\Omega$ be an $(\varepsilon,\delta)$-domain in ${\mathbb{R}}^n$ with 
${\rm rad}\,(\Omega)>0$, and fix $\Upsilon\in{\mathbb{N}}$. 
In this context, retain the notation introduced in the statement of 
Theorem~\ref{cjeg} and, to each $u\in L^1_{loc}(\Omega,{\mathscr{L}}^n)$ 
associate the function $\widehat{\Lambda}\,u$ defined ${\mathscr{L}}^n$-a.e. 
in ${\mathbb{R}}^n$ by 
\begin{eqnarray}\label{PJE-a1.UUV}
\widehat{\Lambda}\,u:=\left\{
\begin{array}{ll}
u & \mbox{ in }\,\,\Omega,
\\[8pt]
\displaystyle
\sum\limits_{Q\in{\mathcal{W}}_s
\big((\Omega^c)^\circ\big)}\widehat{P}_{Q^*}(u)\,\varphi_Q
& \mbox{ in }\,\,\big(\Omega^c\big)^{\circ},
\end{array}
\right.
\end{eqnarray}
where $\widehat{P}_{Q^*}$ is the polynomial projection associated with 
the reflected cube $Q^\ast$ (cf. \eqref{PJE-b1}) as in Lemma~\ref{TYD-ytfc}
and the subsequent comment. 

Then for every $p\in(1,\infty]$ and every $k\in{\mathbb{N}}$ 
such that $k<\Upsilon$, the operator $\widehat{\Lambda}$ satisfies
\begin{eqnarray}\label{fgu4J.P.UUV}
&& \widehat{\Lambda}:W^{k,p}(\Omega)\longrightarrow W^{k,p}({\mathbb{R}}^n)
\quad\mbox{ linearly and boundedly},
\\[4pt]
&& \widehat{\Lambda}\,u\big|_{\Omega}=u,\quad
\mbox{${\mathscr{L}}^n$-a.e. on $\Omega$ for every $u\in W^{k,p}(\Omega)$},
\label{Pa-PL2.UUV}
\end{eqnarray}
and with operator norm which may be controlled 
solely in terms of $\varepsilon,\delta,n,p,{\rm rad}\,(\Omega)$ and 
$\Upsilon$ in a similar manner to \eqref{fgu4J.P75E}.
\end{theorem}

Our first result in this paper brings to the forefront a salient feature of 
Jones' extension operator $\Lambda_k$ from Theorem~\ref{cjeg}, namely the 
property that for any function $u\in W^{k,p}(\Omega)$ the support of 
its extension $\Lambda_k u\in W^{k,p}({\mathbb{R}}^n)$ does not touch 
$\partial\Omega$ outside the region where the support of $u$ itself makes contact 
with $\partial\Omega$. As we shall see momentarily, the same type of property is 
enjoyed by Christ's version of Jones' operator. In order to make this precise, 
we need to introduce some notation. 

Generally speaking, given an open set ${\mathcal{O}}\subseteq{\mathbb{R}}^n$ and 
an ${\mathscr{L}}^n$-measurable function $v$ on ${\mathcal{O}}$, define 
\begin{eqnarray}\label{PJE-1}
{\rm supp}\,v:=\big\{x\in\overline{\mathcal{O}}:
\,\mbox{there is no $r>0$ such that }\,\,v\equiv 0\,\,
\mbox{ ${\mathscr{L}}^n$-a.e. in }\,B(x,r)\cap{\mathcal{O}}\big\}.
\end{eqnarray}
Note that while the function $v$ is known to be defined only in ${\mathcal{O}}$, 
the set ${\rm supp}\,v$ (itself a closed subset of ${\mathbb{R}}^n$) 
is contained in $\overline{\mathcal{O}}$. 
It is also clear from the above definition that if 
${\mathcal{O}}\subseteq{\mathbb{R}}^n$ is an open set and $v$ is 
an ${\mathscr{L}}^n$-measurable function defined on ${\mathcal{O}}$, then 
\begin{eqnarray}\label{PJE-1PP}
{\rm supp}\,\big(v\big|_{U}\big)\subseteq\overline{U}\cap{\rm supp}\,v,\qquad
\mbox{for any open subset $U$ of ${\mathcal{O}}$}.
\end{eqnarray}
Moreover, since every open cover of ${\mathcal{O}}\setminus{\rm supp}\,v$ 
has a countable subcover (given that the open set in question is $\sigma$-compact),
it follows that 
\begin{eqnarray}\label{PJE-UnEE}
\mbox{$v$ vanishes ${\mathscr{L}}^n$-a.e. on ${\mathcal{O}}\setminus{\rm supp}\,v$}.
\end{eqnarray}

Here is the precise formulation of the result announced earlier.

\begin{theorem}[Preservation of support in the closure of domain]\label{cj-TRf}
Let $\Omega$ be an $(\varepsilon,\delta)$-domain in ${\mathbb{R}}^n$
and fix some $k\in{\mathbb{N}}$. Then the Jones' extension operator $\Lambda_k$
from Theorem~\ref{cjeg} has the property that, given any $p\in[1,\infty]$, one has
\begin{eqnarray}\label{fgu4J-D}
\overline{\Omega}\cap\,{\rm supp}\,(\Lambda_k u)
={\rm supp}\,u,\qquad\forall\,u\in W^{k,p}(\Omega).
\end{eqnarray}
Moreover, there exists a finite number $R=R(n,\varepsilon,\delta)>0$ with the
property that 
\begin{eqnarray}\label{fgu4J-Dyab.4RR}
{\rm supp}\,\big(\Lambda_k u\big)\subseteq
\big\{x\in{\mathbb{R}}^n:\,{\rm dist}\,(x,{\rm supp}\,u)\leq R\big\}
\quad\mbox{for every }\,\,u\in W^{k,p}(\Omega).
\end{eqnarray}
In particular, for every $u\in W^{k,p}(\Omega)$, one has
\begin{eqnarray}\label{fgu4J-DUBB}
{\rm supp}\,u\,\,\mbox{ compact}\,\Longrightarrow\,
{\rm supp}\,(\Lambda_k u)\,\,\mbox{ compact}.
\end{eqnarray}

Finally, for any given $\Upsilon\in{\mathbb{N}}$, the operator $\widehat{\Lambda}$ 
from Theorem~\ref{cjeg.UUV} satisfies similar properties, provided 
$k<\Upsilon$ and $1<p\leq\infty$. 
\end{theorem}

\begin{proof}
Fix $k\in{\mathbb{N}}$, $p\in[1,\infty]$, and select an arbitrary $u\in W^{k,p}(\Omega)$.
Then combining \eqref{Pa-PL2} with \eqref{PJE-1PP} (used here with 
$v:=\Lambda_k u$, ${\mathcal{O}}:={\mathbb{R}}^n$, and $U:=\Omega$) yields 
\begin{eqnarray}\label{PJE-1PQ}
{\rm supp}\,u={\rm supp}\,\big(\Lambda_k u\big|_{\Omega}\big)
\subseteq\overline{\Omega}\cap{\rm supp}\,\Lambda_k u,
\end{eqnarray}
proving the right-to-left inclusion in \eqref{fgu4J-D}. Proceeding in the opposite 
direction, we first claim that 
\begin{eqnarray}\label{PJE-23av}
{\rm supp}\,\Lambda_k u\subseteq{\rm supp}\,u\cup\overline{F_u},
\end{eqnarray}
where 
\begin{eqnarray}\label{PJE-5taGB}
F_u:=\bigcup_{\stackrel{Q\in{\mathcal{W}}_s
\big((\Omega^c)^\circ\big),}{Q^*\cap\,{\rm supp}\,u\not=\emptyset}}\tfrac{17}{16}Q.
\end{eqnarray}
To justify this claim, assume that 
$x_o\in{\mathbb{R}}^n\setminus({\rm supp}\,u\cup\overline{F_u})$. Then there exists 
$r>0$ with the property that 
\begin{eqnarray}\label{PJA-j}
\mbox{$u$ vanishes ${\mathscr{L}}^n$-a.e. in $B(x_o,r)\cap\Omega$},
\end{eqnarray}
and such that $B(x_o,r)\cap F_u=\emptyset$. The latter condition further entails
(in concert with \eqref{PJE-H76}) that, on the one hand,  
\begin{eqnarray}\label{PJA-i}
\sum\limits_{\stackrel{Q\in{\mathcal{W}}_s
\big((\Omega^c)^\circ\big),}{Q^*\cap\,{\rm supp}\,u\not=\emptyset}}
P_{Q^*}(u)\,\varphi_Q\equiv 0\quad\mbox{ in }\quad 
B(x_o,r)\cap\big(\Omega^c\big)^{\circ}.
\end{eqnarray}
On the other hand, we claim that we also have
\begin{eqnarray}\label{PJA-ii}
\sum\limits_{\stackrel{Q\in{\mathcal{W}}_s
\big((\Omega^c)^\circ\big),}{Q^*\cap\,{\rm supp}\,u=\emptyset}}
P_{Q^*}(u)\,\varphi_Q\equiv 0\quad\mbox{ in }\quad 
B(x_o,r)\cap\big(\Omega^c\big)^{\circ}.
\end{eqnarray}
Indeed, since for every $Q\in{\mathcal{W}}_s
\big((\Omega^c)^\circ\big)$ the best fit polynomial $P_{Q^\ast}(u)$ has 
degree $k-1$, condition \eqref{PJE-7jmB} entails 
\begin{eqnarray}\label{PJE-Yhan}
P_{Q^\ast}(u)(x)=\sum_{|\alpha|\leq k-1}\frac{x^\alpha}{\alpha!}\,\meanint_{\!Q^\ast}
\partial^\alpha u\,d{\mathscr{L}}^n,\qquad\forall\,x\in{\mathbb{R}}^n.
\end{eqnarray}
Here and elsewhere, $\dmeanint$ stands for integral average.
In particular, this and \eqref{PJE-UnEE} show that 
\begin{eqnarray}\label{PJE-YhP.6}
P_{Q^\ast}(u)\equiv 0\quad\mbox{ for every }\,\,
Q\in{\mathcal{W}}_s\big((\Omega^c)^\circ\big)
\,\,\,\mbox{ with }\,\,\,Q^*\cap\,{\rm supp}\,u=\emptyset,
\end{eqnarray}
and \eqref{PJA-ii} readily follows from \eqref{PJE-YhP.6}. Together, 
\eqref{PJA-i} and \eqref{PJA-ii} imply that 
\begin{eqnarray}\label{PJA-iii}
\sum\limits_{Q\in{\mathcal{W}}_s
\big((\Omega^c)^\circ\big)}P_{Q^*}(u)\,\varphi_Q\equiv 0
\quad\mbox{ in }\quad B(x_o,r)\cap\big(\Omega^c\big)^{\circ}.
\end{eqnarray}
From \eqref{PJE-a1}, \eqref{PJA-j}, \eqref{PJA-iii}, and \eqref{PJE-wcX}, 
we may then deduce that
\begin{eqnarray}\label{PJA-jj}
\mbox{$\Lambda_k u$ vanishes ${\mathscr{L}}^n$-a.e. in $B(x_o,r)$}.
\end{eqnarray}
Hence, $x_o\notin{\rm supp}\,\Lambda_k u$ which finishes the proof of \eqref{PJE-23av}.

Having established \eqref{PJE-23av}, we next claim that 
\begin{eqnarray}\label{PJE-23av.2}
\overline{\Omega}\cap\overline{F_u}\subseteq{\rm supp}\,u.
\end{eqnarray}
To justify this, reason by contradiction and assume that 
there exists a point $x_o\in\overline{\Omega}\cap\overline{F_u}$ such that 
$x_o\notin{\rm supp}\,u$. In particular, the latter condition entails the existence 
of some $r>0$ for which 
\begin{eqnarray}\label{PJA-k}
\mbox{$u$ vanishes ${\mathscr{L}}^n$-a.e. in $B(x_o,r)\cap\Omega$}.
\end{eqnarray}
Let us take a closer look at the fact that 
$x_o\in\overline{\Omega}\cap\overline{F_u}$. For starters, the fact that
$F_u\subseteq(\Omega^c)^\circ$ (as seen from \eqref{Patah-8n.E} and \eqref{PJE-5taGB})
forces
\begin{eqnarray}\label{PJE-ASV}
\overline{\Omega}\cap\overline{F_u}
\subseteq\overline{\Omega}\cap\overline{(\Omega^c)^\circ}
=\overline{\Omega}\setminus(\,\overline{\Omega}\,)^\circ\subseteq
\overline{\Omega}\setminus\Omega=\partial\Omega,
\end{eqnarray}
whereupon
\begin{eqnarray}\label{PJE-ASV.2}
x_o\in\partial\Omega.
\end{eqnarray}
Next, the membership of $x_o$ to the closure of the set $F_u$ defined in 
\eqref{PJE-5taGB} entails the existence of a sequence of dyadic cubes
$\{Q_j\}_j\subseteq{\mathcal{W}}_s
\big((\Omega^c)^\circ\big)$ and a sequence $\{x_j\}_j$ of points in ${\mathbb{R}}^n$ 
satisfying 
\begin{eqnarray}\label{PJE-E.a3}
&& x_j\in\tfrac{17}{16}Q_j\,\,\,\mbox{ for every }\,\,j,
\\[4pt]
&& Q_j^*\cap\,{\rm supp}\,u\not=\emptyset\,\,\,\mbox{ for every }\,\,j,
\label{PJE-E.a3.i}
\\[4pt]
&& \lim_{j}x_j=x_o.
\label{PJE-E.a3.ii}
\end{eqnarray}
Now, \eqref{PJE-E.a3} forces 
\begin{eqnarray}\label{PJE-E.a3LL}
|x_j-x_{Q_j}|\leq\sqrt{n}\tfrac{17}{16}\ell(Q_j)\,\,\,\mbox{ for every }\,\,j,
\end{eqnarray}
while from \eqref{Patah-8n.F} we conclude that there exists $c\in(0,1)$ such that
\begin{eqnarray}\label{Patah-8n.Fi}
c\,\ell(Q_j)\leq{\rm dist}\,\big(\tfrac{17}{16}Q_j,\partial{\Omega}\,\big)
\leq{\rm dist}\,\big(x_j,\partial{\Omega}\,\big)\leq |x_j-x_o|,
\,\,\,\mbox{ for all }\,\,j,
\end{eqnarray}
where the last inequality uses \eqref{PJE-ASV.2}.
From \eqref{PJE-b1} we also deduce that
\begin{eqnarray}\label{PJE-b1AAA}
\ell(Q_j)\leq\ell(Q_j^*)\leq 4\ell(Q_j)\,\,\mbox{ and }\,\,
|x_{Q_j}-x_{Q_j^*}|\leq C_{n,\varepsilon}\,\ell(Q_j),
\,\,\,\mbox{ for all }\,\,j.
\end{eqnarray}
Combining now \eqref{PJE-E.a3.ii}-\eqref{PJE-b1AAA} yields 
\begin{eqnarray}\label{PJbb}
\lim_{j}x_{Q_j^\ast}=x_o\,\,\,\mbox{ and }\,\,\,\lim_{j}\ell(Q_j^\ast)=0.
\end{eqnarray}
In turn, from \eqref{PJbb} we deduce that 
\begin{eqnarray}\label{PJbb.LLs3}
\mbox{there exists $j$ such that $Q_j^\ast\subseteq B(x_o,r)\cap\Omega$}
\end{eqnarray}
which, in light of \eqref{PJA-k}, implies that
\begin{eqnarray}\label{PJu-n1}
\mbox{there exists $j$ such that $u$ vanishes ${\mathscr{L}}^n$-a.e. in $Q_j^\ast$}.
\end{eqnarray}
This, however, contradicts \eqref{PJE-E.a3.i}. The proof of \eqref{PJE-23av.2} is 
therefore complete. 

With \eqref{PJE-23av.2} in hand, and availing ourselves of \eqref{PJE-23av} we may 
write 
\begin{eqnarray}\label{PJE-23av.KL}
\overline{\Omega}\cap{\rm supp}\,\Lambda_k u
\subseteq\big(\,\overline{\Omega}\cap{\rm supp}\,u\big)\cup
\big(\,\overline{\Omega}\cap\overline{F_u}\,\big)
\subseteq {\rm supp}\,u,
\end{eqnarray}
which proves the left-to-right inclusion in \eqref{fgu4J-D}. 

As far as \eqref{fgu4J-DUBB} is concerned, assume some $u\in W^{k,p}(\Omega)$
has been given. Consider an arbitrary point $x_1\in F_u$. Then there exists some 
cube $Q\in{\mathcal{W}}_s\big((\Omega^c)^\circ\big)$
such that $x_1\in\tfrac{17}{16}Q$ and $Q^\ast\cap{\rm supp}\,u\not=\emptyset$. 
Pick some $x_2\in Q^\ast\cap{\rm supp}\,u$ and note that, thanks to 
\eqref{PJE-sd1}-\eqref{PJE-b1}, we may estimate 
\begin{eqnarray}\label{UHi-Jba.6}
&& |x_1-x_Q|\leq 2\sqrt{n}\,\tfrac{17}{16}\ell(Q)\leq C(n,\varepsilon,\delta),
\\[4pt]
&& |x_Q-x_{Q^\ast}|\leq{\rm dist}\,(Q,Q^\ast)
\leq C_{n,\varepsilon}\ell(Q)\leq C(n,\varepsilon,\delta),
\label{UHi-Jba.7}
\\[4pt]
&& |x_{Q^\ast}-x_2|\leq 2\sqrt{n}\,\ell(Q^\ast)
\leq 8\sqrt{n}\,\ell(Q)\leq C(n,\varepsilon,\delta),
\label{UHi-Jba.8}
\end{eqnarray}
for some finite constant $C(n,\varepsilon,\delta)>0$. Collectively, 
\eqref{UHi-Jba.6}-\eqref{UHi-Jba.8} imply that 
\begin{eqnarray}\label{fgu4J-Dyab}
{\rm dist}\,(x_1,{\rm supp}\,u)\leq |x_1-x_2|\leq
|x_1-x_Q|+|x_Q-x_{Q^\ast}|+|x_{Q^\ast}-x_2|\leq R:=3\,C(n,\varepsilon,\delta).
\end{eqnarray}
Since $x_1$ has been arbitrarily chosen in $F_u$, this proves that 
\begin{eqnarray}\label{fgu4J-Dyab.2}
F_u\subseteq\big\{x\in{\mathbb{R}}^n:\,{\rm dist}\,(x,{\rm supp}\,u)\leq R\big\}.
\end{eqnarray}
At this point, \eqref{fgu4J-Dyab.4RR} readily follows from \eqref{fgu4J-Dyab.2} 
and \eqref{PJE-23av}. Finally, if ${\rm supp}\,u$ is compact, this 
shows that ${\rm supp}\,\Lambda_k u$ is a bounded set, hence also compact, 
as desired. 

Finally, to see that similar properties may be established for the operator 
$\widehat{\Lambda}$ from Theorem~\ref{cjeg.UUV} (assuming that 
$k<\Upsilon$ and $1<p\leq\infty$), it suffices to observe that \eqref{PJE-YhP.6} 
holds for $\widehat{P}_{Q^\ast}(u)$ in place of $P_{Q^\ast}(u)$; the rest of the 
proof is virtually identical to the one carried out for $\Lambda_k$.
\end{proof}

\section{Extension of Sobolev functions with partially vanishing traces 
defined in locally $(\varepsilon,\delta)$-domains}
\label{Sect:3}
\setcounter{equation}{0}

The class of domains alluded to in the title of this section is going to be 
introduced a little later, in Definition~\ref{OM-Oah-Y88}. To set the stage, 
we shall make the following definition which plays an important role throughout 
the proceedings. 

\begin{definition}\label{IUha-Tw}
Given a nonempty open set $\Omega$ in ${\mathbb{R}}^n$ and a closed subset 
$D$ of $\overline{\Omega}$, consider 
\begin{eqnarray}\label{uig-Uba}
{\mathscr{C}}^\infty_D(\Omega):=\big\{\widetilde{\psi}\big|_{\Omega}:\,
\psi\in{\mathscr{C}}^\infty_c({\mathbb{R}}^n\setminus D)\big\}
=\big\{\varphi\big|_{\Omega}:\,
\varphi\in{\mathscr{C}}^\infty_c({\mathbb{R}}^n)\,\,\mbox{ with }\,\,
D\cap{\rm supp}\,\varphi=\emptyset\big\}
\end{eqnarray}
where tilde denotes the extension by zero outside the support to ${\mathbb{R}}^n$, 
and for each $k\in{\mathbb{N}}$, $p\in[1,\infty]$, define 
\begin{eqnarray}\label{uig-adjb-2}
W^{k,p}_D(\Omega):=\mbox{the closure of ${\mathscr{C}}^\infty_D(\Omega)$ 
in $\Bigl(W^{k,p}(\Omega)\,,\,\|\cdot\|_{W^{k,p}(\Omega)}\Bigr)$}.
\end{eqnarray}
\end{definition}

\noindent In particular, given a closed subset $D$ of ${\mathbb{R}}^n$, 
for each $k\in{\mathbb{N}}$, $p\in[1,\infty]$, corresponding to the 
case when $\Omega={\mathbb{R}}^n$ in Definition~\ref{IUha-Tw} we have 
\begin{eqnarray}\label{uig-adjb-2AAA}
W^{k,p}_D({\mathbb{R}}^n):=\mbox{the closure of
$\{\varphi\in{\mathscr{C}}^\infty_c({\mathbb{R}}^n):\,D\cap{\rm supp}\,\varphi 
=\emptyset\}$ 
in $\Bigl(W^{k,p}({\mathbb{R}}^n)\,,\,\|\cdot\|_{W^{k,p}({\mathbb{R}}^n)}\Bigr)$}.
\end{eqnarray}
In particular, corresponding to the case when $D=\emptyset$ we have 
$W^{k,p}_\emptyset({\mathbb{R}}^n)=W^{k,p}({\mathbb{R}}^n)$.

A number of other useful elementary properties of the spaces 
considered above are collected in the next lemma (whose routine proof is omitted).

\begin{lemma}\label{TYRD-f4f}
Let $\Omega$ be a nonempty open subset of ${\mathbb{R}}^n$ and consider a closed 
subset $D$ of $\overline{\Omega}$. Also, $k\in{\mathbb{N}}$ and $p\in[1,\infty]$.
Then the following hold:
\begin{enumerate}
\item[(1)] ${\mathscr{C}}^\infty_D(\Omega)$ is a dense linear subspace 
of $W^{k,p}_D(\Omega)$;
\item[(2)] $W^{k,p}_D(\Omega)$ is a closed linear subspace of $W^{k,p}(\Omega)$;
\item[(3)] $\Bigl(W^{k,p}_D(\Omega)\,,\,\|\cdot\|_{W^{k,p}(\Omega)}\Bigr)$ 
is a Banach space, which is separable if $1\leq p<\infty$, and reflexive if $1<p<\infty$;
\item[(4)] the inclusion $\Big\{u\big|_{\Omega}:\,u\in W^{k,p}_D({\mathbb{R}}^n)\Big\}
\hookrightarrow W^{k,p}_D(\Omega)$ is well-defined and continuous, and the restriction
operator $W^{k,p}_D({\mathbb{R}}^n)\ni u\mapsto u\big|_{\Omega}\in W^{k,p}_D(\Omega)$
is well-defined, linear and bounded;
\item[(5)] corresponding to the case when $D=\partial\Omega$, there holds 
$W^{k,p}_{\partial\Omega}(\Omega)=\mathring{W}^{k,p}(\Omega)$;
\item[(6)] if $\partial\Omega\subseteq D$ then the inclusion
$W^{k,p}_D(\Omega)\hookrightarrow\mathring{W}^{k,p}(\Omega)$
is well-defined and isometric;
\item[(7)] ${\mathscr{C}}^\infty_{\partial\Omega}(\Omega)
={\mathscr{C}}^\infty_c(\Omega)$, and if $\partial\Omega\subseteq D$ 
then ${\mathscr{C}}^\infty_D(\Omega)\subseteq{\mathscr{C}}^\infty_c(\Omega)$;
\item[(8)] if $\Sigma$ is a closed subset of $\Omega$ which contains the given set $D$, 
then ${\mathscr{C}}^\infty_{\Sigma}(\Omega)\subseteq{\mathscr{C}}^\infty_D(\Omega)$, 
and the inclusion $W^{k,p}_{\Sigma}(\Omega)\hookrightarrow W^{k,p}_D(\Omega)$
is well-defined, linear and isometric;
\item[(9)] if $v\in W^{k,p}_D(\Omega)$ then 
$v\big|_{{\mathcal{O}}}\in W^{k,p}_{\overline{{\mathcal{O}}}\cap D}({\mathcal{O}})$
for any open subset ${\mathcal{O}}$ of $\Omega$.
\end{enumerate}
\end{lemma}

The special case of the scale \eqref{uig-adjb-2} with $D=\emptyset$ is considered 
below, in the context of $(\varepsilon,\delta)$-domains and for $p<\infty$.

\begin{lemma}\label{LLDEnse}
Let $\Omega$ be an $(\varepsilon,\delta)$-domain in ${\mathbb{R}}^n$
with ${\rm rad}\,(\Omega)>0$,
and fix some $k\in{\mathbb{N}}$ along with some $p\in[1,\infty)$. Then, 
corresponding to the case when $D=\emptyset$, one has 
\begin{eqnarray}\label{PaE-1}
W^{k,p}_{\emptyset}(\Omega)=W^{k,p}(\Omega).
\end{eqnarray}
\end{lemma}

\begin{proof}
The approximation result proved in \cite[\S\,4, pp.\,83-85]{Jon81}
shows that any given $u\in W^{k,p}(\Omega)$ may be approximated arbitrarily well 
by functions of the form $\psi\big|_{\Omega}$ where 
$\psi\in{\mathscr{C}}^\infty({\mathbb{R}}^n)\cap W^{k,\infty}({\mathbb{R}}^n)$
with the property that $\psi\big|_{\Omega}\in W^{k,p}(\Omega)$. Fix now such a 
function $\psi$ and consider $\eta\in{\mathscr{C}}^\infty_c\big(B(0,2)\big)$ 
such that $0\leq\eta\leq 1$ and $\eta\equiv 1$ on $B(0,1)$. For each 
$j\in{\mathbb{N}}$ then set $\eta_j:=\eta(\cdot/j)$ and consider 
$\varphi_j:=\psi\eta_j$. Then clearly 
$\varphi_j\in{\mathscr{C}}^\infty_c({\mathbb{R}}^n)$ and it is 
straightforward to check that $\varphi_j\big|_{\Omega}\to\psi\big|_{\Omega}$ 
in $W^{k,p}(\Omega)$ as $j\to\infty$. This proves the right-to-left inclusion in 
\eqref{PaE-1}. Since the left-to-right inclusion is contained in 
part $(2)$ of Lemma~\ref{TYRD-f4f}, formula \eqref{PaE-1} follows. 
\end{proof}

We are now in a ready to introduce a class of domains for which, roughly speaking, 
the $(\varepsilon,\delta)$-property only holds near a (possibly proper) subset
of the boundary. The following definition is central to the work carried out 
in this paper. 

\begin{definition}\label{OM-Oah-Y88}
Let $\varepsilon,\delta>0$ be given. In addition, suppose that $\Omega$ is 
an open, nonempty, proper subset of ${\mathbb{R}}^n$, and that $N$ is an 
arbitrary subset of $\partial\Omega$. Then $\Omega$ is said to be 
{\tt locally an $(\varepsilon,\delta)$-domain near $N$} provided there exist
a number $\varkappa>0$ and an at most countable family $\{O_j\}_{j\in J}$ of 
open subsets of ${\mathbb{R}}^n$ satisfying
\begin{eqnarray}\label{PPPa.1}
&& \{O_j\}_{j\in J}\,\,\mbox{ is locally finite and has bounded overlap},
\\[4pt]
&& \forall\,j\in J\,\,\,\exists\,\Omega_j\,\,\mbox{ $(\varepsilon,\delta)$-domain 
in ${\mathbb{R}}^n$ with }\,\,{\rm rad}\,(\Omega_j)>\varkappa\,\,\mbox{ and }\,\,
O_j\cap\Omega=O_j\cap\Omega_j,
\label{PPPa.3}
\\[4pt]
&& \exists\,r\in(0,\infty]\,\,\mbox{ such that }\,\,
\forall\,x\in N\,\,\,\exists\,j\in J\,\,\mbox{ for which }\,\,B(x,r)\subseteq O_j.
\label{PPPa.4}
\end{eqnarray}  
\end{definition}

\noindent Occasionally, when the nature of the set $N$ is not important, 
or it is clear from the context, we shall slightly abuse language and refer 
to a domain $\Omega$ as in Definition~\ref{OM-Oah-Y88} simply as being  
{\tt locally an $(\varepsilon,\delta)$-domain}. 

It is obvious from Definition~\ref{OM-Oah-Y88} that the following 
hereditary property holds: if $\Omega$ is an open, nonempty, proper subset of ${\mathbb{R}}^n$ which is locally an $(\varepsilon,\delta)$-domain near 
a subset $N$ of $\partial\Omega$, then $\Omega$ continues to be locally an $(\varepsilon,\delta)$-domain near any subset $N_o$ of $N$. Other features
of the class of domains considered in Definition~\ref{OM-Oah-Y88} are reviewed below.

\begin{remark}\label{Erwv65}
Suppose that $\varepsilon,\delta>0$ are given. Then as a consequence
of Lebesgue's Number Lemma, one may readily verify that 
an open, nonempty, proper subset $\Omega$ of ${\mathbb{R}}^n$ 
is locally an $(\varepsilon,\delta)$-domain near a bounded subset
$N$ of $\partial\Omega$ if and only if there exists
an finite open cover $\{O_j\}_{j\in J}$ of $\overline{N}$ with the 
property that
\begin{eqnarray}\label{PPPa.Tdaf1}
\forall\,j\in J\,\,\,\exists\,\Omega_j\,\,\mbox{ $(\varepsilon,\delta)$-domain 
in ${\mathbb{R}}^n$ with }\,\,
\,\,{\rm rad}\,(\Omega_j)>0\,\,\mbox{ and }\,\,O_j\cap\Omega=O_j\cap\Omega_j.
\end{eqnarray}  
In particular, this more streamlined characterization is valid for the class of 
locally $(\varepsilon,\delta)$-domains with compact boundaries.
\end{remark}

Clearly, when $N$ is a proper subset of the topological boundary allows
a locally $(\varepsilon,\delta)$-domain near $N$ to be quite different than
an ordinary $(\varepsilon,\delta)$-domain in the sense of Definition~\ref{Def-EPDE}
since no condition is imposed on the portion of the boundary outside $N$. 
In particular, the class of domains described in Definition~\ref{Def-EPDE}
corresponding to the case in which $N=\emptyset$ is the collection of all
open, nonempty, proper subsets of ${\mathbb{R}}^n$ (since in this situation
conditions \eqref{PPPa.1}-\eqref{PPPa.4} are satisfied with $J=\emptyset$).
However, in the case in which $N$ coincides with the topological boundary 
of the underlying domain, the classes of domains introduced 
in Definition~\ref{OM-Oah-Y88} and Definition~\ref{Def-EPDE} relate
to one another in the manner made precise in the lemma below.

\begin{lemma}\label{yufa} 
Assume that $\Omega$ is a nonempty proper open
set in ${\mathbb{R}}^n$ and let $\varepsilon,\delta>0$. Then the 
following implications hold:
\begin{eqnarray}\label{Itebn09}
&& \Omega\,\,\mbox{ is an }\,\,(\varepsilon,\delta)\mbox{-domain with }\,\,
{\rm rad}\,(\Omega)>0\,\Longrightarrow\,\,\mbox{
$\Omega$ is a locally $(\varepsilon,\delta)$-domain near $\partial\Omega$},
\\[4pt]
&& \mbox{$\Omega$ is a locally $(\varepsilon,\delta)$-domain near $\partial\Omega$}
\,\,\Longrightarrow\,\,\exists\,\delta'>0\,\,\mbox{ such that }\,\,
\Omega\,\,\mbox{ is an }\,\,(\varepsilon,\delta')\mbox{-domain}. \qquad
\label{Itebn09.bFD}
\end{eqnarray}
\end{lemma}

\begin{proof}
Obviously, any $(\varepsilon,\delta)$-domain $\Omega$ in ${\mathbb{R}}^n$ 
is a locally $(\varepsilon,\delta)$-domain near $\partial\Omega$
(since \eqref{PPPa.1}-\eqref{PPPa.4} are verified if we take $J:=\{1\}$, 
$O_1:={\mathbb{R}}^n$, $\Omega_1:=\Omega$ and $r>0$ arbitrary). 
Conversely, we claim that if $\Omega\subseteq{\mathbb{R}}^n$ is a locally 
$(\varepsilon,\delta)$-domain near $\partial\Omega$ then there exists 
a small number $c=c(\varepsilon,r)>0$, with $r>0$ as in \eqref{PPPa.4}, 
such that $\Omega$ is a actually a $(\varepsilon,c\delta)$-domain in ${\mathbb{R}}^n$.
To see that this is the case, fix two points $x,y\in\Omega$ such that $|x-y|<c\delta$,
with $c\in(0,1)$ small to be determined later. Then the existence of a rectifiable 
curve $\gamma:[0,1]\to\Omega$ such that $\gamma(0)=x$, $\gamma(1)=y$, and 
\eqref{jcdj} holds is readily verified when $B(x,100 c\delta)\cap\partial\Omega
=\emptyset$ (by simply taking $\gamma$ to be the line segment joining $x$ and $y$), 
whereas in the case when $B(x,100 c\delta)\cap\partial\Omega$ contains a point 
$x_o$ one may reason as follows. First, from \eqref{PPPa.4} (with $N=\partial\Omega$)
there exists some $j\in J$ such that $B(x_o,r)\subseteq O_j$. Second, since
$\Omega_j$ is an $(\varepsilon,\delta)$-domain, Definition~\ref{Def-EPDE}
guarantees the existence of a rectifiable curve $\gamma:[0,1]\to\Omega_j$ 
with $\gamma(0)=x$, $\gamma(1)=y$, and such that
\begin{eqnarray}\label{jcdj-Tnav43}
{\rm length}(\gamma)\leq\frac{1}{\varepsilon}|x-y|\quad\mbox{and}\quad
\frac{\varepsilon|z-x|\,|z-y|}{|x-y|}
\leq {\rm dist}\,(z,\partial\Omega_j),\quad\forall\,z\in\gamma([0,1]).
\end{eqnarray}
In particular, ${\rm length}(\gamma)\leq c\delta/\varepsilon$. Also, 
starting from $O_j\cap\Omega=O_j\cap\Omega_j$, elementary topology gives that 
\begin{eqnarray}\label{jcdj-Tnav44}
O_j\cap\partial\Omega=O_j\cap\partial\Omega_j.
\end{eqnarray}
Combining the above facts it is now not difficult to show that by choosing 
$c=c(\varepsilon,r)\in(0,1)$ small enough then 
$\gamma([0,1])\subseteq B(x_o,r)\cap\Omega$ and for every $z\in\gamma([0,1])$
\begin{eqnarray}\label{jcdj-Tnigd}
{\rm dist}\,(z,\partial\Omega)
={\rm dist}\,(z,B(x_o,r)\cap\partial\Omega)
={\rm dist}\,(z,B(x_o,r)\cap\partial\Omega_j)
={\rm dist}\,(z,\partial\Omega_j). 
\end{eqnarray}
With this in hand, the desired conclusion follows. 
\end{proof}

In view of Lemma~\ref{yufa} and the comments preceding 
it, the class of domains introduced in Definition~\ref{Def-EPDE} 
bridges between arbitrary open, nonempty, proper subsets of ${\mathbb{R}}^n$,
on the one hand, and Jones' class of $(\varepsilon,\delta)$-domains
(with the set $N$ playing the role of a fine-tuning parameter).
The most significant feature of the category of locally $(\varepsilon,\delta)$-domains
is the fact that they are extension domains relative to the scale of Sobolev
spaces introduced in Definition~\ref{IUha-Tw}, in a sense made precise in our 
next theorem. The main ingredients in its proof are Theorem~\ref{cjeg}
(including the nature of the quantitative bound from \eqref{fgu4J.P75E}), 
Theorem~\ref{cj-TRf}, and a geometrically compatible quantitative partition of unity.

\begin{theorem}[Extension Theorem for locally
$(\varepsilon,\delta)$-domains]\label{cjeg.5AD}
Suppose that $\Omega\subseteq{\mathbb{R}}^n$ and $D\subseteq\overline{\Omega}$
are such that $D$ is closed and $\Omega$ is locally an $(\varepsilon,\delta)$-domain 
near $\partial\Omega\setminus D$. Then for any $k\in{\mathbb{N}}$ there exists a 
linear operator ${\mathfrak{E}}_{k,D}$, mapping locally integrable functions in 
$\Omega$ into Lebesgue measurable functions in ${\mathbb{R}}^n$, such that for 
each $p\in[1,\infty]$ and each closed subset $\Sigma$ of $\overline{\Omega}$ 
satisfying 
\begin{eqnarray}\label{Jabna-86GV}
D\cap\partial\Omega\subseteq\Sigma\cap\partial\Omega,
\end{eqnarray}
one has
\begin{eqnarray}\label{Tgb-11VVV.BIS}
{\mathfrak{E}}_{k,D}:W^{k,p}_\Sigma(\Omega)\longrightarrow W^{k,p}_\Sigma({\mathbb{R}}^n)
\quad\mbox{linearly and boundedly},
\end{eqnarray}
(with operator norm controlled in terms of $\varepsilon,\delta,n,k,p$,
and the quantitative aspects of \eqref{PPPa.1}-\eqref{PPPa.4}), and
\begin{eqnarray}\label{Pa-PL2VB.BIS}
\big({\mathfrak{E}}_{k,D}\,u\big)\big|_{\Omega}=u,\quad
\mbox{${\mathscr{L}}^n$-a.e. on $\Omega$ for every $u\in W^{k,p}_\Sigma(\Omega)$}.
\end{eqnarray}
In particular, corresponding to the case when $\Sigma:=D$, one has
\begin{eqnarray}\label{Tgb-11VVV}
&& {\mathfrak{E}}_{k,D}:W^{k,p}_D(\Omega)\longrightarrow W^{k,p}_D({\mathbb{R}}^n)
\quad\mbox{linearly and boundedly, and}
\\[4pt]
&& \big({\mathfrak{E}}_{k,D}\,u\big)\big|_{\Omega}=u,\quad
\mbox{${\mathscr{L}}^n$-a.e. on $\Omega$ for every $u\in W^{k,p}_D(\Omega)$}.
\label{Pa-PL2VB}
\end{eqnarray}
\end{theorem}

\begin{proof}
Introduce $N:=\partial\Omega\setminus D$. We shall consider first the case 
when $N\not=\emptyset$, then indicate the alterations needed if $N=\emptyset$
in the last part of the proof. As a preliminary matter, we shall construct
a quantitative partition of unity which is then used to glue together Jones'
extension operators acting from genuine $(\varepsilon,\delta)$-domains
which agree with $\Omega$ near points in $N$. Turning to specifics, since 
$\Omega$ is assumed to be locally an $(\varepsilon,\delta)$-domain near $N$,
Definition~\ref{OM-Oah-Y88} ensures that there exist a number $\varkappa>0$ 
and a collection $\{O_j\}_{j\in J}$, where $J$ is an at most countable set, 
of open subsets of ${\mathbb{R}}^n$ satisfying
\begin{eqnarray}\label{IOHxffx.1}
&& \exists\,L\in{\mathbb{N}}\,\,\mbox{ such that }\,\,\sum_{j\in J}{\bf 1}_{O_j}\leq L
\,\,\,\mbox{ in }\,\,\,{\mathbb{R}}^n,
\\[4pt]
&& \forall\,K\,\,\mbox{ compact in }\,\,{\mathbb{R}}^n,\,\,\,
\#\{j\in J:\,O_j\cap K\not=\emptyset\}<+\infty,
\label{IOHxffx.2}
\\[4pt]
&& \forall\,j\in J\,\,\,\exists\,\Omega_j\,\,\mbox{ $(\varepsilon,\delta)$-domain 
in ${\mathbb{R}}^n$ with }\,\,{\rm rad}\,(\Omega_j)>\varkappa\,\,\mbox{ and }\,\,
O_j\cap\Omega=O_j\cap\Omega_j,
\label{IOHxffx.3}
\\[4pt]
&& \exists\,r\in(0,\infty)\,\,\mbox{ such that }\,\,
\forall\,x\in N\,\,\,\exists\,j\in J\,\,\mbox{ for which }\,\,B(x,r)\subseteq O_j.
\label{IOHxffx.4}
\end{eqnarray}
Now, generally speaking, for each nonempty $E\subseteq{\mathbb{R}}^n$ and each 
$\rho\in(0,\infty)$ define its $\rho$-contraction by 
\begin{eqnarray}\label{IOHxffx.5b}
[E]_\rho:=\big\{x\in{\mathbb{R}}^n:\,B(x,\rho)\subseteq E\big\}
=\big\{x\in{\mathbb{R}}^n:\,{\rm dist}\,(x,E^c)\geq\rho\big\}.
\end{eqnarray}
This and \eqref{IOHxffx.4} imply that
\begin{eqnarray}\label{IOHxffx.5}
N\subseteq\bigcup_{j\in J}[O_j]_r.
\end{eqnarray}
Choose a function $\theta\in{\mathscr{C}}^\infty_c({\mathbb{R}}^n)$
such that $0\leq\theta\leq 1$, ${\rm supp}\,\theta\subseteq B(0,1)$, and
$\int_{{\mathbb{R}}^n}\theta\,d{\mathscr{L}}^n=1$. As usual, for each
$t>0$, define $\theta_t(x):=t^{-n}\theta(x/t)$ for every $x\in{\mathbb{R}}^n$.
Next, for each $j\in J$, consider (with $r>0$ as in \eqref{IOHxffx.4}) 
\begin{eqnarray}\label{IOHxffx.6}
\psi_j:=\theta_{r/4}\ast{\bf 1}_{[O_j]_{r/4}}\,\,\,\mbox{ in }\,\,
{\mathbb{R}}^n.
\end{eqnarray}
Then for each $j\in J$,
\begin{eqnarray}\label{IOHxffx.7}
\begin{array}{c}
\psi_j\in{\mathscr{C}}^\infty({\mathbb{R}}^n),\quad 0\leq\psi_j\leq 1,
\quad{\rm supp}\,\psi_j\subseteq O_j,\,\,\,
\psi_j\equiv 1\,\,\mbox{ in }\,\,[O_j]_{r/2},\,\,\,
\\[6pt]
\mbox{and }\,\,
\big|\partial^\alpha\psi_j\big|\leq 4^{|\alpha|}r^{-|\alpha|}
\|\partial^\alpha\theta\|_{L^1({\mathbb{R}}^n,{\mathscr{L}}^n)}\quad
\mbox{ for each }\,\,\alpha\in{\mathbb{N}}^n_0.
\end{array}
\end{eqnarray}
Let us also note that $\sum_{j\in J}\psi_j^2$ is a well-defined function belonging to 
${\mathscr{C}}^\infty({\mathbb{R}}^n)$, thanks to \eqref{IOHxffx.7} and 
\eqref{IOHxffx.3}, and that
\begin{eqnarray}\label{IOHxffx.10}
\sum_{j\in J}\psi_j^2\geq 1\,\,\,\mbox{ on }\,\,\,
\{x\in{\mathbb{R}}^n:\,{\rm dist}\,(x,N)<r/2\}.
\end{eqnarray}
Indeed, if $x\in{\mathbb{R}}^n$ is such that ${\rm dist}\,(x,N)<r/2$, 
then there exists $x_0\in N$ such that $x\in B(x_0,r/2)$.
Thanks to \eqref{IOHxffx.4} we know that there 
exists $j_0\in J$ such that $B(x_0,r)\subseteq O_{j_0}$. 
Combining these facts it follows that $x\in[O_{j_0}]_{r/2}$, 
whereupon $\psi_{j_0}(x)=1$ by \eqref{IOHxffx.7}. Consequently, 
$\sum_{j\in J}\psi_j(x)^2\geq 1$, as wanted. 

Next, introduce 
\begin{eqnarray}\label{IOHxffx.7ia}
U:=\{x\in{\mathbb{R}}^n:\,{\rm dist}\,(x,N)<r/4\}.
\end{eqnarray}
If we now define 
\begin{eqnarray}\label{IOHxffx.8}
\eta:=\theta_{r/8}\ast{\bf 1}_{U}\,\,\,\mbox{ in }\,\,{\mathbb{R}}^n,
\end{eqnarray}
then 
\begin{eqnarray}\label{IOHxffx.9}
\begin{array}{c}
\eta\in{\mathscr{C}}^\infty({\mathbb{R}}^n),\,\,\,0\leq\eta\leq 1,
\\[6pt]
{\rm supp}\,\eta\subseteq\{x\in{\mathbb{R}}^n:\,{\rm dist}\,(x,N)<r/2\},\,\,\,
\\[6pt]
\eta\equiv 1\,\,\mbox{ in }\,\,\{x\in{\mathbb{R}}^n:\,{\rm dist}\,(x,N)<r/8\},\,\,\,
\\[6pt]
\mbox{and }\,\,
\big|\partial^\alpha\eta\big|\leq 8^{|\alpha|}r^{-|\alpha|}
\|\partial^\alpha\theta\|_{L^1({\mathbb{R}}^n,{\mathscr{L}}^n)}\quad
\mbox{ for each }\,\,\alpha\in{\mathbb{N}}^n_0.
\end{array}
\end{eqnarray}
Given that by \eqref{IOHxffx.10} and \eqref{IOHxffx.9} we have 
$\sum_{j\in J}\psi_j^2\geq 1$ in a neighborhood of ${\rm supp}\,\eta$, 
it follows that if for each $j\in J$ we now set 
\begin{eqnarray}\label{IOHxffx.11}
\varphi_j:=\frac{\eta\psi_j}{\sum_{i\in J}\psi_i^2},
\end{eqnarray}
then for every $j\in J$
\begin{eqnarray}\label{IOHxffx.12}
\begin{array}{c}
\varphi_j\in{\mathscr{C}}^\infty({\mathbb{R}}^n),\quad
{\rm supp}\,\varphi_j\subseteq O_j,\quad 0\leq\varphi_j\leq 1,
\\[6pt]
\mbox{and }\,\,
\big|\partial^\alpha\varphi_j\big|\leq C_{\theta,\alpha}\,r^{-|\alpha|},\quad
\mbox{ for each }\,\,\alpha\in{\mathbb{N}}^n_0,
\end{array}
\end{eqnarray}
where $C_{\theta,\alpha}>0$ is a finite constant independent of $j$. Moreover,
\begin{eqnarray}\label{IOHxffx.13}
\sum_{j\in J}\psi_j\varphi_j=\eta\,\,\mbox{ in }\,\,{\mathbb{R}}^n.
\end{eqnarray}

Fix now a closed set $\Sigma$ satisfying $D\subseteq\Sigma
\subseteq\overline{\Omega}$, along with a number $k\in{\mathbb{N}}$ and, 
for each $j\in J$, denote by $\Lambda_{k,j}$ Jones' extension operator 
for the $(\varepsilon,\delta)$-domain $\Omega_j$. In this context, we now define the operator 
\begin{eqnarray}\label{Tgb-ak}
{\mathfrak{E}}_{k,D}\,u:=\widetilde{(1-\eta)u}+\sum_{j\in J}\varphi_j
\Lambda_{k,j}\big(E_j(\psi_ju)\big),
\quad\mbox{for every }\,\,u\in{\mathscr{C}}^\infty_\Sigma(\Omega),
\end{eqnarray}
where tilde is the operation of extending functions from 
${\mathscr{C}}^\infty_c(\Omega)$ by zero outside of their support to the entire 
${\mathbb{R}}^n$. Also, for each $j\in J$ we have denoted by $E_j$ 
the operator mapping each function $v$ from ${\mathscr{C}}^\infty(\overline{\Omega})$ 
with the property that there exits a compact subset $K$ of $O_j$ such that 
$v\equiv 0$ on $\Omega\setminus K$ into 
\begin{eqnarray}\label{Tgb-akiOpp}
E_j v:=\left\{
\begin{array}{ll}
v & \mbox{ in }\,\,\Omega\cap O_j=\Omega_j\cap O_j,
\\[4pt]
0 & \mbox{ in }\,\,\Omega_j\setminus O_j.
\end{array}
\right.
\end{eqnarray}
In particular, it is clear that for any $v$ as above,
\begin{eqnarray}\label{Tgb-akiOpp.T}
\begin{array}{c}
E_jv\in{\mathscr{C}}^\infty(\overline{\Omega_j}\,),\quad
{\rm supp}\,(E_jv)\subseteq{\rm supp}\,v,\,\,\,\mbox{ and}
\\[8pt]
\partial^\alpha (E_jv)=E_j(\partial^\alpha v)\,\,\mbox{ in }\,\,\Omega_j
\,\,\,\mbox{ for each }\,\,\,\alpha\in{\mathbb{N}}_0^n.
\end{array}
\end{eqnarray}
With these conventions we first claim that, given any 
$u\in{\mathscr{C}}^\infty_\Sigma(\Omega)$, the right-hand side of \eqref{Tgb-ak}
is well-defined. To justify this claim note that if $u=\varphi\big|_{\Omega}$
for some $\varphi\in{\mathscr{C}}^\infty_c({\mathbb{R}}^n)$ with the
property that $\Sigma\cap{\rm supp}\,\varphi=\emptyset$, then  
\begin{eqnarray}\label{Tgb-aUhaYav}
{\mathfrak{K}}:={\rm supp}\,\varphi\cap
\{x\in{\mathbb{R}}^n:\,\,{\rm dist}\,(x,N)\geq r/8\}
\end{eqnarray}
is a compact subset of ${\mathbb{R}}^n$ which satisfies
${\mathfrak{K}}\cap N=\emptyset$ and ${\mathfrak{K}}\cap\partial\Omega\cap D
\subseteq{\mathfrak{K}}\cap\partial\Omega\cap\Sigma=\emptyset$, by \eqref{Jabna-86GV}. 
Hence, 
\begin{eqnarray}\label{Tgb-iahYBBv}
{\mathfrak{K}}\cap\partial\Omega=\big({\mathfrak{K}}\cap N\big)
\cup\big({\mathfrak{K}}\cap\partial\Omega\cap D)=\emptyset.
\end{eqnarray}
As such, ${\mathfrak{K}}\cap\Omega$ is a compact subset of $\Omega$ outside of which 
the function $(1-\eta)u$ vanishes identically. 
This shows that $(1-\eta)u\in {\mathscr{C}}^\infty_c(\Omega)$,
hence the first term in the right-hand side of \eqref{Tgb-ak} is meaningful.
The sum in the right-hand side of \eqref{Tgb-ak} is also meaningful,
thanks to the support condition in \eqref{IOHxffx.12} and \eqref{IOHxffx.2}.
In fact, \eqref{IOHxffx.12}-\eqref{IOHxffx.2} may also be used to justify that,
given an arbitrary function $u\in{\mathscr{C}}^\infty_\Sigma(\Omega)$, 
for each $\alpha\in{\mathbb{N}}_0^n$ with $|\alpha|\leq k$, we have
\begin{eqnarray}\label{TTag-Ugat653e.f}
\partial^\alpha{\mathfrak{E}}_{k,D}\,u
=\sum_{\beta+\gamma=\alpha}\frac{\alpha!}{\beta!\gamma!}(-1)^{|\beta|}
\widetilde{\partial^\beta\eta\partial^\gamma u}
+\sum_{j\in J}\sum_{\beta+\gamma=\alpha}\frac{\alpha!}{\beta!\gamma!}
\partial^\beta\varphi_j\partial^\gamma\Lambda_{k,j}\big(E_j(\psi_ju)\big).
\end{eqnarray}
Consequently, given any $p\in[1,\infty)$, 
for every $u\in{\mathscr{C}}^\infty_\Sigma(\Omega)$ we may estimate
\begin{eqnarray}\label{TTag-Ugat653e}
\big|\partial^\alpha{\mathfrak{E}}_{k,D}\,u\big|^p
&\leq & C_k\Big(\sum_{\beta+\gamma=\alpha}
\widetilde{|\partial^\beta\eta||\partial^\gamma u|}
+\sum_{j\in J}\sum_{\beta+\gamma=\alpha}|\partial^\beta\varphi_j|
\big|\partial^\gamma\Lambda_{k,j}\big(E_j(\psi_ju)\big)\big|\Big)^p
\nonumber\\[4pt]
&\leq & C_{k,\theta,r}\Big(\sum_{\gamma\leq\alpha}
\widetilde{|\partial^\gamma u|}
+\sum_{j\in J}\sum_{\gamma\leq\alpha}{\bf 1}_{O_j}
\big|\partial^\gamma\Lambda_{k,j}\big(E_j(\psi_ju)\big)\big|\Big)^p
\nonumber\\[4pt]
&\leq & C_{k,\theta,r,L,p}\Big(\sum_{\gamma\leq\alpha}
\widetilde{|\partial^\gamma u|^p}
+\sum_{j\in J}\sum_{\gamma\leq\alpha}
\big|\partial^\gamma\Lambda_{k,j}\big(E_j(\psi_ju)\big)\big|^p\Big),
\end{eqnarray}
where the last inequality uses the support condition in \eqref{IOHxffx.12} 
and \eqref{IOHxffx.1}. Thus, further,
\begin{eqnarray}\label{gat6Toh6yb}
\int_{{\mathbb{R}}^n}\big|\partial^\alpha{\mathfrak{E}}_{k,D}\,u\big|^p\,d{\mathscr{L}}^n
&\leq & C_{k,\theta,r,L,p}\Big(\int_{{\mathbb{R}}^n}\sum_{\gamma\leq\alpha}
\widetilde{|\partial^\gamma u|^p}\,d{\mathscr{L}}^n
+\sum_{j\in J}\sum_{\gamma\leq\alpha}\int_{{\mathbb{R}}^n}
\big|\partial^\gamma\Lambda_{k,j}\big(E_j(\psi_ju)\big)\big|^p\,d{\mathscr{L}}^n\Big)
\nonumber\\[4pt]
&\leq & C_{k,\theta,r,L,p}\Big(\int_{\Omega}\sum_{\gamma\leq\alpha}
|\partial^\gamma u|^p\,d{\mathscr{L}}^n
+\sum_{j\in J}\big\|\Lambda_{k,j}\big(E_j(\psi_ju)\big)\big\|_{W^{k,p}({\mathbb{R}}^n)}^p
\Big)
\nonumber\\[4pt]
&\leq & C_{k,\theta,r,L,p,\varepsilon,\delta,n,\varkappa}\Big(\|u\|_{W^{k,p}(\Omega)}^p
+\sum_{j\in J}\big\|E_j(\psi_ju)\big\|_{W^{k,p}(\Omega_j)}^p
\Big)
\nonumber\\[4pt]
&=& C_{k,\theta,r,L,p,\varepsilon,\delta,n,\varkappa}\Big(\|u\|_{W^{k,p}(\Omega)}^p
+\sum_{j\in J}\|\psi_ju\|_{W^{k,p}(\Omega)}^p\Big)
\nonumber\\[4pt]
&\leq & C_{k,\theta,r,L,p,\varepsilon,\delta,n,\varkappa}\Big(\|u\|_{W^{k,p}(\Omega)}^p
+\sum_{j\in J}\sum_{\beta+\gamma=\alpha}\frac{\alpha!}{\beta!\gamma!}
\int_{\Omega}|\partial^\beta\psi_j|^p|\partial^\gamma u|^p\,d{\mathscr{L}}^n\Big)
\nonumber\\[4pt]
&\leq & C_{k,\theta,r,L,p,\varepsilon,\delta,n,\varkappa}\Big(\|u\|_{W^{k,p}(\Omega)}^p
+\sum_{j\in J}\sum_{\gamma\leq\alpha}
\int_{\Omega}{\bf 1}_{O_j}|\partial^\gamma u|^p\,d{\mathscr{L}}^n\Big)
\nonumber\\[4pt]
& = & C_{k,\theta,r,L,p,\varepsilon,\delta,n,\varkappa}\Big(\|u\|_{W^{k,p}(\Omega)}^p
+\sum_{\gamma\leq\alpha}
\int_{\Omega}\Big(\sum_{j\in J}{\bf 1}_{O_j}\Big)
|\partial^\gamma u|^p\,d{\mathscr{L}}^n\Big)
\nonumber\\[4pt]
&\leq & C_{k,\theta,r,L,p,\varepsilon,\delta,n,\varkappa}\Big(\|u\|_{W^{k,p}(\Omega)}^p
+L\sum_{\gamma\leq\alpha}\int_{\Omega}
|\partial^\gamma u|^p\,d{\mathscr{L}}^n\Big)
\nonumber\\[4pt]
&\leq & C_{k,\theta,r,L,p,\varepsilon,\delta,n,\varkappa}\,\|u\|_{W^{k,p}(\Omega)}^p.
\end{eqnarray}
Above, the first inequality is implied by \eqref{TTag-Ugat653e}, the second 
inequality is obvious, the third inequality is based on Theorem~\ref{cjeg} and the 
fact that each $\Omega_j$ is an $(\varepsilon,\delta)$-domain in ${\mathbb{R}}^n$
with ${\rm rad}\,(\Omega_j)\geq,\varkappa$, 
the subsequent equality is readily seen from \eqref{Tgb-akiOpp}-\eqref{Tgb-akiOpp.T},
the fourth inequality uses Leibniz's formula, the fifth inequality follows
from \eqref{IOHxffx.7}, the next equality is trivial, the sixth inequality 
is clear from \eqref{IOHxffx.1}, while the last inequality is obvious.
In turn, \eqref{gat6Toh6yb} goes to show that
\begin{eqnarray}\label{InavfE59}
\|{\mathfrak{E}}_{k,D}\,u\|_{W^{k,p}({\mathbb{R}}^n)}
\leq C_{k,\theta,r,L,p,\varepsilon,\delta,n,\varkappa}\,\|u\|_{W^{k,p}(\Omega)},
\qquad\forall\,u\in{\mathscr{C}}^\infty_\Sigma(\Omega),
\end{eqnarray}
at least if $p\in[1,\infty)$. In fact, a minor variation of the above argument 
shows that estimate \eqref{InavfE59} is actually valid for $p=\infty$ as well.
This is based on the fact that for any family of $[0,\infty]$-valued 
Lebesgue-measurable functions $\{\xi_j\}_{j\in J}$ in ${\mathbb{R}}^n$
with the property that there exists $M\in{\mathbb{N}}$ such that 
$\#\{j\in J:\,\xi_j(x)\not=0\}\leq M$ for ${\mathscr{L}}^n$-a.e. $x\in{\mathbb{R}}^n$,
there holds
\begin{eqnarray}\label{bavTGa65i}
\Big\|\sum_{j\in J}\xi_j\Big\|_{L^\infty({\mathbb{R}}^n,{\mathscr{L}}^n)}
\leq M\,\sup_{j\in J}\|\xi_j\|_{L^\infty({\mathbb{R}}^n,{\mathscr{L}}^n)}.
\end{eqnarray}

Based on \eqref{InavfE59} and \eqref{uig-adjb-2}, we may therefore conclude that 
\begin{eqnarray}\label{fgu4J.ii}
{\mathfrak{E}}_{k,D}:W^{k,p}_\Sigma(\Omega)\longrightarrow W^{k,p}({\mathbb{R}}^n)
\quad\mbox{ linearly and boundedly},
\end{eqnarray}
with operator norm controlled in terms of $n,\varepsilon,\delta,k,p$.
On account of this and \eqref{uig-adjb-2AAA}, the claim in \eqref{Tgb-11VVV.BIS}
will follow as soon as we show that 
\begin{eqnarray}\label{uig-aTL7.B.L}
{\mathfrak{E}}_{k,D}\big[{\mathscr{C}}^\infty_\Sigma(\Omega)\big]
\subseteq W^{k,p}_\Sigma({\mathbb{R}}^n).
\end{eqnarray}
To this end, we first remark that 
\begin{eqnarray}\label{Ian-Yba}
\overline{\Omega}\cap O_j\subseteq\overline{\Omega\cap O_j}
\,\,\,\mbox{ for every }\,\,j\in\{1,...,J\}.
\end{eqnarray}
Indeed, if $j\in J$ and $x\in\overline{\Omega}\cap O_j$ are arbitrary,
then there exists a sequence $\{x_i\}_{i\in{\mathbb{N}}}\subseteq\Omega$
such that $\lim\limits_{i\to\infty}x_i=x\in O_j$. Given that 
$O_j$ is open, there is no loss of generality in assuming that 
$\{x_i\}_{i\in{\mathbb{N}}}\subseteq\Omega\cap O_j$ which, in turn, goes 
to show that the limit point $x$ belongs to $\overline{\Omega\cap O_j}$.
This proves \eqref{Ian-Yba}. Now, if $u\in{\mathscr{C}}^\infty_\Sigma(\Omega)$ 
then for every $j\in J$ we may write 
\begin{eqnarray}\label{TmmmN}
\overline{\Omega}\cap{\rm supp}\,\Big(\varphi_j\Lambda_{k,j}
\big(E_j(\psi_j u)\big)\Big)
&\subseteq & \overline{\Omega}\cap O_j
\cap{\rm supp}\,\Big(\Lambda_{k,j}
\big(E_j(\psi_j u)\big)\Big)
\nonumber\\[4pt]
&\subseteq & \big(\,\overline{\Omega\cap O_j}\,\big)
\cap{\rm supp}\,\Big(\Lambda_{k,j}\big(E_j(\psi_j u)\big)\Big)
\nonumber\\[4pt]
&=& \big(\,\overline{\Omega_j\cap O_j}\,\big)
\cap{\rm supp}\,\Big(\Lambda_{k,j}\big(E_j(\psi_j u)\big)\Big)
\nonumber\\[4pt]
&\subseteq & \overline{\Omega_j}
\cap{\rm supp}\,\Big(\Lambda_{k,j}\big(E_j(\psi_j u)\big)\Big)
\nonumber\\[4pt]
&=& {\rm supp}\,\big(E_j(\psi_j u)\big)
\nonumber\\[4pt]
&\subseteq & {\rm supp}\,(\psi_j u)
\nonumber\\[4pt]
&\subseteq & {\rm supp}\,u.
\end{eqnarray}
The first inclusion above is a consequence of the fact that 
${\rm supp}\,\varphi_j\subseteq{\mathcal{O}}_j$, the second inclusion is based
on \eqref{Ian-Yba}, the first equality is guaranteed by \eqref{IOHxffx.3},
the third inclusion is obvious, the second equality follows from Theorem~\ref{cj-TRf},
the penultimate inclusion is implied by \eqref{Tgb-akiOpp.T}, while the last one is 
obvious. From \eqref{Tgb-ak} and \eqref{TmmmN} we then deduce that 
\begin{eqnarray}\label{TajbUJM}
\overline{\Omega}\cap{\rm supp}\,{\mathfrak{E}}_{k,D}\,u\subseteq{\rm supp}\,u
\,\,\,\mbox{ for each }\,\,u\in{\mathscr{C}}^\infty_\Sigma(\Omega).
\end{eqnarray}
As a consequence, from \eqref{TajbUJM} and the fact that 
$\Sigma\subseteq\overline{\Omega}$, for each 
$u\in{\mathscr{C}}^\infty_\Sigma(\Omega)$ we have
\begin{eqnarray}\label{TajbUJM.54}
\Sigma\cap{\rm supp}\,{\mathfrak{E}}_{k,D}\,u=\big(\,\Sigma\cap\overline{\Omega}\,\big)
\cap{\rm supp}\,{\mathfrak{E}}_{k,D}\,u=\Sigma\cap\big(\,\overline{\Omega}
\cap{\rm supp}\,{\mathfrak{E}}_{k,D}\,u\,\big)
\subseteq\Sigma\cap{\rm supp}\,u=\emptyset.
\end{eqnarray}
Hence, $\Sigma\cap{\rm supp}\,{\mathfrak{E}}_{k,D}\,u=\emptyset$ for each 
$u\in{\mathscr{C}}^\infty_\Sigma(\Omega)$. Given that for each 
$u\in{\mathscr{C}}^\infty_\Sigma(\Omega)$ the set ${\rm supp}\,{\mathfrak{E}}_{k,D}\,u$ 
is compact (by \eqref{Tgb-ak} and \eqref{fgu4J-Dyab.4RR} in Theorem~\ref{cj-TRf}),
and since $\Sigma$ is closed, it follows that 
\begin{eqnarray}\label{aNNE-1}
{\rm dist}\,\big(\Sigma\,,\,{\rm supp}\,{\mathfrak{E}}_{k,D}\,u\big)>0
\,\,\,\mbox{ for each }\,\,u\in{\mathscr{C}}^\infty_\Sigma(\Omega).
\end{eqnarray}
At this stage, for each $i\in{\mathbb{N}}$
define $\theta_i(x):=i^n\theta(ix)$ for every $x\in{\mathbb{R}}^n$ and,
having fixed some $u\in{\mathscr{C}}^\infty_\Sigma(\Omega)$, set 
\begin{eqnarray}\label{f-Gs1.Ta.W}
\xi_i:=\theta_i\ast{\mathfrak{E}}_{k,D}\,u\,\,\mbox{ in }\,\,{\mathbb{R}}^n,
\,\,\mbox{ for each }\,\,i\in{\mathbb{N}}.
\end{eqnarray}
Then 
\begin{eqnarray}\label{f-Gs1.Ta}
\begin{array}{c}
\xi_i\in{\mathscr{C}}^\infty_c({\mathbb{R}}^n),\,\,\mbox{ and }\,\,
\Sigma\cap{\rm supp}\,\xi_i=\emptyset\,\,\,\mbox{ if $i$ is large enough},
\\[8pt]
\mbox{and }\,\,\,
\xi_i\longrightarrow{\mathfrak{E}}_{k,D}\,u\,\,\,\mbox{ in }\,\,W^{k,p}({\mathbb{R}}^n)
\,\,\,\mbox{ as }\,\,i\to\infty,
\end{array}
\end{eqnarray}
by virtue of \eqref{aNNE-1} and the fact that ${\mathfrak{E}}_{k,D}\,u$ 
has compact support and belongs to $W^{k,p}({\mathbb{R}}^n)$. 
In light of \eqref{uig-adjb-2AAA}, the approximation result in \eqref{f-Gs1.Ta} 
implies that actually 
\begin{eqnarray}\label{LJE-ajan}
{\mathfrak{E}}_{k,D}\,u\in W^{k,p}_\Sigma({\mathbb{R}}^n)\,\,\mbox{ for each }\,\,
u\in{\mathscr{C}}^\infty_\Sigma(\Omega).
\end{eqnarray}
From \eqref{LJE-ajan}, \eqref{fgu4J.ii}, \eqref{uig-adjb-2}, and the fact that 
$W^{k,p}_\Sigma({\mathbb{R}}^n)$ is a closed subspace of $W^{k,p}({\mathbb{R}}^n)$, 
it now follows that the operator \eqref{Tgb-ak} extends to a linear and bounded
mapping as in \eqref{Tgb-11VVV.BIS}. 

There remains to show that the mapping just defined also satisfies \eqref{Pa-PL2VB.BIS}.
To this end, for each $j\in J$ denote by $F_j$ the operator mapping functions $w$ 
defined in $O_j\cap\Omega$ into functions defined in $\Omega$ according to 
\begin{eqnarray}\label{nab-7gVh6}
F_j w:=\left\{
\begin{array}{ll}
w & \mbox{ in }\,\,\Omega\cap O_j=\Omega_j\cap O_j,
\\[4pt]
0 & \mbox{ in }\,\,\Omega\setminus O_j.
\end{array}
\right.
\end{eqnarray}
For any function $u\in{\mathscr{C}}^\infty_\Sigma(\Omega)$ we then have
\begin{eqnarray}\label{Tgb-ak.887}
{\mathfrak{E}}_{k,D}\,u\Big|_{\Omega} &=& \widetilde{(1-\eta)u}\Big|_{\Omega}
+\sum_{j\in J}\Big[\varphi_j\Lambda_{k,j}\big(E_j(\psi_j u)\big)
\Big]\Big|_{\Omega}
\nonumber\\[4pt]
&=& (1-\eta)u+\sum_{j\in J}\big(\varphi_j\big|_{\Omega}\big)
F_j\left(\Big[\Lambda_{k,j}\big(E_j(\psi_j u)\big)
\Big]\Big|_{O_j\cap\Omega}\right)
\nonumber\\[4pt]
&=& (1-\eta)u+\sum_{j\in J}\big(\varphi_j\big|_{\Omega}\big)
F_j\left(\Big[\Lambda_{k,j}\big(E_j(\psi_j u)\big)
\Big]\Big|_{O_j\cap\Omega_j}\right)
\nonumber\\[4pt]
&=& (1-\eta)u+\sum_{j\in J}\big(\varphi_j\big|_{\Omega}\big)
F_j\left(\Big[E_j(\psi_j u)\Big]\Big|_{O_j\cap\Omega_j}\right)
\nonumber\\[4pt]
&=& (1-\eta)u+\sum_{j\in J}\varphi_j
F_j\left(\big(\psi_j u\big)\Big|_{O_j\cap\Omega}\right)
\nonumber\\[4pt]
&=& (1-\eta)u+\sum_{j\in J}\varphi_j\psi_j u
\nonumber\\[4pt]
&=& (1-\eta)u+\eta u=u,
\end{eqnarray}
thanks to \eqref{Tgb-ak}, \eqref{IOHxffx.2}, \eqref{IOHxffx.3}, \eqref{nab-7gVh6},
the fact that $\Lambda_{k,j}$ is an extension operator for $O_j\cap\Omega_j$ for each $j\in J$, as well as \eqref{Tgb-akiOpp} and \eqref{Tgb-akiOpp.T}.
Having established this, \eqref{Pa-PL2VB.BIS} now follows from \eqref{Tgb-ak.887}, 
part $(4)$ in Lemma~\ref{TYRD-f4f}, \eqref{Tgb-11VVV.BIS}, and \eqref{uig-adjb-2}.
This concludes the construction and analysis of the operator 
${\mathfrak{E}}_{k,D}$ in the case when $N\not=\emptyset$.

Consider now the case when $N=\emptyset$, i.e., when $\partial\Omega\subseteq D$;
in particular, $\partial\Omega\subseteq\Sigma$ by \eqref{Jabna-86GV}. 
In this scenario, $\Omega$ is just an arbitrary open, nonempty, proper subset of ${\mathbb{R}}^n$ (as noted in the comments before the statement of Lemma~\ref{yufa}),
and we define the operator 
\begin{eqnarray}\label{YL-uvv.rt2}
{\mathfrak{E}}_{k,D}:W^{k,p}_\Sigma(\Omega)\longrightarrow W^{k,p}_\Sigma({\mathbb{R}}^n),
\qquad
{\mathfrak{E}}_{k,D}\,u:=\left\{
\begin{array}{ll}
u & \mbox{ in }\,\,\Omega,
\\[4pt]
0 & \mbox{ in }\,\,\Omega^c:={\mathbb{R}}^n\setminus\Omega,
\end{array}
\right.
\qquad\forall\,u\in W^{k,p}_\Sigma(\Omega),
\end{eqnarray}
which formally corresponds to choosing $\eta\equiv 0$ and $J:=\emptyset$
in \eqref{Tgb-ak}.
Property \eqref{Pa-PL2VB.BIS} is now a simple feature of the design 
of ${\mathfrak{E}}_{k,D}$ and we are left with checking that this 
operator maps $W^{k,p}_\Sigma(\Omega)$ boundedly into $W^{k,p}_\Sigma({\mathbb{R}}^n)$.
With this goal in mind, consider first the operator 
${\mathfrak{E}}_{k,D}:W^{k,p}_\Sigma(\Omega)\rightarrow W^{k,p}({\mathbb{R}}^n)$
defined by the same formula as in the last part of \eqref{YL-uvv.rt2},
and observe that this operator may be viewed as a composition between the 
isometric inclusion $W^{k,p}_\Sigma(\Omega)\hookrightarrow 
W^{k,p}_{\partial\Omega}(\Omega)$ (cf. $(8)$ in Lemma~\ref{TYRD-f4f}) and 
the bounded linear mapping from \eqref{YJBKL-uvv.5}.
That this composition is meaningful is ensured by part $(5)$ in Lemma~\ref{TYRD-f4f}.
Given the goals we have in mind, there remains to prove that
$\widetilde{u}\in W^{k,p}_\Sigma({\mathbb{R}}^n)$ for every
$u\in W^{k,p}_\Sigma(\Omega)$ (where tilde denotes the extension by zero from $\Omega$
to ${\mathbb{R}}^n$). To justify this, note that if $u\in W^{k,p}_\Sigma(\Omega)$ 
then $u\in{\mathring{W}}^{k,p}(\Omega)$ by  $(5)$ in Lemma~\ref{TYRD-f4f}, and also
there exists a sequence $\{\varphi_j\}_{j\in{\mathbb{N}}}\subseteq
{\mathscr{C}}^\infty_c({\mathbb{R}}^n)$ such that 
$\Sigma\cap{\rm supp}\,\varphi_j=\emptyset$
for each $j\in{\mathbb{N}}$ and $\varphi_j\big|_\Omega\to u$ in 
$W^{k,p}(\Omega)$ as $j\to\infty$, by \eqref{uig-adjb-2}.
In particular, $\varphi_j\big|_\Omega\in{\mathscr{C}}^\infty_c(\Omega)\subseteq
{\mathring{W}}^{k,p}(\Omega)$ for every $j\in{\mathbb{N}}$, by the 
support condition and our assumption on $D$. In turn, this readily entails that
for each $j\in{\mathbb{N}}$ we have $\widetilde{\varphi_j|_\Omega}\in
{\mathscr{C}}^\infty_c({\mathbb{R}}^n)$ and $\Sigma\cap{\rm supp}\,
\big(\,\widetilde{\varphi_j|_\Omega}\,\big)=\emptyset$ which, in light of
\eqref{YJBKL-uvv.5}, implies that $\widetilde{\varphi_j|_\Omega}\to\widetilde{u}$ in $W^{k,p}({\mathbb{R}}^n)$ as $j\to\infty$. By \eqref{uig-adjb-2AAA}, the latter
convergence may be interpreted as saying that 
$\widetilde{u}\in W^{k,p}_\Sigma({\mathbb{R}}^n)$, as wanted.
Hence, the operator \eqref{YL-uvv.rt2} has all desired properties in 
this case as well, and this completes the proof of the theorem.
\end{proof}

It is instructive to note that Theorem~\ref{cjeg.5AD} contains as a 
particular case the fact that Jones' mapping $\Lambda_k$ continues
to be a well-defined and bounded extension operator  
on the scale of Sobolev spaces introduced in Definition~\ref{IUha-Tw}
considered on $(\varepsilon,\delta)$-domains. This is made precise in
the following corollary.

\begin{corollary}\label{cjeg.5}
Let $\Omega$ be an $(\varepsilon,\delta)$-domain in ${\mathbb{R}}^n$ 
with ${\rm rad}\,(\Omega)>0$, and fix 
an arbitrary number $k\in{\mathbb{N}}$. Then Jones' extension operator $\Lambda_k$ 
(from Theorem~\ref{cjeg}) has the property that, for each closed subset $D$ 
of $\overline{\Omega}$ and each $p\in[1,\infty]$,
\begin{eqnarray}\label{fgu4J}
\Lambda_k:W^{k,p}_D(\Omega)\longrightarrow W^{k,p}_D({\mathbb{R}}^n)
\quad\mbox{ linearly and boundedly},
\end{eqnarray}
with operator norm controlled solely in terms of $n,\varepsilon,\delta,k,p$.
\end{corollary}

\begin{proof}
Let $\Omega$ be an $(\varepsilon,\delta)$-domain in ${\mathbb{R}}^n$
with ${\rm rad}\,(\Omega)>0$ and fix a closed subset $D$ of $\overline{\Omega}$. 
Then the hypotheses of Theorem~\ref{cjeg.5AD} are satisfied if 
we choose $\Sigma$ to be the current $D$ and take $D$ (in the statement 
of Theorem~\ref{cjeg.5AD}) to be the empty set. Indeed, conditions 
\eqref{PPPa.1}-\eqref{PPPa.4} are presently satisfied if we take
$J:=\{1\}$, $O_1:={\mathbb{R}}^n$, $\Omega_1:=\Omega$ and $r:=\infty$.
In such a scenario, the choice $\eta=\varphi_1=\psi_1\equiv 1$ in ${\mathbb{R}}^n$
is permissible and formula \eqref{Tgb-ak} reduces to ${\mathfrak{E}}_{k,\emptyset}
=\Lambda_k$, Jones' extension operator for the domain $\Omega$. As such, 
all desired conclusions follow from Theorem~\ref{cjeg.5AD}.
\end{proof}

\noindent It should be noted that, at least if $1\leq p<\infty$, 
specializing Corollary~\ref{cjeg.5} to the particular case $D:=\emptyset$ 
yields (in light of Lemma~\ref{LLDEnse}) Jones' extension result recorded 
in Theorem~\ref{cjeg}. Corollary~\ref{cjeg.5} may be regarded as a suitable 
analogue of the property of Calder\'on's extension operator in Lipschitz domains
of not increasing the support of functions $u\in W^{k,p}(\Omega)$ which vanish 
near $\partial\Omega$, in the case of Jones' extension operator in 
$(\varepsilon,\delta)$-domains. 

Substituting Theorem~\ref{cjeg.UUV} for Theorem~\ref{cjeg} in the 
proof of Theorem~\ref{cjeg.5AD} yields a semi-universal extension operator 
for locally $(\varepsilon,\delta)$-domains. Specifically, we have the following result.

\begin{theorem}[A semi-universal extension operator for locally
$(\varepsilon,\delta)$-domains]\label{cjeg.5AD.UUV}
Assume that $\Omega\subseteq{\mathbb{R}}^n$ and $D\subseteq\overline{\Omega}$
are such that $D$ is closed and $\Omega$ is locally an $(\varepsilon,\delta)$-domain 
near $\partial\Omega\setminus D$. Then for any $\Upsilon\in{\mathbb{N}}$ there exists 
a linear operator ${\mathfrak{E}}_{D}$, mapping locally integrable functions in 
$\Omega$ into Lebesgue measurable functions in ${\mathbb{R}}^n$, such that for 
each $p\in(1,\infty]$, each $k\in{\mathbb{N}}$ such that $k<\Upsilon$, and each 
closed subset $\Sigma$ of $\overline{\Omega}$ satisfying
$D\cap\partial\Omega\subseteq\Sigma\cap\partial\Omega$, one has 
\begin{eqnarray}\label{Tgb-11VVV.BIS.UUV}
{\mathfrak{E}}_{D}:W^{k,p}_\Sigma(\Omega)\longrightarrow W^{k,p}_\Sigma({\mathbb{R}}^n)
\quad\mbox{linearly and boundedly},
\end{eqnarray}
(with operator norm controlled in terms of $\varepsilon,\delta,n,p,\Upsilon$,
and the quantitative aspects of \eqref{PPPa.1}-\eqref{PPPa.4}), and
\begin{eqnarray}\label{Pa-PL2VB.BIS.UUV}
\big({\mathfrak{E}}_{D}\,u\big)\big|_{\Omega}=u,\quad
\mbox{${\mathscr{L}}^n$-a.e. on $\Omega$ for every $u\in W^{k,p}_\Sigma(\Omega)$}.
\end{eqnarray}
As a consequence, corresponding to the case when $\Sigma:=D$, one has
\begin{eqnarray}\label{Tgb-11VVV.UUV}
&& {\mathfrak{E}}_{D}:W^{k,p}_D(\Omega)\longrightarrow W^{k,p}_D({\mathbb{R}}^n)
\quad\mbox{linearly and boundedly, and}
\\[4pt]
&& \big({\mathfrak{E}}_{D}\,u\big)\big|_{\Omega}=u,\quad
\mbox{${\mathscr{L}}^n$-a.e. on $\Omega$ for every $u\in W^{k,p}_D(\Omega)$}.
\label{Pa-PL2VB.UUV}
\end{eqnarray}
\end{theorem}

\begin{proof}
Construct ${\mathfrak{E}}_{D}$ as in the proof of Theorem~\ref{cjeg.5AD}
(cf. \eqref{Tgb-ak} and \eqref{YL-uvv.rt2}) replacing in \eqref{Tgb-ak} the 
operators $\Lambda_{k,j}$ by $\widehat{\Lambda}_j$, naturally associated with 
$\Omega_j$ as in Theorem~\ref{cjeg.UUV} for the given $\Upsilon$. Then the same
argument as in the proof of Theorem~\ref{cjeg.5AD} yields all desired conclusions,
by substituting Theorem~\ref{cjeg.UUV} for Theorem~\ref{cjeg}.
\end{proof}

It is also worth noting that the semi-universal extension operator $\widehat{\Lambda}$ 
constructed in Theorem~\ref{cjeg.UUV} in relation to a given 
$(\varepsilon,\delta)$-domain $\Omega$ in ${\mathbb{R}}^n$ with 
${\rm rad}\,(\Omega)>0$ and a given $\Upsilon\in{\mathbb{N}}$, 
has the property that for each closed subset $D$ of $\overline{\Omega}$, 
each $p\in(1,\infty]$, and each $k\in{\mathbb{N}}$ with $k<\Upsilon$, 
\begin{eqnarray}\label{fgu4J.UUV}
\widehat{\Lambda}:W^{k,p}_D(\Omega)\longrightarrow W^{k,p}_D({\mathbb{R}}^n)
\quad\mbox{ linearly and boundedly},
\end{eqnarray}
with operator norm controlled solely in terms of $n,\varepsilon,\delta,p,\Upsilon$.
This can be seen by reasoning as in the proof of Corollary~\ref{cjeg.5}.

\vskip 0.10in

We conclude this section by recording the following useful consequence 
of Theorem~\ref{cjeg.5AD}.

\begin{corollary}\label{cjPPa}
Suppose that $\Omega\subseteq{\mathbb{R}}^n$ and $D\subseteq\overline{\Omega}$ 
are such that $D$ is closed and $\Omega$ is locally an $(\varepsilon,\delta)$-domain near 
$\partial\Omega\setminus D$. Then for any $k\in{\mathbb{N}}$, $p\in[1,\infty]$,
and any closed subset $\Sigma$ of $\overline{\Omega}$ satisfying 
$D\cap\partial\Omega\subseteq\Sigma\cap\partial\Omega$, there holds 
\begin{eqnarray}\label{kaYabYHH.Sd}
W^{k,p}_\Sigma(\Omega)=\Big\{u\big|_{\Omega}:\,u\in W^{k,p}_\Sigma({\mathbb{R}}^n)\Big\}.
\end{eqnarray} 
As a consequence, 
\begin{eqnarray}\label{kaYabYHH}
W^{k,p}_D(\Omega)=\Big\{u\big|_{\Omega}:\,u\in W^{k,p}_D({\mathbb{R}}^n)\Big\}.
\end{eqnarray} 
In particular, formula \eqref{kaYabYHH} holds for any $k\in{\mathbb{N}}$ and
$p\in[1,\infty]$ whenever $\Omega$ is an $(\varepsilon,\delta)$-domain in
${\mathbb{R}}^n$ and $D$ is a closed subset of $\overline{\Omega}$.
\end{corollary}

\begin{proof}
The right-to-left inclusion in \eqref{kaYabYHH.Sd} is implied by $(4)$ in 
Lemma~\ref{TYRD-f4f}, whereas the left-to-right inclusion in \eqref{kaYabYHH.Sd} 
is a consequence of Theorem~\ref{cjeg.5AD}. Finally, \eqref{kaYabYHH} follows by 
specializing \eqref{kaYabYHH.Sd} to the case when $\Sigma=D$.
\end{proof}

\section{The structure of Sobolev spaces with partially vanishing traces}
\label{Sect:4}
\setcounter{equation}{0}

In this section we shall make use of the extension result established in 
Theorem~\ref{cjeg.5AD} in order to further shed 
light on the nature of the spaces introduced in \eqref{uig-adjb-2}. 
To set the stage, we first record some useful capacity results. For an 
authoritative extensive discussion on this topic see the monographs \cite{AH96},
\cite{Ma85}, \cite{Ma11}, and \cite{Zi89}. Given $\alpha>0$ and $p\in(1,\infty)$, 
denote by $C_{\alpha,p}(\cdot)$ the $L^p$-based Bessel capacity of order 
$\alpha$ in ${\mathbb{R}}^n$. When $K\subseteq{\mathbb{R}}^n$ is a compact set, 
this is defined by
\begin{equation}\label{cap-1}
C_{\alpha,p}(K):=\inf\,\Bigl\{\int_{{\mathbb{R}}^n}f^p\,d{\mathscr{L}}^n:\,
f\mbox{ nonnegative, measurable, and }G_\alpha\ast f\geq 1\mbox{ on }K\Bigr\},
\end{equation} 
where the Bessel kernel $G_\alpha$ is defined as the function whose
Fourier transform is given by
\begin{equation}\label{cap-2}
\widehat{G_\alpha}(\xi)=(2\pi)^{-n/2}(1+|\xi|^2)^{-\alpha/2},
\quad\xi\in{\mathbb{R}}^n.
\end{equation} 
When ${\mathcal{O}}\subseteq{\mathbb{R}}^n$ is open, we define
\begin{equation}\label{cap-1X}
C_{\alpha,p}({\mathcal{O}}):=
\sup\,\{C_{\alpha,p}(K):\,K\subseteq{\mathcal{O}},\,K\mbox{ compact}\,\},
\end{equation} 
and, finally, when $E\subseteq{\mathbb{R}}^n$ is an arbitrary set,
\begin{equation}\label{cap-1Y}
C_{\alpha,p}(E):=
\inf\,\{C_{\alpha,p}({\mathcal{O}}):\,{\mathcal{O}}\supseteq E,\,
{\mathcal{O}}\mbox{ open}\,\}.
\end{equation} 
As is customary, generic properties which hold with the possible exception 
of a set $A\subseteq{\mathbb{R}}^n$ satisfying $C_{\alpha,p}(A)=0$ are said to be true 
$(\alpha,p)$-quasieverywhere (or, briefly, $(\alpha,p)$-q.e.).

Given a function $u\in L^q_{loc}({\mathbb{R}}^n,{\mathscr{L}}^n)$ for some 
$q\in[1,\infty)$, denote by 
$\overline{L}_{u,q}\subseteq{\mathbb{R}}^n$ the set of points $x\in{\mathbb{R}}^n$ 
with the property that 
\begin{eqnarray}\label{Paj-YaP}
\overline{u}(x):=\lim\limits_{r\to 0^{+}}\meanint_{B(x,r)}u\,d{\mathscr{L}}^n
\,\,\mbox{ exists }\,\,\mbox{ and }\,\,
\lim\limits_{r\to 0^{+}}\meanint_{B(x,r)}|u-\overline{u}(x)|^q\,d{\mathscr{L}}^n=0.
\end{eqnarray}
It is then clear that for every $u\in L^q_{loc}({\mathbb{R}}^n,{\mathscr{L}}^n)$,
$1\leq q<\infty$, one has
\begin{eqnarray}\label{Paj-YaP.2}
L_{u,q}:=\Big\{x\in{\mathbb{R}}^n:\,\lim\limits_{r\to 0^{+}}
\meanint_{B(x,r)}|u-u(x)|^q\,d{\mathscr{L}}^n=0\Big\}\subseteq\overline{L}_{u,q}.
\end{eqnarray}
Based on this and the classical Lebesgue Differentiation Theorem, it follows that
for every function $u\in L^q_{loc}({\mathbb{R}}^n,{\mathscr{L}}^n)$ with $1\leq q<\infty$,
\begin{eqnarray}\label{Paj-YaP.3}
{\mathscr{L}}^n({\mathbb{R}}^n\setminus L_{u,q})=
{\mathscr{L}}^n({\mathbb{R}}^n\setminus\overline{L}_{u,q})=0\,\,\mbox{ and }\,\,
\overline{u}(x)=u(x)\,\,\mbox{ for every }\,\,x\in L_{u,q}.
\end{eqnarray}

Fix now $p\in(1,\infty)$, $k\in{\mathbb{N}}$, and consider the Sobolev space 
$W^{k,p}({\mathbb{R}}^n)$. If $p>n/k$ then classical embedding results ensure that
in the equivalence class of any $u\in W^{k,p}({\mathbb{R}}^n)$ there is a continuous
representative. It is then immediate from definitions that in this case 
\begin{eqnarray}\label{Paj-Yna.8u}
\overline{L}_{u,q}={\mathbb{R}}^n\,\,\mbox{ for every }\,\,q\in[1,\infty),
\end{eqnarray}
and that $\overline{u}$ (defined in \eqref{Paj-YaP}) 
actually equals the aforementioned continuous representative of $u$
everywhere in ${\mathbb{R}}^n$.

Consider next the case when $p\leq n/k$, in which scenario also fix some 
\begin{eqnarray}\label{Paj-UVb.1}
q\in\big[1,\tfrac{np}{n-kp}\big]\,\,\mbox{ if }\,\,kp<n
\,\,\mbox{ and }\,\,q\in[1,\infty)\,\,\mbox{ if }\,\,kp=n. 
\end{eqnarray}
Then $W^{k,p}({\mathbb{R}}^n)\subseteq L^q_{loc}({\mathbb{R}}^n,{\mathscr{L}}^n)$
and, according to a classical result in potential theory, 
for every $u\in W^{k,p}({\mathbb{R}}^n)$ there exist $A\subseteq{\mathbb{R}}^n$
and some $g\in L^p({\mathbb{R}}^n,{\mathscr{L}}^n)$ such that
\begin{eqnarray}\label{Paj-UUU.1}
&& C_{k,p}(A)=0,\quad{\mathbb{R}}^n\setminus A\subseteq\overline{L}_{u,q},
\quad u=\overline{u}\,\,{\mathscr{L}}^n\mbox{-a.e. in ${\mathbb{R}}^n$},
\\[4pt]
&& u=G_k\ast g\,\,{\mathscr{L}}^n\mbox{-a.e. in ${\mathbb{R}}^n$},\,\,\mbox{ and }\,\,
\overline{u}=G_k\ast g\,\,\mbox{ on }\,\,{\mathbb{R}}^n\setminus A.
\label{Paj-UUU.2}
\end{eqnarray}
See, e.g., \cite[Theorem~6.2.1, p.\,159]{AH96}. Moreover, based on the fact that
$G_k\ast g$ is $(k,p)$-quasicontinuous in the sense of 
\cite[Definition~6.1.1, p.\,156]{AH96} (see \cite[Proposition~6.1.2, p.\,156]{AH96})
it may be easily checked that $\overline{u}$ is also $(k,p)$-quasicontinuous.
This makes $\overline{u}$ a $(k,p)$-quasicontinuous representative of $u$.

\begin{definition}\label{DEF234}
Assume that $p\in(1,\infty)$, $k\in{\mathbb{N}}$, and fix an arbitrary set
$E\subseteq{\mathbb{R}}^n$. In this setting, define the 
operator of restriction to $E$ as the mapping associating to each 
$u\in W^{k,p}({\mathbb{R}}^n)$ the function 
\begin{eqnarray}\label{Piba.P}
{\mathscr{R}}_E u:=
\left\{
\begin{array}{ll}
\overline{u} & \mbox{ on }\,\,E\cap\overline{L}_{u,q},
\\[8pt]
0& \mbox{ on }\,\,E\setminus\overline{L}_{u,q},
\end{array}
\right.
\end{eqnarray}
where $q$ is as in \eqref{Paj-UVb.1}. Moreover, we shall interpret 
${\mathscr{R}}_E u$ as being independent of $q$, at the price of 
regarding this function as being defined only $(k,p)$-q.e. on $E$. 
\end{definition}

It  is useful to note that, in the context of Definition~\ref{DEF234}, 
\begin{eqnarray}\label{PJa-abG}
\big({\mathscr{R}}_E u\big)(x)=
\left\{
\begin{array}{ll}
\lim\limits_{r\to 0^{+}}\dmeanint_{B(x,r)}u\,d{\mathscr{L}}^n
& \mbox{ if }\,x\in E\cap\overline{L}_{u,q},
\\[8pt]
0& \mbox{ if }\,x\in E\setminus\overline{L}_{u,q},
\end{array}
\right.
\qquad\forall\,x\in E.
\end{eqnarray}
In particular, when interpreting ${\mathscr{R}}_E u$ as being independent of $q$
(hence, regarding it as being defined only $(k,p)$-q.e. on $E$), 
the fact that ${\mathscr{R}}_E u=0$ becomes equivalent to 
\begin{eqnarray}\label{PJa-abF}
\lim\limits_{r\to 0^{+}}\meanint_{B(x,r)}u\,d{\mathscr{L}}^n=0\qquad
(k,p)\mbox{-q.e. on }\,\,E.
\end{eqnarray}

We are now in a position to prove, in the case $1\leq p<\infty$, the following 
intrinsic characterization result of the spaces \eqref{uig-adjb-2}, originally 
defined via a completion procedure. Given any $d\in[0,n]$, denote by 
${\mathcal{H}}^d$ the $d$-{\tt dimensional Hausdorff measure} in ${\mathbb{R}}^n$.

\begin{theorem}[Structure Theorem for spaces on subsets 
of ${\mathbb{R}}^n$: Version~1]\label{YTab-YHb}
Let $\Omega$ be an open subset of ${\mathbb{R}}^n$ and let $D$ be a 
closed subset $D$ of $\overline{\Omega}$. In addition, assume that $\Omega$ is either 
the entire ${\mathbb{R}}^n$, or locally an $(\varepsilon,\delta)$-domain 
near $\partial\Omega\setminus D$. Finally, fix an arbitrary $k\in{\mathbb{N}}$.
Then for every $p\in(1,\infty)$ one has 
\begin{eqnarray}\label{kaan655}
&&\hskip -1.00in
W^{k,p}_D(\Omega)=\Big\{u\big|_{\Omega}:\,u\in W^{k,p}({\mathbb{R}}^n)
\,\,\mbox{ and }\,\,
{\mathscr{R}}_D(\partial^\alpha u)=0\quad(k-|\alpha|,p)\mbox{-q.e. on }D,
\nonumber\\[4pt]
&&\hskip 1.90in
\mbox{ for each }\,\,\alpha\in{\mathbb{N}}_0^n\,\,\mbox{ with }\,\,
|\alpha|\leq k-1\Big\},
\end{eqnarray}
whereas corresponding to $p=1$ one has
\begin{eqnarray}\label{kaan655.NNO}
&&\hskip -1.00in
W^{k,1}_D(\Omega)=\Big\{u\big|_{\Omega}:\,u\in W^{k,1}({\mathbb{R}}^n)
\,\,\mbox{ and }\,\,
{\mathscr{R}}_D(\partial^\alpha u)=0\quad{\mathcal{H}}^{\max\,\{0,n-k+|\alpha|\}}
\mbox{-a.e. on }D,
\nonumber\\[4pt]
&&\hskip 1.90in
\mbox{ for each }\,\,\alpha\in{\mathbb{N}}_0^n\,\,\mbox{ with }\,\,
|\alpha|\leq k-1\Big\}.
\end{eqnarray}

As a consequence, if $\Omega$ is an arbitrary nonempty open set in 
${\mathbb{R}}^n$ and $k\in{\mathbb{N}}$, then for any $p\in(1,\infty)$ one has
\begin{eqnarray}\label{kaan655.LPK}
&&\hskip -1.00in
\mathring{W}^{k,p}(\Omega)=\Big\{u\big|_{\Omega}:\,u\in W^{k,p}({\mathbb{R}}^n)
\,\,\mbox{ and }\,\,{\mathscr{R}}_{\partial\Omega}(\partial^\alpha u)=0
\quad(k-|\alpha|,p)\mbox{-q.e. on }{\partial\Omega},
\nonumber\\[4pt]
&&\hskip 1.90in
\mbox{ for each }\,\,\alpha\in{\mathbb{N}}_0^n\,\,\mbox{ with }\,\,
|\alpha|\leq k-1\Big\},
\end{eqnarray}
and, corresponding to $p=1$, 
\begin{eqnarray}\label{kaan655.LPK.NNO}
&&\hskip -1.00in
\mathring{W}^{k,1}(\Omega)=\Big\{u\big|_{\Omega}:\,u\in W^{k,1}({\mathbb{R}}^n)
\,\,\mbox{ and }\,\,{\mathscr{R}}_{\partial\Omega}(\partial^\alpha u)=0
\quad{\mathcal{H}}^{\max\,\{0,n-k+|\alpha|\}}\mbox{-a.e. on }D,
\nonumber\\[4pt]
&&\hskip 1.90in
\mbox{ for each }\,\,\alpha\in{\mathbb{N}}_0^n\,\,\mbox{ with }\,\,
|\alpha|\leq k-1\Big\}.
\end{eqnarray}
\end{theorem}

\begin{proof}
We begin by treating the case $\Omega={\mathbb{R}}^n$. In this scenario,
for each $p\in(1,\infty)$ formula \eqref{kaan655} becomes
\begin{eqnarray}\label{kaan67}
&&\hskip -0.80in
W^{k,p}_D({\mathbb{R}}^n)=\big\{u\in W^{k,p}({\mathbb{R}}^n):\,
{\mathscr{R}}_D(\partial^\alpha u)=0\quad(k-|\alpha|,p)\mbox{-q.e. on }D,
\nonumber\\[4pt]
&&\hskip 1.80in
\mbox{ for each }\,\,\alpha\in{\mathbb{N}}_0^n\,\,\mbox{ with }\,\,
|\alpha|\leq k-1\big\}.
\end{eqnarray}
In turn, this is a consequence of a remarkable result of L.I.~Hedberg and
T.H.~Wolff to the effect that any closed set in ${\mathbb{R}}^n$
admits what has become known as $(k,p)$-synthesis, for any $p\in(1,\infty)$ 
and any $k\in{\mathbb{N}}$. See \cite[Theorem~5, p.\,166]{HW83}, as well as
\cite[Theorem~9.1.3, p.\,234]{AH96}. The end-point case $p=1$ of this result,
namely 
\begin{eqnarray}\label{kaan655.NNO2}
&&\hskip -1.00in
W^{k,1}_D({\mathbb{R}}^n)=\Big\{u\in W^{k,1}({\mathbb{R}}^n):\,
{\mathscr{R}}_D(\partial^\alpha u)=0\quad{\mathcal{H}}^{\max\,\{0,n-k+|\alpha|\}}
\mbox{-a.e. on }D,
\nonumber\\[4pt]
&&\hskip 1.90in
\mbox{ for each }\,\,\alpha\in{\mathbb{N}}_0^n\,\,\mbox{ with }\,\,
|\alpha|\leq k-1\Big\},
\end{eqnarray}
has been obtained by Yu.\,Netrusov in \cite{Net}. 
With \eqref{kaan67}-\eqref{kaan655.NNO2} in hand, \eqref{kaan655}-\eqref{kaan655.NNO}
follow, granted the assumptions made on $\Omega$ and $D$ in the statement of the
theorem, by appealing to Corollary~\ref{cjPPa}. Finally, 
\eqref{kaan655.LPK}-\eqref{kaan655.LPK.NNO} are immediate 
from \eqref{kaan655}-\eqref{kaan655.NNO} with $D=\partial\Omega$ 
and part $(5)$ in Lemma~\ref{TYRD-f4f}.
\end{proof}

Theorem~\ref{YTab-YHb} suggests considering higher-order restriction operators of the
following nature. Assuming that $p\in(1,\infty)$, $k\in{\mathbb{N}}$, and 
$E\subseteq{\mathbb{R}}^n$ is arbitrary, for each $u\in W^{k,p}({\mathbb{R}}^n)$ define 
\begin{eqnarray}\label{Piba}
{\mathscr{R}}^{(k)}_E u:=\big\{{\mathscr{R}}_E(\partial^\alpha u)\big\}
_{|\alpha|\leq k-1}.
\end{eqnarray}
In particular, ${\mathscr{R}}^{(1)}_E={\mathscr{R}}_E$. With this piece of notation,
given a function $u\in W^{k,p}({\mathbb{R}}^n)$, where $p\in(1,\infty)$ and $k\in{\mathbb{N}}$, along with an arbitrary set $E\subseteq{\mathbb{R}}^n$,
we agree to interpret the condition 
\begin{eqnarray}\label{Piba.2}
{\mathscr{R}}^{(k)}_E u=0\,\,\mbox{ quasi-everywhere on }\,\,E
\end{eqnarray}
as an abbreviation of the fact that 
\begin{eqnarray}\label{Piba.3}
{\mathscr{R}}_E(\partial^\alpha u)=0\quad(k-|\alpha|,p)\mbox{-q.e. on }E,
\mbox{ for each }\,\,\alpha\in{\mathbb{N}}_0^n\,\,\mbox{ with }\,\,|\alpha|\leq k-1.
\end{eqnarray}
With this interpretation, in the context of Theorem~\ref{YTab-YHb}, 
formula \eqref{kaan655} becomes 
\begin{eqnarray}\label{kkiii}
W^{k,p}_D(\Omega)=\Big\{u\big|_{\Omega}:\,u\in W^{k,p}({\mathbb{R}}^n)
\,\,\mbox{ and }\,\,{\mathscr{R}}^{(k)}_D u=0\,\,
\mbox{ quasi-everywhere on }\,\,D\Big\}.
\end{eqnarray}

We are now interested in the case in which the vanishing trace condition 
intervening in the right-hand side of \eqref{kkiii} may be reformulated using the 
Hausdorff measure in lieu of Bessel capacities. This requires some preparations. 
To set the stage, recall that a subset $D$ of ${\mathbb{R}}^n$ is said to be 
$d$-{\tt Ahlfors regular} provided there exits some finite constant $C\geq 1$ 
with the property that
\begin{eqnarray}\label{Fq-A.4}
C^{-1}r^{d}\leq {\mathcal{H}}^{d}(B(x,r)\cap D)
\leq Cr^{d},\qquad\forall\,x\in D,\,\,\,0<r\leq{\rm diam}\,(D)
\end{eqnarray}
(where, as before, ${\mathcal{H}}^d$ is the $d$-dimensional Hausdorff measure 
in ${\mathbb{R}}^n$). For example, the boundary of the Koch's snowflake 
in ${\mathbb{R}}^2$ is $d$-Ahlfors regular for $d=\frac{\ln 4}{\ln 3}$.

Assuming that $D\subseteq{\mathbb{R}}^n$ is closed and $d$-Ahlfors regular for 
some $0<d\leq n$, we shall define the measure
\begin{eqnarray}\label{Fq-A.5afv}
\sigma:={\mathcal{H}}^d\lfloor D.
\end{eqnarray}
In this context, a brand of Besov spaces have been introduced by 
A.~Jonsson and H.~Wallin in \cite{JoWa84} as follows. Given $p,q\in[1,\infty]$ and
$s\in(0,\infty)\setminus{\mathbb{N}}$, define the Besov space $B^{p,q}_{s}(D)$ as 
the collection of families $\dot{f}:=\{f_\alpha\}_{|\alpha|\leq [s]}$
(where $[s]$ denotes the integer part of $s$), whose 
components are functions from $L^p(D,\sigma)$, with the property that if for each $\alpha\in{\mathbb{N}}_0^n$ with $|\alpha|\leq[s]$ we set 
\begin{eqnarray}\label{Fq-A.5PL1}
R_\alpha(x,y):=f_{\alpha}(x)-\sum_{|\beta|\leq[s]-|\alpha|}\frac{(x-y)^\beta}{\beta!}
f_{\alpha+\beta}(y)\qquad\mbox{ for $\sigma$-a.e. }\,\,x,y\in D,
\end{eqnarray}
then 
\begin{eqnarray}\label{Fq-A.5PL2}
&&\hskip -0.45in
\|\dot{f}\|_{B^{p,q}_{s}(D)}:=\sum_{|\alpha|\leq[s]}\|f_\alpha\|_{L^p(D,\sigma)}
\\[6pt]
&&\hskip 0.40in
+\left(\sum_{j=0}^\infty \sum_{|\alpha|\leq[s]} 2^{j(s-|\alpha|)q}
\Big(2^{jd}\int\!\!\int_{|x-y|<2^{-j}}|R_\alpha(x,y)|^p
\,d\sigma(x)\,d\sigma(y)\Bigr)^{q/p}\right)^{1/q}<+\infty,
\nonumber
\end{eqnarray}
with a natural interpretation when $\max\,\{p,q\}=\infty$. 
Hereafter, we shall always understand that $B^{p,q}_{s}(D)$ is equipped
with the norm \eqref{Fq-A.5PL2}, in which case this becomes a Banach space.
In fact, a suitable definition of $B^{p,q}_{s}(D)$ may also be given when $s\in{\mathbb{N}}$ as the collection of families 
$\dot{f}=\{f_\alpha\}_{|\alpha|\leq s-1}$ whose components satisfy a certain 
approximation property, though we shall not be needing this case here (this being 
said, the interested reader is refer to \cite[Definition~2, p.\,123]{JoWa84} for
details). 
 
The following theorem is a particular case of more general results, regarding 
traces and extensions on (and from) $d$-Ahlfors regular closed subsets 
of ${\mathbb{R}}^n$ proved by A.\,Jonsson and H.\,Wallin in 
\cite[Theorem~1, p.\,182]{JoWa84}, \cite[Theorem~2, p.\,183]{JoWa84}, 
and \cite[Theorem~3, p.\,197]{JoWa84}.

\begin{theorem}[Jonsson-Wallin trace/extension theory from/into ${\mathbb{R}}^n$]
\label{GCC-67}
Let $D\subseteq{\mathbb{R}}^n$ be a closed set which is $d$-Ahlfors regular
for some $d\in(0,n)$. Also, fix a number $k\in{\mathbb{N}}$ and assume that
\begin{eqnarray}\label{GGG-89J5g}
\max\{1,n-d\}<p<\infty.
\end{eqnarray}

Then for every $v\in W^{k,p}({\mathbb{R}}^n)$ the vector-valued limit 
\begin{eqnarray}\label{THna}
\left\{\lim\limits_{r\to 0^{+}}\meanint_{B(x,r)}
\partial^\alpha v\,d{\mathscr{L}}^n\right\}_{|\alpha|\leq k-1}
\,\,\,\mbox{ exists at ${\mathcal{H}}^d$-a.e. }\,\,x\in D,
\end{eqnarray}
the higher-order trace operator from \eqref{Piba} satisfies
\begin{eqnarray}\label{THna.2}
\big({\mathscr{R}}_D^{(k)}v\big)(x)=\left\{\lim\limits_{r\to 0^{+}}\meanint_{B(x,r)}
\partial^\alpha v\,d{\mathscr{L}}^n\right\}_{|\alpha|\leq k-1}
\,\,\,\mbox{ at ${\mathcal{H}}^d$-a.e. }\,\,x\in D,
\end{eqnarray}
and induces a well-defined, linear, and bounded mapping  
\begin{eqnarray}\label{GGG-89}
{\mathscr{R}}_D^{(k)}:W^{k,p}({\mathbb{R}}^n)
\longrightarrow B^{p,p}_{k-(n-d)/p}(D).
\end{eqnarray}

Conversely, to each $\dot{f}=\{f_\alpha\}_{|\alpha|\leq k-1}\in B^{p,p}_{k-(n-d)/p}(D)$
associate the polynomial 
\begin{eqnarray}\label{GoPLa.1}
P_{\dot{f}}(x,y):=\sum_{|\alpha|\leq k-1}\frac{(x-y)^\alpha}{\alpha!}f_\alpha(y),
\qquad x\in{\mathbb{R}}^n,\,\,\,y\in D,
\end{eqnarray}
and introduce the function ${\mathscr{E}}^{(k)}_D\dot{f}$ 
defined ${\mathscr{L}}^n$-a.e. in ${\mathbb{R}}^n$ (since ${\mathscr{L}}^n(D)=0$
given that $D$ is $d$-Ahlfors regular with $d<n$) according to 
\begin{eqnarray}\label{GoPLa.2}
\big({\mathscr{E}}^{(k)}_D\dot{f}\big)(x):=
\sum\limits_{\stackrel{Q\in{\mathcal{W}}({\mathbb{R}}^n\setminus D)}{\ell(Q)\leq 1}}
\varphi_Q(x)\meanint_{D\cap B(x_Q,6\,{\rm diam}\,(Q))}P_{\dot{f}}(x,y)
\,d{\mathcal{H}}^d(y),\quad\forall\,x\in{\mathbb{R}}^n\setminus D,
\end{eqnarray}
where the family $\big\{\varphi_Q\big\}_{Q\in{\mathcal{W}}({\mathbb{R}}^n\setminus D)}$ consists of functions satisfying 
\begin{eqnarray}\label{GoPLa.3}
\varphi_Q\in{\mathscr{C}}^\infty_c({\mathbb{R}}^n),\quad
{\rm supp}\,\varphi_Q\subseteq\tfrac{17}{16}Q,\quad
0\leq\varphi_Q\leq 1,\quad
\big|\partial^\alpha\varphi_Q\big|\leq C_\alpha\ell(Q)^{-|\alpha|},\quad
\forall\,\alpha\in{\mathbb{N}}_0^n,
\end{eqnarray}
for every $Q\in{\mathcal{W}}({\mathbb{R}}^n\setminus D)$, as well as
\begin{eqnarray}\label{GoPLa.4}
\sum\limits_{\stackrel{Q\in{\mathcal{W}}({\mathbb{R}}^n\setminus D)}{\ell(Q)\leq 1}}
\varphi_Q\equiv 1\qquad\mbox{ on }\,\,\,
\bigcup_{\stackrel{Q\in{\mathcal{W}}({\mathbb{R}}^n\setminus D)}{\ell(Q)\leq 1}}Q.
\end{eqnarray}
Then the operator ${\mathscr{E}}^{(k)}_D$ (whose action depends only on $D$ 
and $k$) has the property that, for each $p$ as in \eqref{GGG-89J5g},
\begin{eqnarray}\label{GGG-90}
{\mathscr{E}}^{(k)}_D:B^{p,p}_{k-(n-d)/p}(D)\longrightarrow W^{k,p}({\mathbb{R}}^n)
\,\,\,\mbox{ linearly and boundedly,}
\end{eqnarray}
and
\begin{eqnarray}\label{GGG-91}
{\mathscr{R}}_D^{(k)}\circ{\mathscr{E}}^{(k)}_D=I,
\,\,\,\mbox{ the identity on }\,\,\,B^{p,p}_{k-(n-d)/p}(D).
\end{eqnarray}
\end{theorem}

Clearly, the existence of a right-inverse makes the higher-order trace 
operator in \eqref{GGG-89} surjective whenever \eqref{GGG-89J5g} holds.
However, one result conspicuously absent from the above theorem 
is an intrinsic description of the null-space of \eqref{GGG-89}.
Our next theorem addresses this aspect.

\begin{theorem}[Characterization of the null-space of the trace operator
acting from ${\mathbb{R}}^n$]\label{YTah-YYHa8}
Suppose that $D\subseteq{\mathbb{R}}^n$ is a closed set which is $d$-Ahlfors regular
for some $d\in(0,n)$, and fix a number $k\in{\mathbb{N}}$. Then 
\begin{eqnarray}\label{GYan-YHN64}
W^{k,p}_D({\mathbb{R}}^n)=
\Big\{u\in W^{k,p}({\mathbb{R}}^n):\,{\mathscr{R}}_D^{(k)}u=0
\,\,\mbox{ at ${\mathcal{H}}^d$-a.e. point on }\,\,D\Big\}
\end{eqnarray}
whenever 
\begin{eqnarray}\label{GGG-89J5gii}
\max\{1,n-d\}<p<\infty.
\end{eqnarray}
\end{theorem}

\begin{proof}
Suppose that $p$ is as in \eqref{GGG-89J5gii}
and assume that $u\in W^{k,p}({\mathbb{R}}^n)$ is such that 
${\mathscr{R}}_D^{(k)}u=0$ at $\sigma$-a.e. point on $D$ or, equivalently,
in $B^{p,p}_{k-(n-d)/p}(D)$. Since ${\mathscr{C}}_c^\infty({\mathbb{R}}^n)$
is dense in $W^{k,p}({\mathbb{R}}^n)$, it is possible to select a sequence
\begin{eqnarray}\label{ayYYH.6a}
\{\varphi_j\}_{j\in{\mathbb{N}}}\subseteq{\mathscr{C}}_c^\infty({\mathbb{R}}^n)
\,\,\mbox{ such that }\,\,\varphi_j\longrightarrow u\,\,
\mbox{ in }\,\,W^{k,p}({\mathbb{R}}^n)\,\,\mbox{ as }\,\,j\to\infty.
\end{eqnarray}
In particular, by our assumptions on $u$ and the continuity of \eqref{GGG-89},
we have 
\begin{eqnarray}\label{ayYYH.6ai}
{\mathscr{R}}_D^{(k)}\varphi_j\to{\mathscr{R}}_D^{(k)}u=0
\,\,\mbox{ in }\,\,B^{p,p}_{k-(n-d)/p}(D)\,\,\mbox{ as }\,\,j\to\infty.
\end{eqnarray}
Going further, for each $j\in{\mathbb{N}}$ define 
\begin{eqnarray}\label{ayYYH.6b}
v_j:=\varphi_j-{\mathscr{E}}^{(k)}_D\big({\mathscr{R}}_D^{(k)}\varphi_j\big)
\,\,\,\mbox{ in }\,\,{\mathbb{R}}^n,
\end{eqnarray}
and note that, in light of \eqref{ayYYH.6b}, \eqref{ayYYH.6a}, and 
Theorem~\ref{GCC-67}, we have 
\begin{eqnarray}\label{aTBbii-U}
v_j\in W^{k,q}({\mathbb{R}}^n)\,\,\,\mbox{ whenever }\,\,\,\max\{1,n-d\}<q<\infty.
\end{eqnarray}
In concert with standard embedding results, this further entails
\begin{eqnarray}\label{aTBbii-U.2}
v_j\in{\mathscr{C}}^{k-1}({\mathbb{R}}^n)\,\,\,\mbox{ for each }\,\,\,j\in{\mathbb{N}}.
\end{eqnarray}
Next, thanks to \eqref{GGG-90}, for every $j\in{\mathbb{N}}$ we may estimate
\begin{eqnarray}\label{ayYYH.6c}
\|u-v_j\|_{W^{k,p}({\mathbb{R}}^n)}
&\leq & \|u-\varphi_j\|_{W^{k,p}({\mathbb{R}}^n)}
+\big\|{\mathscr{E}}^{(k)}_D\big({\mathscr{R}}_D^{(k)}\varphi_j\big)
\big\|_{W^{k,p}({\mathbb{R}}^n)}
\nonumber\\[4pt]
&\leq & \|u-\varphi_j\|_{W^{k,p}({\mathbb{R}}^n)}
+C\big\|{\mathscr{R}}_D^{(k)}\varphi_j\big\|_{B^{p,p}_{k-(n-d)/p}(D)}.
\end{eqnarray}
Consequently,
\begin{eqnarray}\label{ayTTga}
v_j\longrightarrow u\,\,\mbox{ in }\,\,W^{k,p}({\mathbb{R}}^n)\,\,\,
\mbox{ as }\,\,j\to\infty,
\end{eqnarray}
by \eqref{ayYYH.6c}, \eqref{ayYYH.6a}, and \eqref{ayYYH.6ai}.
Moreover, from \eqref{ayYYH.6b} and \eqref{GGG-91} we deduce that for every 
$j\in{\mathbb{N}}$, 
\begin{eqnarray}\label{ayTTga.8}
{\mathscr{R}}_D^{(k)}v_j=0\,\,\,\sigma\mbox{-a.e. on }\,\,D. 
\end{eqnarray}
From \eqref{ayTTga.8}, \eqref{Piba}, \eqref{PJa-abG}, \eqref{aTBbii-U.2},
and \eqref{Fq-A.4} we may now conclude that, for each $j\in{\mathbb{N}}$,
\begin{eqnarray}\label{Yan-Uba77}
(\partial^\alpha v_j)\big|_{D}=0\,\,\mbox{ everywhere on }\,\,D,
\,\,\mbox{ for each }\,\,\alpha\in{\mathbb{N}}_0^n\,\,\mbox{ with }\,\,
|\alpha|\leq k-1.
\end{eqnarray}
When used together with \eqref{kaan67}, \eqref{aTBbii-U} and \eqref{GGG-89J5gii}, 
the everywhere vanishing trace condition from \eqref{Yan-Uba77} implies that 
\begin{eqnarray}\label{kaan67-YYG}
v_j\in W^{k,p}_D({\mathbb{R}}^n)\,\,\,\mbox{ for every }\,\,j\in{\mathbb{N}}.
\end{eqnarray}
Finally, from \eqref{kaan67-YYG}, \eqref{ayTTga} and the fact that 
$W^{k,p}_D({\mathbb{R}}^n)$ is a closed subspace of $W^{k,p}({\mathbb{R}}^n)$
we deduce that 
\begin{eqnarray}\label{PJE-1Uhab}
u\in W^{k,p}_D({\mathbb{R}}^n).
\end{eqnarray}
The membership in \eqref{PJE-1Uhab} proves the right-to-left inclusion 
in \eqref{GYan-YHN64}.

Conversely, if $u\in W^{k,p}_D({\mathbb{R}}^n)$ then there exists
a sequence
\begin{eqnarray}\label{Ja-YjanH}
\begin{array}{c}
\{\varphi_j\}_{j\in{\mathbb{N}}}\subseteq{\mathscr{C}}_c^\infty({\mathbb{R}}^n)
\,\,\mbox{ such that }\,\,D\cap{\rm supp}\,\varphi_j=\emptyset
\,\,\mbox{ for each }\,\,j\in{\mathbb{N}},
\\[6pt]
\mbox{and }\,\,\varphi_j\longrightarrow u\,\,\mbox{ in }\,\,W^{k,p}({\mathbb{R}}^n)
\,\,\mbox{ as }\,\,j\to\infty.
\end{array}
\end{eqnarray}
In particular, from \eqref{Ja-YjanH}, \eqref{Piba} and \eqref{PJa-abG} we see that
\begin{eqnarray}\label{UamBBB.1}
{\mathscr{R}}_D^{(k)}\varphi_j=0\,\,\,\mbox{ everywhere on }\,\,D,
\,\,\,\mbox{ for each }\,\,j\in{\mathbb{N}}.
\end{eqnarray}
Collectively, \eqref{Ja-YjanH}, the continuity of
\eqref{GGG-89}, and \eqref{UamBBB.1} imply that ${\mathscr{R}}_D^{(k)}u=0$ in 
$B^{p,p}_{k-(n-d)/p}(D)$ or, equivalently, at ${\mathcal{H}}^d$-a.e. point on $D$.
This establishes the left-to-right inclusion in \eqref{GYan-YHN64}, and finishes
the proof of the theorem.
\end{proof}

\begin{corollary}\label{Gbab5tg4}
Assume that $D$ is a closed subset of ${\mathbb{R}}^n$ which is $d$-Ahlfors regular 
for some $d\in(0,n)$, and fix $k\in{\mathbb{N}}$ and $p$ such that 
$\max\,\{1,n-d\}<p<\infty$. Then for any function $v\in W^{k,p}({\mathbb{R}}^n)$, 
\begin{eqnarray}\label{Uamiii}
{\mathscr{R}}_D^{(k)}v=0\,\,\mbox{ at ${\mathcal{H}}^d$-a.e. point on }\,\,D
\end{eqnarray}
if and only if 
\begin{eqnarray}\label{Piba.2Bf4tg}
{\mathscr{R}}^{(k)}_D v=0\,\,\mbox{ quasi-everywhere on }\,\,D.
\end{eqnarray}
\end{corollary}

\begin{proof}
This is a immediate consequence of formulas \eqref{kkiii}, \eqref{kaan67}.
\end{proof}

We shall now provide an alternative description of the space $W^{k,p}_D(\Omega)$, 
which should be contrasted to that provided in Theorem~\ref{YTab-YHb}. The new 
feature is is that the vanishing trace condition is now formulated using the 
Hausdorff measure in place of Bessel capacities.

\begin{theorem}[Structure Theorem for spaces on subdomains of 
${\mathbb{R}}^n$: Version~2]\label{YTah-YYHa9}
Suppose that $\Omega\subseteq{\mathbb{R}}^n$ and $D\subseteq\overline{\Omega}$ 
are such that $D$ is closed and $d$-Ahlfors regular for some $d\in(0,n)$, while 
$\Omega$ is locally an $(\varepsilon,\delta)$-domain near $\partial\Omega\setminus D$. 
In addition, fix a number $k\in{\mathbb{N}}$ and assume that $\max\,\{1,n-d\}<p<\infty$. Then 
\begin{eqnarray}\label{GYan-YJaYn}
W^{k,p}_D(\Omega)=\Big\{u\big|_{\Omega}:\,u\in W^{k,p}({\mathbb{R}}^n)
\,\,\mbox{ and }\,\,{\mathscr{R}}_D^{(k)}u=0
\,\,\mbox{ at ${\mathcal{H}}^d$-a.e. point on }\,\,D\Big\}.
\end{eqnarray}

In particular, if $\Omega$ is a nonempty open subset of ${\mathbb{R}}^n$ 
with the property that $\partial\Omega$ is $d$-Ahlfors regular for some 
$d\in(0,n)$, then 
\begin{eqnarray}\label{kaan655.LPK.2}
\mathring{W}^{k,p}(\Omega)=\Big\{u\big|_{\Omega}:\,u\in W^{k,p}({\mathbb{R}}^n)
\,\,\mbox{ and }\,\,{\mathscr{R}}_{\partial\Omega}^{(k)}u=0
\,\,\mbox{ at ${\mathcal{H}}^d$-a.e. point on }\,\,\partial\Omega\Big\},
\end{eqnarray}
whenever $k\in{\mathbb{N}}$ and $\max\,\{1,n-d\}<p<\infty$.
\end{theorem}

\begin{proof}
Formula \eqref{GYan-YJaYn} follows from Theorem~\ref{YTah-YYHa8} in combination with 
Corollary~\ref{cjPPa}, while formula \eqref{kaan655.LPK.2} is a direct consequence 
of \eqref{GYan-YJaYn} and $(5)$ in Lemma~\ref{TYRD-f4f}.
\end{proof}

In the last part of this section we shall revisit the Jonsson-Wallin extension operator 
from Theorem~\ref{GCC-67}, the main goal being establishing the refinement of 
property \eqref{GGG-90} presented in Theorem~\ref{Ohav-7UJ.2}. This requires a number 
of preliminaries to which we now turn.  

Assume that $D\subseteq{\mathbb{R}}^n$ is a given closed set which is $d$-Ahlfors 
regular for some $d\in(0,n)$, and consider $\sigma:={\mathcal{H}}^d\lfloor D$. 
For any $\sigma$-measurable function $f$ on $D$, define 
\begin{eqnarray}\label{PJE-1XV.4R-1}
{\rm supp}\,f:=\big\{x\in D:\,\mbox{there is no $r>0$ such that }\,\,f\equiv 0\,\,
\mbox{ $\sigma$-a.e. in }\,B(x,r)\cap D\big\}.
\end{eqnarray}
In particular, ${\rm supp}\,f$ is a closed subset of $D$ and  
$f$ vanishes $\sigma$-a.e. on $D\setminus{\rm supp}\,f$. If, in addition, 
two numbers $k,p$ are given such that $k\in{\mathbb{N}}$ and $\max\,\{1,n-d\}<p<\infty$, then for every $\dot{f}=\{f_\alpha\}_{|\alpha|\leq k-1}\in B^{p,p}_{k-(n-d)/p}(D)$ 
we define 
\begin{eqnarray}\label{PJE-1XV.4R-2}
{\rm supp}\,\dot{f}:=\bigcup_{|\alpha|\leq k-1}{\rm supp}\,f_\alpha.
\end{eqnarray}
In a first stage, we wish to augment Theorem~\ref{GCC-67} with the following result.

\begin{proposition}\label{Ohav-7UJ}
Let $D\subseteq{\mathbb{R}}^n$ be a closed set which is $d$-Ahlfors 
regular for some $d\in(0,n)$, and define $\sigma:={\mathcal{H}}^d\lfloor D$.
Also, assume that $k\in{\mathbb{N}}$ and $\max\,\{1,n-d\}<p<\infty$.
Then the extension operator ${\mathscr{E}}^{(k)}_D$ from Theorem~\ref{GCC-67}
has the property that 
\begin{eqnarray}\label{PJE-1XV.4R-3}
D\cap{\rm supp}\,\big({\mathscr{E}}^{(k)}_D\dot{f}\big)
={\rm supp}\,\dot{f},\qquad\forall\,\dot{f}\in B^{p,p}_{k-(n-d)/p}(D).
\end{eqnarray}
Furthermore, 
\begin{eqnarray}\label{PJE-1XV.4R-3B}
\dot{f}\in B^{p,p}_{k-(n-d)/p}(D)\,\,\mbox{ with }\,\,
{\rm supp}\,\dot{f}\,\,\mbox{ compact}\,\Longrightarrow\,
{\rm supp}\,\big({\mathscr{E}}^{(k)}_D\dot{f}\big)\,\,\mbox{ compact}.
\end{eqnarray}
\end{proposition}

\begin{proof}
Pick an arbitrary $\dot{f}\in B^{p,p}_{k-(n-d)/p}(D)$ and note that, 
thanks to \eqref{GoPLa.1}, we have 
\begin{eqnarray}\label{PJE-1XV.4R-4.a}
{\rm supp}\,\big(P_{\dot{f}}(x,\cdot)\big)\subseteq
{\rm supp}\,\dot{f},\qquad\forall\,x\in{\mathbb{R}}^n.
\end{eqnarray}
Consequently, for every $x\in{\mathbb{R}}^n\setminus D$, 
formula \eqref{GoPLa.2} may be re-written in the form 
\begin{eqnarray}\label{PJE-1XV.4R-4}
\big({\mathscr{E}}^{(k)}_D\dot{f}\big)(x)=
\sum\limits_{\stackrel{Q\in{\mathcal{W}}({\mathbb{R}}^n\setminus D)}{\ell(Q)\leq 1}}
\frac{\varphi_Q(x)}{{\mathcal{H}}^d\big(D\cap B(x_Q,6\,{\rm diam}\,(Q)\big)}
\int_{B(x_Q,6\,{\rm diam}\,(Q))\cap{\rm supp}\,\dot{f}}P_{\dot{f}}(x,y)
\,d{\mathcal{H}}^d(y).
\end{eqnarray}
From this and the support condition on the $\varphi_Q$'s from \eqref{GoPLa.3}, 
we may then conclude that 
\begin{eqnarray}\label{PJE-23av.DFDF}
{\rm supp}\,\big({\mathscr{E}}^{(k)}_D\dot{f}\big)
\subseteq\overline{G_{\dot{f}}}
\end{eqnarray}
where 
\begin{eqnarray}\label{PJE-5taGB.DFDF}
G_{\dot{f}}:=\bigcup_{\stackrel{Q\in{\mathcal{W}}({\mathbb{R}}^n\setminus D),\,\,
\ell(Q)\leq 1}{B(x_Q,6\,{\rm diam}\,(Q))
\cap{\rm supp}\,\dot{f}\not=\emptyset}}\tfrac{17}{16}Q.
\end{eqnarray}
As such, the left-to-right inclusion in \eqref{PJE-1XV.4R-3} 
follows as soon as we establish that
\begin{eqnarray}\label{PJE-23av.2.DFDF}
D\cap\overline{G_{\dot{f}}}\subseteq{\rm supp}\,\dot{f}.
\end{eqnarray}
To justify \eqref{PJE-23av.2.DFDF}, select an arbitrary point 
$x_o\in D\cap\overline{G_{\dot{f}}}$. The membership of $x_o$ to 
$\overline{G_{\dot{f}}}$ entails the existence of a sequence of dyadic cubes 
$\{Q_j\}_j\subseteq{\mathcal{W}}\big({\mathbb{R}}^n\setminus D\big)$ 
with $\ell(Q_j)\leq 1$ for every $j$, along with a sequence $\{x_j\}_j$ 
of points in ${\mathbb{R}}^n$, satisfying 
\begin{eqnarray}\label{PJE-E.a3.DFDF}
&& x_j\in\tfrac{17}{16}Q_j\,\,\,\mbox{ for every }\,\,j,
\\[4pt]
&& B\big(x_{Q_j},6\,{\rm diam}\,(Q_j)\big)\cap{\rm supp}\,\dot{f}\not=\emptyset
\,\,\,\mbox{ for every }\,\,j,
\label{PJE-E.a3.i.DFDF}
\\[4pt]
&& \lim_{j}x_j=x_o.
\label{PJE-E.a3.ii.DFDF}
\end{eqnarray}
Now, from \eqref{Patah-8n.F} we conclude that there exists $c\in(0,1)$ such that
\begin{eqnarray}\label{Patah-8n.Fi.DFDF}
c\,\ell(Q_j)\leq{\rm dist}\,\big(\tfrac{17}{16}Q_j,D\,\big)
\leq{\rm dist}\,\big(x_j,D\,\big)\leq |x_j-x_o|,\,\,\,\mbox{ for all }\,\,j,
\end{eqnarray}
where the last inequality uses the fact that $x_o\in D$. In concert with 
\eqref{PJE-E.a3.ii.DFDF} and \eqref{PJE-E.a3.DFDF}, this forces
\begin{eqnarray}\label{PJbb.DFDF}
\lim_{j}\ell(Q_j)=0\,\,\,\mbox{ and }\,\,\,\lim_{j}x_{Q_j}=x_o.
\end{eqnarray}
On the other hand, from \eqref{PJE-E.a3.i.DFDF} we deduce that 
for each $j$ there exists $y_j\in{\rm supp}\,\dot{f}$ such that
\begin{eqnarray}\label{PJbb.DFDF.2}
\big|x_{Q_j}-y_j\big|<6\,{\rm diam}\,(Q_j)=6\sqrt{n}\,\ell(Q_j).
\end{eqnarray}
Consequently, 
\begin{eqnarray}\label{PJbb.LLs3.DFDF}
x_o=\lim_{j}x_{Q_j}=\lim_{j}y_j\in{\rm supp}\,\dot{f}
\end{eqnarray}
by \eqref{PJbb.DFDF}, \eqref{PJbb.DFDF.2}, and the fact that 
${\rm supp}\,\dot{f}$ is a closed set. This justifies \eqref{PJE-23av.2.DFDF} and
finishes the proof of the left-to-right inclusion in \eqref{PJE-1XV.4R-3}. 

To proceed in the opposite direction, assume now that $x_o\in D$ is such that 
$x_o\notin{\rm supp}\,\big({\mathscr{E}}^{(k)}_D\dot{f}\big)$. Then there 
exists $r>0$ with the property that ${\mathscr{E}}^{(k)}_D\dot{f}=0$ at 
${\mathscr{L}}^n$-a.e. point in $B(x_o,r)$. As a consequence of this, \eqref{GGG-91}, 
and \eqref{THna.2}, at ${\mathcal{H}}^d$-a.e. $x\in D\cap B(x_o,r)$ we may write 
\begin{eqnarray}\label{PlaYG865gv3}
\dot{f}(x)={\mathscr{R}}^{(k)}_D\big({\mathscr{E}}^{(k)}_D\dot{f}\big)(x)
=\left\{\lim\limits_{\rho\to 0^{+}}\meanint_{B(x,\rho)}
\partial^\alpha\big({\mathscr{E}}^{(k)}_D\dot{f}\big)
\,d{\mathscr{L}}^n\right\}_{|\alpha|\leq k-1}=0.
\end{eqnarray}
Thus, $\dot{f}=0$ at ${\mathcal{H}}^d$-a.e. point on $D\cap B(x_o,r)$, which shows 
that $x_o\notin{\rm supp}\,\dot{f}$, by \eqref{PJE-1XV.4R-1}-\eqref{PJE-1XV.4R-2}. 
Altogether, this argument shows that ${\rm supp}\,\dot{f}\subseteq
D\cap{\rm supp}\,\big({\mathscr{E}}^{(k)}_D\dot{f}\big)$, hence 
the right-to-left inclusion in \eqref{PJE-1XV.4R-3} holds as well. 
This concludes the proof of \eqref{PJE-1XV.4R-3}.

Turning our attention to \eqref{PJE-1XV.4R-3B}, assume that some 
$\dot{f}\in B^{p,p}_{k-(n-d)/p}(D)$ such that ${\rm supp}\,\dot{f}$ is a 
compact set has been fixed. In view of \eqref{PJE-23av.DFDF}, it suffices
to show that $G_{\dot{f}}$ is a bounded set. However, a simple calculation based
on \eqref{PJE-5taGB.DFDF} reveals that 
\begin{eqnarray}\label{PJbb.LLs3.DFDF.6}
G_{\dot{f}}\subseteq\big\{x\in{\mathbb{R}}^n:\,{\rm dist}\,(x,{\rm supp}\,\dot{f})
<\sqrt{n}\big(6+\tfrac{17}{32}\big)\big\},
\end{eqnarray}
hence the desired conclusion follows. 
\end{proof}

In addition to the Besov space naturally associated with the 
quantitative condition \eqref{Fq-A.5PL2}, we shall now bring into focus 
a related (closed) subspace of it, whose distinguished feature is the requirement that 
the Besov functions vanish, in an appropriate sense, on a given subset of 
the (Ahlfors regular) ambient. This class is formally introduced in the 
following definition. 

\begin{definition}\label{DDDv-7UJ.2PD}
Let $D\subseteq{\mathbb{R}}^n$ be a closed set which is $d$-Ahlfors regular for 
some $d\in(0,n)$, and fix some $k\in{\mathbb{N}}$. For each $p$ such that
$\max\,\{1,n-d\}<p<\infty$ and each closed subset $F$ of $D$ define 
\begin{eqnarray}\label{maygbrg57mi}
\begin{array}{c}
B^{p,p}_{k-(n-d)/p}(D;F):=
\,\mbox{the closure in }\,\,B^{p,p}_{k-(n-d)/p}(D)\,\,\mbox{ of the space }\,\,
\\[8pt]
\Big\{\{(\partial^\alpha\varphi)\big|_{D}\big\}_{|\alpha|\leq k-1}:\,
\varphi\in{\mathscr{C}}^\infty_c({\mathbb{R}}^n)\,\,\mbox{ and }\,\,
F\cap\Big(\bigcup\limits_{|\alpha|\leq k-1}
{\rm supp}\,\big(\partial^\alpha\varphi\big|_D\big)\Big)=\emptyset\Big\},
\end{array}
\end{eqnarray}
where each ${\rm supp}\,\big(\partial^\alpha\varphi\big|_D\big)$ is interpreted in 
the sense of \eqref{PJE-1XV.4R-1}, regarding $\partial^\alpha\varphi\big|_D$ 
as a ${\mathcal{H}}^d\lfloor D$-measurable function on $D$.
\end{definition}

\noindent Obviously, in the context of the above definition, 
\begin{eqnarray}\label{aj8trrN.1}
&& B^{p,p}_{k-(n-d)/p}(D;F)\,\,\mbox{ is a closed subspace of }\,\,
B^{p,p}_{k-(n-d)/p}(D), 
\\[4pt]
&& \mbox{the class $B^{p,p}_{k-(n-d)/p}(D;F)$ 
is monotonic with respect to $F$, and} 
\label{aj8trrN.2}
\\[4pt]
&& B^{p,p}_{k-(n-d)/p}(D;D)=\{0\}.
\label{aj8trrN.3}
\end{eqnarray}
Moreover, from Theorem~\ref{GCC-67} it is also clear that 
\begin{eqnarray}\label{aj8trrN.4}
B^{p,p}_{k-(n-d)/p}(D;\emptyset)=B^{p,p}_{k-(n-d)/p}(D).
\end{eqnarray}

The relevance of the category of Besov spaces considered in
Definition~\ref{DDDv-7UJ.2PD} is most apparent in the context of the theorem 
below, refining the Jonsson-Wallin trace/extension results from Theorem~\ref{GCC-67}.

\begin{theorem}[Extending/restricting partially vanishing functions]\label{Ohav-7UJ.2}
Suppose that $D\subseteq{\mathbb{R}}^n$ is a closed set which is $d$-Ahlfors 
regular for some $d\in(0,n)$, and assume that $F$ is a closed subset of $D$.
Also, fix $k\in{\mathbb{N}}$ and some $p$ such that $\max\,\{1,n-d\}<p<\infty$. 

Then the extension and restriction operators, ${\mathscr{E}}^{(k)}_D$,
${\mathscr{R}}^{(k)}_D$, from Theorem~\ref{GCC-67} have the property that 
\begin{eqnarray}\label{rg57miii}
&& {\mathscr{E}}^{(k)}_D:B^{p,p}_{k-(n-d)/p}(D;F)\longrightarrow
W^{k,p}_F({\mathbb{R}}^n),
\\[4pt]
&& {\mathscr{R}}^{(k)}_D:W^{k,p}_F({\mathbb{R}}^n)\longrightarrow 
B^{p,p}_{k-(n-d)/p}(D;F),
\label{rg57miii.888}
\end{eqnarray}
are well-defined, linear and bounded mappings satisfying 
\begin{eqnarray}\label{rg57miii.999}
{\mathscr{R}}^{(k)}_D\circ{\mathscr{E}}^{(k)}_D=I,\quad
\mbox{ the identity on }\,\,B^{p,p}_{k-(n-d)/p}(D;F).
\end{eqnarray}
In particular, the restriction operator in \eqref{rg57miii.888} is onto. 
\end{theorem}

We wish to note that the functional analytic properties of the Jonsson-Wallin 
trace and extension operators recorded in Theorem~\ref{GCC-67} are particular 
manifestations of the above result, as seen by specializing Theorem~\ref{Ohav-7UJ.2}
to the case when $F:=\emptyset$ (cf. \eqref{aj8trrN.4} and \eqref{uig-adjb-2AAA}
in this regard). 

\begin{proof}[Proof of Theorem~\ref{Ohav-7UJ.2}]
Since $B^{p,p}_{k-(n-d)/p}(D;F)\hookrightarrow B^{p,p}_{k-(n-d)/p}(D)$ isometrically,
it follows from Theorem~\ref{GCC-67} that ${\mathscr{E}}^{(k)}_D$ maps 
$B^{p,p}_{k-(n-d)/p}(D;F)$ linearly and boundedly into $W^{k,p}({\mathbb{R}}^n)$.
Given that $W^{k,p}_F({\mathbb{R}}^n)$ is a closed subspace of $W^{k,p}({\mathbb{R}}^n)$ 
(cf. part $(2)$ in Lemma~\ref{TYRD-f4f}), the fact that ${\mathscr{E}}^{(k)}_D$ is 
well-defined, linear and bounded mapping in the context of \eqref{rg57miii} follows
as soon as we prove that 
\begin{eqnarray}\label{rg57miii.2}
{\mathscr{E}}^{(k)}_D\dot{\varphi}\in W^{k,p}_F({\mathbb{R}}^n)
\end{eqnarray}
where $\dot{\varphi}:=\{(\partial^\alpha\varphi)\big|_{D}\big\}_{|\alpha|\leq k-1}$, 
whenever $\varphi\in{\mathscr{C}}^\infty_c({\mathbb{R}}^n)$ is such that
\begin{eqnarray}\label{rg57miii.3}
F\cap\Big(\bigcup\limits_{|\alpha|\leq k-1}
{\rm supp}\,\big(\partial^\alpha\varphi\big|_D\big)\Big)=\emptyset.
\end{eqnarray}
With this goal in mind, fix some $\varphi\in{\mathscr{C}}^\infty_c({\mathbb{R}}^n)$
such that \eqref{rg57miii.3} holds, and note that this condition amounts to 
(cf. \eqref{PJE-1XV.4R-2}) 
\begin{eqnarray}\label{rg57miii.4}
F\cap{\rm supp}\,\dot{\varphi}=\emptyset,
\end{eqnarray}
with $\dot{\varphi}=\{(\partial^\alpha\varphi)\big|_{D}\big\}_{|\alpha|\leq k-1}$
is regarded as an element in $B^{p,p}_{k-(n-d)/p}(D)$. Since 
$F\subseteq D$, Proposition~\ref{Ohav-7UJ} and \eqref{rg57miii.4} imply that 
\begin{eqnarray}\label{rg57miii.5}
F\cap{\rm supp}\,\big({\mathscr{E}}^{(k)}_D\dot{\varphi}\big)
&=& (F\cap D)\cap{\rm supp}\,\big({\mathscr{E}}^{(k)}_D\dot{\varphi}\big)
\nonumber\\[4pt]
&=& F\cap\Big(D\cap{\rm supp}\,\big({\mathscr{E}}^{(k)}_D\dot{\varphi}\big)\Big)
=F\cap{\rm supp}\,\dot{\varphi}=\emptyset.
\end{eqnarray}
Furthermore, given that ${\rm supp}\,\dot{\varphi}$ is compact, 
\eqref{PJE-1XV.4R-3B} ensures that 
${\rm supp}\,\big({\mathscr{E}}^{(k)}_D\dot{\varphi}\big)$ is a compact subset 
of ${\mathbb{R}}^n$. As such, \eqref{rg57miii.5} yields 
${\rm dist}\,\big(F\,,\,{\rm supp}\,
\big({\mathscr{E}}^{(k)}_D\dot{\varphi}\big)\Big)>0$. Having proved this, mollifying 
the function ${\mathscr{E}}^{(k)}_D\dot{\varphi}\in W^{k,p}({\mathbb{R}}^n)$ 
(much as in the past) yields a sequence of 
functions $\{\psi_j\}_{j\in{\mathbb{N}}}\subseteq
{\mathscr{C}}^\infty_c({\mathbb{R}}^n)$ with the property that 
$F\cap{\rm supp}\,\psi_j=\emptyset$ for each $j\in{\mathbb{N}}$ and such that
$\psi_j\to{\mathscr{E}}^{(k)}_D\dot{\varphi}$ in $ W^{k,p}({\mathbb{R}}^n)$ as
$j\to\infty$. In light of \eqref{uig-adjb-2AAA}, we may therefore conclude that 
\eqref{rg57miii.2} holds, finishing the proof of the fact that ${\mathscr{E}}^{(k)}_D$ 
is a well-defined, linear and bounded operator in the context of \eqref{rg57miii}.

As regards \eqref{rg57miii.888}, the starting point is the observation that 
whenever $\varphi\in{\mathscr{C}}^\infty_c({\mathbb{R}}^n)$ is such that
$F\cap{\rm supp}\,\varphi=\emptyset$ then 
\begin{eqnarray}\label{rg57miii.3gNBa}
{\mathscr{R}}^{(k)}_D\varphi=
\{(\partial^\alpha\varphi)\big|_{D}\big\}_{|\alpha|\leq k-1}
\,\,\,\mbox{ and }\,\,
F\cap\Big(\bigcup\limits_{|\alpha|\leq k-1}
{\rm supp}\,\big(\partial^\alpha\varphi\big|_D\big)\Big)=\emptyset.
\end{eqnarray}
In view of \eqref{maygbrg57mi}, this proves that ${\mathscr{R}}^{(k)}_D$ maps 
$\big\{\varphi\in{\mathscr{C}}^\infty_c({\mathbb{R}}^n):\,
F\cap{\rm supp}\,\varphi=\emptyset\big\}$ into $B^{p,p}_{k-(n-d)/p}(D;F)$, 
and since the former space is dense in $W^{k,p}_F({\mathbb{R}}^n)$ which is
mapped by ${\mathscr{R}}^{(k)}_D$ boundedly into $B^{p,p}_{k-(n-d)/p}(D)$
(by Theorem~\ref{GCC-67}), we conclude that the restriction 
operator ${\mathscr{R}}^{(k)}_D$ maps 
$W^{k,p}_F({\mathbb{R}}^n)$ boundedly into the closed subspace 
$B^{p,p}_{k-(n-d)/p}(D;F)$ of $B^{p,p}_{k-(n-d)/p}(D)$.
Hence, ${\mathscr{R}}^{(k)}_D$ is indeed a well-defined, linear and bounded 
operator in the context of \eqref{rg57miii.888}. Finally, \eqref{rg57miii.999}
is a direct consequence of \eqref{GGG-91}, finishing the proof of the theorem. 
\end{proof}

Further information about the version of Besov spaces 
introduced in Definition~\ref{DDDv-7UJ.2PD} is contained in our next result.

\begin{proposition}\label{Oh-UYatt}
Let $D\subseteq{\mathbb{R}}^n$ be a closed set which is $d$-Ahlfors 
regular for some $d\in(0,n)$, and let $F$ be a closed subset of $D$.
Also, fix $k\in{\mathbb{N}}$ and assume that $\max\,\{1,n-d\}<p<\infty$. 
Then 
\begin{eqnarray}\label{anyGVjaYGa}
B^{p,p}_{k-(n-d)/p}(D;F)\hookrightarrow\big\{\dot{f}
\in B^{p,p}_{k-(n-d)/p}(D):\,\dot{f}=0\,\,\mbox{ ${\mathcal{H}}^d$-a.e. on }\,\,F\big\}
\end{eqnarray}
and, moreover, 
\begin{eqnarray}\label{anyGVjaYGa.2}
\begin{array}{c}
B^{p,p}_{k-(n-d)/p}(D;F)\,\,\,\mbox{ coincides with the space}
\\[8pt]
\big\{\dot{f}\in B^{p,p}_{k-(n-d)/p}(D):\,
\dot{f}=0\,\,\mbox{ ${\mathcal{H}}^d$-a.e. on }\,\,F\big\}
\\[8pt]
\mbox{whenever the set $F$ is $d$-Ahlfors regular}.
\end{array}
\end{eqnarray}
\end{proposition}

\begin{proof}
Let $\dot{f}\in B^{p,p}_{k-(n-d)/p}(D;F)$ be arbitrary. Then 
$\dot{f}\in B^{p,p}_{k-(n-d)/p}(D)$ and there exists a sequence 
$\{\varphi_j\}_{j\in{\mathbb{N}}}\subseteq{\mathscr{C}}^\infty_c({\mathbb{R}}^n)$ 
such that if $\dot{\varphi}_j:=
\{(\partial^\alpha\varphi_j)\big|_{D}\big\}_{|\alpha|\leq k-1}$ 
for each $j\in{\mathbb{N}}$ then 
\begin{eqnarray}\label{aUba2auab}
F\cap{\rm supp}\,\dot{\varphi}_j=\emptyset
\,\,\mbox{ for each $j\in{\mathbb{N}}$, and }\,\,
\dot{\varphi}_j\longrightarrow\dot{f}\,\,\mbox{ in $B^{p,p}_{k-(n-d)/p}(D)$ 
as $j\to\infty$}.
\end{eqnarray}
In particular, for each $j\in{\mathbb{N}}$ we have 
$\dot{\varphi}_j=0$ at ${\mathcal{H}}^d$-a.e. point on $F$ which,
in concert with $\lim_{j\to\infty}\dot{\varphi}_j=\dot{f}$ in 
$B^{p,p}_{k-(n-d)/p}(D)$ and the fact that 
$B^{p,p}_{k-(n-d)/p}(D)\hookrightarrow 
L^p\big(D,{\mathcal{H}}^d\lfloor D\big)$ continuously, implies that $\dot{f}=0$ 
at ${\mathcal{H}}^d$-a.e. point on $F\subseteq D$. This proves \eqref{anyGVjaYGa}.

Moving on, make the additional assumption that the set $F$ is $d$-Ahlfors regular,
and consider an arbitrary $\dot{f}\in B^{p,p}_{k-(n-d)/p}(D)$ with the property that
$\dot{f}=0$ at ${\mathcal{H}}^d$-a.e. point on $F$. If we now define 
$u:={\mathscr{E}}^{(k)}_D\dot{f}$, then $u\in W^{k,p}({\mathbb{R}}^n)$ 
and ${\mathscr{R}}^{(k)}_D u=\dot{f}$
at ${\mathcal{H}}^d$-a.e. point on $D$ by Theorem~\ref{GCC-67}.
As a consequence, ${\mathscr{R}}^{(k)}_F u=0$
at ${\mathcal{H}}^d$-a.e. point on $F$, hence $u\in W^{k,p}_F({\mathbb{R}}^n)$
by Theorem~\ref{YTah-YYHa8}, given that $F$ is $d$-Ahlfors regular. In turn, 
the membership of $u$ to $W^{k,p}_F({\mathbb{R}}^n)$ entails, by virtue of 
\eqref{rg57miii.888} that $\dot{f}={\mathscr{R}}^{(k)}_D u$ belongs 
to $B^{p,p}_{k-(n-d)/p}(D;F)$. In concert with \eqref{anyGVjaYGa},
this proves that the claim made in \eqref{anyGVjaYGa.2} holds.
\end{proof}

\section{Trace/Extension theory on locally $(\varepsilon,\delta)$-domains 
onto/from Ahlfors regular subsets}
\label{Sect:5}
\setcounter{equation}{0}

The first goal in this section is to extend the scope of 
Theorems~\ref{GCC-67}-\ref{YTah-YYHa8} by proving results 
similar in spirit but formulated in a domain $\Omega$ in place of the entire
ambient ${\mathbb{R}}^n$. This is done in Theorem~\ref{NIceTRace}, where 
appropriate versions of the Jonsson-Wallin restriction and extension operators
in locally $(\varepsilon,\delta)$-domains are introduced and studied.
To facilitate the reading of this result, the reader is advised to recall the 
version of Besov spaces introduced in Definition~\ref{DDDv-7UJ.2PD}.

\begin{theorem}[Trace/Extension theory on locally $(\varepsilon,\delta)$-domains]\label{NIceTRace}
Assume that $\Omega\subseteq{\mathbb{R}}^n$ and $D\subseteq\overline{\Omega}$
are such that $D$ is closed and $\Omega$ is locally an $(\varepsilon,\delta)$-domain 
near $\partial\Omega\setminus D$. In addition, suppose that $\Sigma$ is a closed 
subset of $\overline{\Omega}$ which is $d$-Ahlfors regular, for some 
$d\in(0,n)$. Finally, set $\sigma:={\mathcal{H}}^d\lfloor\Sigma$ and fix 
$k\in{\mathbb{N}}$ along with some $p$ satisfying $\max\,\{1,n-d\}<p<\infty$. 

Then for every $u\in W^{k,p}_D(\Omega)$ the function given by 
\begin{eqnarray}\label{Ver-S2TG.2-tr}
\big({\mathscr{R}}_{\Omega\to\Sigma}^{(k)}u\big)(x):=
\left\{\lim\limits_{r\to 0^{+}}\meanint_{B(x,r)}
\partial^\alpha v\,d{\mathscr{L}}^n\right\}_{|\alpha|\leq k-1}
\,\,\,\mbox{ at $\sigma$-a.e. }\,\,x\in\Sigma,
\end{eqnarray}
is meaningfully and unambiguously defined whenever 
\begin{eqnarray}\label{Ver-S2TG.2U}
v\in W^{k,p}_D({\mathbb{R}}^n)\,\,\mbox{ is such that }\,\,v\big|_{\Omega}=u
\end{eqnarray}		
(the existence of such functions being guaranteed by \eqref{kaYabYHH} in 
Corollary~\ref{cjPPa}). Moreover, 
\begin{eqnarray}\label{GGG-89.UUU-tr}
{\mathscr{R}}_{\Omega\to\Sigma}^{(k)}:W^{k,p}_D(\Omega)
\longrightarrow\big\{\dot{f}\in B^{p,p}_{k-(n-d)/p}(\Sigma):\,
\dot{f}=0\,\,\mbox{ $\sigma$-a.e. on }\,\,\Sigma\cap D\big\}
\end{eqnarray}
is a well-defined, linear and bounded operator, 
whose range is contained in $B^{p,p}_{k-(n-d)/p}(\Sigma;D\cap\Sigma)$, 
and whose null-space is precisely $W^{k,p}_{D\cup\Sigma}(\Omega)$, i.e., 
\begin{eqnarray}\label{AUan-Taev75.Ls}
\Big\{u\in W^{k,p}_D(\Omega):\,{\mathscr{R}}_{\Omega\to\Sigma}^{(k)}u=0
\,\,\mbox{ at ${\mathcal{H}}^d$-a.e. point on }\,\,\Sigma\Big\}
=W^{k,p}_{D\cup\Sigma}(\Omega).
\end{eqnarray}

Moreover, under the additional assumption that 
\begin{eqnarray}\label{GTg75g.GeE}
D\subseteq\Sigma,
\end{eqnarray}
if one defines
\begin{eqnarray}\label{GTg75g-tr}
{\mathscr{E}}^{(k)}_{\Sigma\to\Omega}\dot{f}:=
\big({\mathscr{E}}^{(k)}_{\Sigma}\dot{f}\,\big)\Big|_{\Omega}
\,\,\,\mbox{ for each }\,\,\dot{f}\in B^{p,p}_{k-(n-d)/p}(\Sigma)
\end{eqnarray}
(where ${\mathscr{E}}^{(k)}_{\Sigma}$ denotes the operator 
${\mathscr{E}}^{(k)}_D$ from Theorem~\ref{GCC-67} corresponding to $D:=\Sigma$),
then 
\begin{eqnarray}\label{GUab9Fa}
{\mathscr{E}}^{(k)}_{\Sigma\to\Omega}:B^{p,p}_{k-(n-d)/p}(\Sigma;D)
\longrightarrow W^{k,p}_D(\Omega)\,\,\,\mbox{ linearly and boundedly},
\end{eqnarray}
and 
\begin{eqnarray}\label{GiG-91Yab-tr}
{\mathscr{R}}_{\Omega\to\Sigma}^{(k)}\circ{\mathscr{E}}^{(k)}_{\Sigma\to\Omega}=I,
\,\,\,\mbox{ the identity on }\,\,\,B^{p,p}_{k-(n-d)/p}(\Sigma;D).
\end{eqnarray}

As a corollary, whenever \eqref{GTg75g.GeE} holds, the operator 
\begin{eqnarray}\label{GGG-89.UUU.re3}
{\mathscr{R}}_{\Omega\to\Sigma}^{(k)}:W^{k,p}_D(\Omega)
\longrightarrow B^{p,p}_{k-(n-d)/p}(\Sigma;D)\,\,\,\mbox{ is surjective}.
\end{eqnarray}
Finally, if actually $D$ is a $d$-Ahlfors regular subset of $\Sigma$, 
then $B^{p,p}_{k-(n-d)/p}(\Sigma;D)$ may be replaced everywhere above 
by $\big\{\dot{f}\in B^{p,p}_{k-(n-d)/p}(\Sigma):\,
\dot{f}=0\,\,\mbox{ $\sigma$-a.e. on }\,\,D\big\}$.
\end{theorem}

\begin{proof}
The fact that for every $v\in W^{k,p}_D({\mathbb{R}}^n)\subseteq 
W^{k,p}({\mathbb{R}}^n)$ the vector-valued 
limit in \eqref{Ver-S2TG.2-tr} exists at $\sigma$-a.e. point in $\Sigma$ is contained
in Theorem~\ref{GCC-67}. Consider now the task of proving that the definition 
of the higher-order restriction operator from \eqref{Ver-S2TG.2-tr} does not 
depend on the extension of $u\in W^{k,p}_D(\Omega)$ to a function $v$ in 
$W^{k,p}_D({\mathbb{R}}^n)$. With this goal in mind assume that 
$u\in W^{k,p}_D(\Omega)$ has been given and suppose that 
$v_1,v_2\in W^{k,p}_D({\mathbb{R}}^n)$ are such that
$v_1\big|_{\Omega}=v_2\big|_{\Omega}=u$. Then the function $v:=v_1-v_2$
satisfies
\begin{eqnarray}\label{Ww-deVVC.1}
v\in W^{k,p}_D({\mathbb{R}}^n)\,\,\,\mbox{ and }\,\,\,
v\big|_{\Omega}=0.
\end{eqnarray}
Since by design $\big\{\varphi\in{\mathscr{C}}_c^\infty({\mathbb{R}}^n):\,
D\cap{\rm supp}\,\varphi=\emptyset\big\}$ is dense in 
$W^{k,p}_D({\mathbb{R}}^n)$, it is possible to select a sequence
\begin{eqnarray}\label{ayYYH.6a.n}
\begin{array}{c}
\{\varphi_j\}_{j\in{\mathbb{N}}}\subseteq{\mathscr{C}}_c^\infty({\mathbb{R}}^n)
\,\,\mbox{ such that }\,\,D\cap{\rm supp}\,\varphi_j=\emptyset
\,\,\mbox{ for every }\,\,j\in{\mathbb{N}},
\\[4pt]
\mbox{and }\,\,\varphi_j\longrightarrow v\,\,
\mbox{ in }\,\,W^{k,p}({\mathbb{R}}^n)\,\,\mbox{ as }\,\,j\to\infty.
\end{array}
\end{eqnarray}
Then, by \eqref{Ww-deVVC.1}-\eqref{ayYYH.6a.n} and part 
$(4)$ in Lemma~\ref{TYRD-f4f}, we have 
\begin{eqnarray}\label{IhabTGGv8rg}
\begin{array}{c}
\varphi_j\big|_{\Omega}\in W^{k,q}_D(\Omega)\,\,\mbox{ for each }\,\,j\in{\mathbb{N}}
\,\,\mbox{ and }\,\,q\in[1,\infty],
\\[4pt]
\mbox{and }\,\,
\varphi_j\big|_{\Omega}\longrightarrow 0\,\,
\mbox{ in }\,\,W^{k,p}(\Omega)\,\,\mbox{ as }\,\,j\to\infty.
\end{array}
\end{eqnarray}
Next, recall the extension operator ${\mathfrak{E}}_{k,D}$ from Theorem~\ref{cjeg.5AD}
(relative to $\Omega$) and, for each $j\in{\mathbb{N}}$, introduce
\begin{eqnarray}\label{ayYYH.6b.m}
w_j:=\varphi_j-{\mathfrak{E}}_{k,D}\big(\varphi_j\big|_{\Omega}\big)
\,\,\,\mbox{ in }\,\,{\mathbb{R}}^n.
\end{eqnarray}
Thanks to \eqref{ayYYH.6b.m}, \eqref{IhabTGGv8rg}, and 
Theorem~\ref{cjeg.5AD}, for each $j\in{\mathbb{N}}$ we have 
\begin{eqnarray}\label{aTBbii-U87}
w_j\in W^{k,q}_D({\mathbb{R}}^n)\hookrightarrow W^{k,q}({\mathbb{R}}^n)
\,\,\,\mbox{ for every }\,\,\,q\in[1,\infty],\,\,\,\mbox{ and }\,\,\,
w_j\big|_{\Omega}=0.
\end{eqnarray}
In concert with standard embedding results, this further entails
\begin{eqnarray}\label{aTBbii-U.2.Jh}
w_j\in{\mathscr{C}}^{k-1}({\mathbb{R}}^n)\,\,\,\mbox{ for each }\,\,\,j\in{\mathbb{N}}.
\end{eqnarray}
Thus, for each $j\in{\mathbb{N}}$ we may compute 
\begin{eqnarray}\label{InabEW964}
\partial^\alpha w_j=0\,\,\mbox{ everywhere on }\,\,\overline{\Omega},
\,\,\mbox{ for all }\,\,\alpha\in{\mathbb{N}}_0^n\,\,\mbox{ with }\,\,
|\alpha|\leq k-1,
\end{eqnarray}
using \eqref{THna.2}, \eqref{InabEW964}, and \eqref{aTBbii-U.2.Jh}.
In particular, given that $\Sigma\subseteq\overline{\Omega}$, 
for each $j\in{\mathbb{N}}$ we have
\begin{eqnarray}\label{Yfav-h3f4br}
\big({\mathscr{R}}_{\Sigma}^{(k)}w_j\big)(x)
&=& \left\{\lim\limits_{r\to 0^{+}}\meanint_{B(x,r)}
\partial^\alpha w_j\,d{\mathscr{L}}^n\right\}_{|\alpha|\leq k-1}
\nonumber\\[4pt]
&=& \left\{(\partial^\alpha w_j)(x)\right\}_{|\alpha|\leq k-1}
\nonumber\\[4pt]
&=& (0,\dots,0)\,\,\,\mbox{ at every }\,\,x\in\Sigma.
\end{eqnarray}
Next, for every $j\in{\mathbb{N}}$ we may estimate
\begin{eqnarray}\label{ayYYH.6c.MB}
\|v-w_j\|_{W^{k,p}({\mathbb{R}}^n)}
&\leq & \|v-\varphi_j\|_{W^{k,p}({\mathbb{R}}^n)}
+\big\|{\mathfrak{E}}_{k,D}\big(\varphi_j\big|_{\Omega}\big)
\big\|_{W^{k,p}({\mathbb{R}}^n)}
\nonumber\\[4pt]
&\leq & \|v-\varphi_j\|_{W^{k,p}({\mathbb{R}}^n)}
+C\big\|\varphi_j\big|_{\Omega}\big\|_{W^{k,p}(\Omega)}.
\end{eqnarray}
Consequently, by \eqref{ayYYH.6c.MB}, \eqref{ayYYH.6a.n}, and \eqref{IhabTGGv8rg}, 
\begin{eqnarray}\label{ayTTga.FD}
w_j\longrightarrow v\,\,\mbox{ in }\,\,W^{k,p}({\mathbb{R}}^n)\,\,\,
\mbox{ as }\,\,j\to\infty,
\end{eqnarray}
hence, further, using the boundedness of \eqref{GGG-89} and \eqref{Yfav-h3f4br},
\begin{eqnarray}\label{Yfav-h3bbb}
{\mathscr{R}}_{\Sigma}^{(k)}v=\lim_{j\to\infty}
{\mathscr{R}}_{\Sigma}^{(k)}w_j=0\,\,\,\mbox{ in }\,\,
B^{p,p}_{k-(n-d)/p}(\Sigma).
\end{eqnarray}
In turn, the fact that ${\mathscr{R}}_{\Sigma}^{(k)}v=0$ in $B^{p,p}_{k-(n-d)/p}(\Sigma)$
forces ${\mathscr{R}}_{\Sigma}^{(k)}v_1={\mathscr{R}}_{\Sigma}^{(k)}v_2$ 
in $B^{p,p}_{k-(n-d)/p}(\Sigma)\hookrightarrow L^p(\Sigma,\sigma)$, hence 
$\sigma$-a.e. on $\Sigma$. In light of \eqref{THna.2}, this implies
\begin{eqnarray}\label{Vkbv06fv}
\left\{\lim\limits_{r\to 0^{+}}\meanint_{B(x,r)}
\partial^\alpha v_1\,d{\mathscr{L}}^n\right\}_{|\alpha|\leq k-1}
=\left\{\lim\limits_{r\to 0^{+}}\meanint_{B(x,r)}
\partial^\alpha v_2\,d{\mathscr{L}}^n\right\}_{|\alpha|\leq k-1}
\,\,\,\mbox{ at $\sigma$-a.e. }\,\,x\in\Sigma.
\end{eqnarray}
This finishes the proof of the fact that the higher-order restriction operator 
\eqref{Ver-S2TG.2-tr}-\eqref{Ver-S2TG.2U} is meaningfully and unambiguously defined
for each $u\in W^{k,p}(\Omega)$. Subsequently, this shows that it is also linear. 

To prove that this operator is bounded in the context of \eqref{GGG-89.UUU-tr}, 
recall the extension operator ${\mathfrak{E}}_{k,D}$ from 
Theorem~\ref{cjeg.5AD} (relative to $\Omega$). Then given any 
$u\in W^{k,p}_D(\Omega)$ we have ${\mathfrak{E}}_{k,D}\,u\in W^{k,p}_D({\mathbb{R}}^n)
\subseteq W^{k,p}_{D\cap\Sigma}({\mathbb{R}}^n)$ and 
$\big({\mathfrak{E}}_{k,D}\,u\big)\big|_{\Omega}=u$. 
Since $D\cap\Sigma$ is a closed subset of the $d$-Ahlfors regular set $\Sigma$, 
these conditions and Theorem~\ref{Ohav-7UJ.2} imply that 
${\mathscr{R}}_{\Omega\to\Sigma}^{(k)}u={\mathscr{R}}_{\Sigma}^{(k)}
\big({\mathfrak{E}}_{k,D}\,u\big)\in B^{p,p}_{k-(n-d)/p}(\Sigma;D\cap\Sigma)$,
plus a naturally accompanying estimate. Granted this, a reference to \eqref{anyGVjaYGa}
then proves that the restriction operator in the context of \eqref{GGG-89.UUU-tr}
is a well-defined, linear and bounded operator, whose range is contained 
in $B^{p,p}_{k-(n-d)/p}(\Sigma;D\cap\Sigma)$.

Turning to the task of justifying the right-to-left inclusion in \eqref{AUan-Taev75.Ls},
consider an arbitrary function $u\in W^{k,p}_{D\cup\Sigma}(\Omega)$. 
From part $(8)$ in Lemma~\ref{TYRD-f4f} we know that $u\in W^{k,p}_D(\Omega)$.
Moreover, there exists a sequence
$\{\varphi_j\}_{j\in{\mathbb{N}}}\subseteq{\mathscr{C}}_c^\infty({\mathbb{R}}^n)$ 
such that 
\begin{eqnarray}\label{aKna.1}
\begin{array}{c}
(D\cup\Sigma)\cap{\rm supp}\,\varphi_j=\emptyset
\,\,\mbox{ for every }\,\,j\in{\mathbb{N}},
\\[4pt]
\mbox{and }\,\,\varphi_j\big|_{\Omega}\longrightarrow u\,\,
\mbox{ in }\,\,W^{k,p}(\Omega)\,\,\mbox{ as }\,\,j\to\infty.
\end{array}
\end{eqnarray}
As a consequence of this and the boundedness of \eqref{GGG-89.UUU-tr}, we have
\begin{eqnarray}\label{aKna.2}
{\mathscr{R}}_{\Omega\to\Sigma}^{(k)}u
=\lim_{j\to\infty}{\mathscr{R}}_{\Omega\to\Sigma}^{(k)}\big(\varphi_j\big|_{\Omega}\big)
\,\,\,\mbox{ in }\,\,B^{p,p}_{k-(n-d)/p}(\Sigma)\hookrightarrow
L^p(\Sigma,\sigma).
\end{eqnarray}
Given that for each $j\in{\mathbb{N}}$ we have
${\mathscr{R}}_{\Omega\to\Sigma}^{(k)}\big(\varphi_j\big|_{\Omega}\big)
=\big\{\partial^\alpha\varphi_j\big\}_{|\alpha|\leq k-1}=(0,...,0)$ everywhere on 
$\Sigma$, it follows that ${\mathscr{R}}_{\Omega\to\Sigma}^{(k)}u=0$ at 
$\sigma$-a.e. point on $\Sigma$. This places $u$ in the left-hand side of 
\eqref{AUan-Taev75.Ls}, as desired. 

Consider next the left-to-right inclusion in \eqref{AUan-Taev75.Ls}. 
In this regard, suppose that $u\in W^{k,p}_D(\Omega)$ is such that
${\mathscr{R}}_{\Omega\to\Sigma}^{(k)}u=0$ at ${\mathcal{H}}^d$-a.e. point 
on $\Sigma$. With ${\mathfrak{E}}_{k,D}$ denoting the extension operator
from Theorem~\ref{cjeg.5AD} relative to $\Omega$, define 
$v:={\mathfrak{E}}_{k,D}\,u$ and note that, thanks to Theorem~\ref{cjeg.5AD},  
\begin{eqnarray}\label{VeKNa}
v\in W^{k,p}_D({\mathbb{R}}^n)\,\,\mbox{ and }\,\,v\big|_{\Omega}=u.
\end{eqnarray}
Furthermore, $\big({\mathscr{R}}_{\Sigma}^{(k)}v\big)(x)
=\big({\mathscr{R}}_{\Omega\to\Sigma}^{(k)}u\big)(x)=0$
at $\sigma$-a.e. $x\in\Sigma$, by \eqref{Ver-S2TG.2-tr}-\eqref{Ver-S2TG.2U}, 
\eqref{THna.2}, and our assumptions on $u$. Based on this and Corollary~\ref{Gbab5tg4},
we may then conclude that, on the one hand,  
\begin{eqnarray}\label{Piba.2BkLL}
{\mathscr{R}}^{(k)}_\Sigma v=0\,\,\mbox{ quasi-everywhere on }\,\,\Sigma.
\end{eqnarray}
On the other hand, the membership of $v$ to $W^{k,p}_D({\mathbb{R}}^n)$ 
entails, in light of \eqref{kaan67}, that
\begin{eqnarray}\label{Piba.2BkLL.2}
{\mathscr{R}}^{(k)}_D v=0\,\,\mbox{ quasi-everywhere on }\,\,D.
\end{eqnarray}
Collectively, \eqref{Piba.2BkLL}-\eqref{Piba.2BkLL.2} now imply that the 
function $v\in W^{k,p}({\mathbb{R}}^n)$ satisfies
\begin{eqnarray}\label{Piba.2BkLL.3}
{\mathscr{R}}^{(k)}_{D\cup\Sigma} v=0\,\,\mbox{ quasi-everywhere on }\,\,D\cup\Sigma.
\end{eqnarray}
As such, $u=v\big|_{\Omega}$ belongs to $W^{k,p}_{D\cup\Sigma}(\Omega)$, 
by \eqref{kaan655}. This finishes the justification of \eqref{AUan-Taev75.Ls}. 

For the remainder of the proof make the additional assumption that 
\eqref{GTg75g.GeE} holds. To proceed, pick an arbitrary 
$\dot{f}\in B^{p,p}_{k-(n-d)/p}(\Sigma;D)$ and set 
$v:={\mathscr{E}}^{(k)}_{\Sigma}\dot{f}$ in ${\mathbb{R}}^n$. 
Then Theorem~\ref{Ohav-7UJ.2} gives 
\begin{eqnarray}\label{Piba.2BkLL.4}
v\in W^{k,p}_D({\mathbb{R}}^n),\quad {\mathscr{R}}^{(k)}_\Sigma v=\dot{f},
\,\,\,\mbox{ and }\,\,\,
\|v\|_{W^{k,p}({\mathbb{R}}^n)}\leq C\|\dot{f}\|_{B^{p,p}_{k-(n-d)/p}(\Sigma)},
\end{eqnarray}
for some finite constant $C>0$ independent of $\dot{f}$. Based on this and part 
$(4)$ in Lemma~\ref{TYRD-f4f} we deduce that $v\big|_{\Omega}\in W^{k,p}_D(\Omega)$ 
and $\big\|v\big|_{\Omega}\big\|_{W^{k,p}(\Omega)}
\leq\|v\|_{W^{k,p}({\mathbb{R}}^n)}$. The above argument shows that 
${\mathscr{E}}^{(k)}_{\Sigma\to\Omega}\dot{f}:=
\big({\mathscr{E}}^{(k)}_{\Sigma}\dot{f}\,\big)\big|_{\Omega}$
belongs to $W^{k,p}_D(\Omega)$ and 
$\big\|{\mathscr{E}}^{(k)}_{\Sigma\to\Omega}\dot{f}\big\|_{W^{k,p}(\Omega)}
\leq C\|\dot{f}\|_{B^{p,p}_{k-(n-d)/p}(\Sigma)}$ for some constant independent 
of $\dot{f}$. Hence, the operator in \eqref{GUab9Fa} is well-defined, linear,
and bounded.

To shows that \eqref{GiG-91Yab-tr} holds, for every 
$\dot{f}\in B^{p,p}_{k-(n-d)/p}(\Sigma;D)$ we write 
\begin{eqnarray}\label{Piba.2BkLL.5}
{\mathscr{R}}_{\Omega\to\Sigma}^{(k)}
\Big({\mathscr{E}}^{(k)}_{\Sigma\to\Omega}\dot{f}\Big)
={\mathscr{R}}_{\Omega\to\Sigma}^{(k)}
\Big(\big({\mathscr{E}}^{(k)}_{\Sigma}\dot{f}\,\big)\big|_{\Omega}\Big)
={\mathscr{R}}_{\Sigma}^{(k)}\Big({\mathscr{E}}^{(k)}_{\Sigma}\dot{f}\Big)
=\dot{f},
\end{eqnarray}
by \eqref{GTg75g-tr}, \eqref{Ver-S2TG.2-tr}-\eqref{Ver-S2TG.2U} and \eqref{rg57miii.999}.

Finally, the claim in \eqref{GGG-89.UUU.re3} is a direct consequence of 
\eqref{GiG-91Yab-tr}, while the very last claim in the statement of the theorem
follows from Proposition~\ref{Oh-UYatt}.
\end{proof}

We now proceed to record several basic consequences of Theorem~\ref{NIceTRace}, 
starting with the following result which provides an intrinsic characterization
of the Sobolev spaces from Definition~\ref{IUha-Tw} considered in
$(\varepsilon,\delta)$-domains. 

\begin{theorem}[Intrinsic description of spaces on domains]\label{Kance.745}
Assume that $\Omega$ is an $(\varepsilon,\delta)$-domain in ${\mathbb{R}}^n$ 
with ${\rm rad}\,(\Omega)>0$, and that $D$ is a closed subset of $\overline{\Omega}$
which is $d$-Ahlfors regular for some $d\in(0,n)$. In addition, fix $k\in{\mathbb{N}}$ 
and suppose that $\max\,\{1,n-d\}<p<\infty$. Then 
\begin{eqnarray}\label{AUan-TLyat6g}
W^{k,p}_D(\Omega)=\Big\{u\in W^{k,p}(\Omega):\,{\mathscr{R}}_{\Omega\to D}^{(k)}u=0
\,\,\mbox{ at ${\mathcal{H}}^d$-a.e. point on }\,\,D\Big\}.
\end{eqnarray}
\end{theorem}

\begin{proof}
This follows from \eqref{AUan-Taev75.Ls}, specialized to the case when 
$D:=\emptyset$ and $\Sigma$ playing the role of the current set $D$, 
in combination with \eqref{Itebn09} and Lemma~\ref{LLDEnse}.
\end{proof}

Another useful application of Theorem~\ref{NIceTRace} is presented in the next
corollary. 

\begin{corollary}\label{Kae.Yab7b6}
Assume that $\Omega\subseteq{\mathbb{R}}^n$ and $D\subseteq\overline{\Omega}$ 
are such that $D$ is closed and $d$-Ahlfors regular for some $d\in(0,n)$, while 
$\Omega$ is locally an $(\varepsilon,\delta)$-domain near $\partial\Omega\setminus D$. 
Then, whenever $k\in{\mathbb{N}}$ and $\max\,\{1,n-d\}<p<\infty$, one has
\begin{eqnarray}\label{amZbYJnab}
{\mathscr{R}}_{\Omega\to D}^{(k)}u=0
\,\,\mbox{ at ${\mathcal{H}}^d$-a.e. point on }\,\,D,
\,\,\mbox{ for each }\,\,\,u\in W^{k,p}_D(\Omega).
\end{eqnarray}
\end{corollary}

\begin{proof}
This is a particular case of \eqref{GGG-89.UUU-tr}, considered here with $\Sigma=D$.
\end{proof}

Finally, it is of independent interest to state the version of 
Theorem~\ref{NIceTRace} corresponding to the case of genuinely
$(\varepsilon,\delta)$-domains, given its potential for applications and 
since in such a scenario the conclusions have a more 
streamlined format. We do so in the corollary below. 

\begin{corollary}\label{NIceTRace.CC}
Let $\Omega$ be an $(\varepsilon,\delta)$-domain in ${\mathbb{R}}^n$ with  
${\rm rad}\,(\Omega)>0$ and such that $\partial\Omega$ is $d$-Ahlfors
regular for some $d\in(0,n)$. Also, fix some $k\in{\mathbb{N}}$ along with $p$ 
satisfying $\max\,\{1,n-d\}<p<\infty$. Then for every $u\in W^{k,p}(\Omega)$ one has
\begin{eqnarray}\label{Ver-S2TG.2}
\big({\mathscr{R}}_{\Omega\to\partial\Omega}^{(k)}u\big)(x)=
\left\{\lim\limits_{r\to 0^{+}}\meanint_{B(x,r)}
\partial^\alpha v\,d{\mathscr{L}}^n\right\}_{|\alpha|\leq k-1}
\,\,\,\mbox{ at ${\mathcal{H}}^d$-a.e. }\,\,x\in\partial\Omega,
\end{eqnarray}
for every function $v\in W^{k,p}({\mathbb{R}}^n)$ satisfying 
$v\big|_{\Omega}=u$. Moreover, 
\begin{eqnarray}\label{GGG-89.UUU}
{\mathscr{R}}_{\Omega\to\partial\Omega}^{(k)}:W^{k,p}(\Omega)
\longrightarrow B^{p,p}_{k-(n-d)/p}(\partial\Omega)
\end{eqnarray}
is a well-defined, linear and bounded operator. Also, if 
\begin{eqnarray}\label{GTg75g}
{\mathscr{E}}^{(k)}_{\partial\Omega\to\Omega}\dot{f}:=
\big({\mathscr{E}}^{(k)}_{\partial\Omega}\dot{f}\,\big)\big|_{\Omega}
\,\,\,\mbox{ for each }\,\,\dot{f}\in B^{p,p}_{k-(n-d)/p}(\partial\Omega),
\end{eqnarray}
(where ${\mathscr{E}}^{(k)}_{\partial\Omega}$ denotes the operator 
${\mathscr{E}}^{(k)}_D$ from Theorem~\ref{GCC-67} corresponding to $D:=\partial\Omega$),
then 
\begin{eqnarray}\label{GUab9Fa.34}
{\mathscr{E}}^{(k)}_{\partial\Omega\to\Omega}:
B^{p,p}_{k-(n-d)/p}(\partial\Omega)\longrightarrow W^{k,p}(\Omega)
\,\,\,\mbox{ linearly and boundedly},
\end{eqnarray}
and 
\begin{eqnarray}\label{GiG-91Yab}
{\mathscr{R}}_{\Omega\to\partial\Omega}^{(k)}
\circ{\mathscr{E}}^{(k)}_{\partial\Omega\to\Omega}=I,
\,\,\,\mbox{ the identity on }\,\,\,B^{p,p}_{k-(n-d)/p}(\partial\Omega).
\end{eqnarray}

Finally, the restriction operator 
${\mathscr{R}}_{\Omega\to\partial\Omega}^{(k)}$ from \eqref{GGG-89.UUU}
is surjective, and its null-space is $\mathring{W}^{k,p}(\Omega)$, i.e., 
\begin{eqnarray}\label{Uan-Taev75}
\Big\{u\in W^{k,p}(\Omega):\,{\mathscr{R}}_{\Omega\to\partial\Omega}^{(k)}u=0
\,\,\mbox{ at ${\mathcal{H}}^d$-a.e. point on }\,\,\partial\Omega\Big\}
=\mathring{W}^{k,p}(\Omega).
\end{eqnarray}
In particular, 
\begin{eqnarray}\label{Ww-dense}
{\mathscr{C}}^\infty_c(\Omega)\hookrightarrow \Bigl\{u\in W^{k,p}(\Omega):\,
{\mathscr{R}}_{\Omega\to\partial\Omega}^{(k)}u=0
\,\,\mbox{ at ${\mathcal{H}}^d$-a.e. point on }\,\,\partial\Omega\Big\}
\,\,\mbox{ densely}.
\end{eqnarray}
\end{corollary}

\begin{proof}
All claims up to, and including, \eqref{Uan-Taev75} are direct consequences
of Theorem~\ref{NIceTRace} specialized to the case when $\Sigma:=\partial\Omega$
and $D:=\emptyset$, keeping in mind \eqref{Itebn09} and $(5)$ in Lemma~\ref{TYRD-f4f}. 
Finally, \eqref{Ww-dense} is immediate from \eqref{Uan-Taev75} and \eqref{azTagbM}. 
\end{proof}

\begin{remark}\label{yagav}
Suppose that $\Omega$ be an $(\varepsilon,\delta)$-domain in 
${\mathbb{R}}^n$ with the property that $\partial\Omega$ is $(n-1)$-Ahlfors regular,
and set $\sigma:={\mathcal{H}}^{n-1}\lfloor\partial\Omega$. Also, 
fix $k\in{\mathbb{N}}$ and $p\in(1,\infty)$. Then for every
$u\in W^{k,p}(\Omega)$ the vector-valued limit 
\begin{eqnarray}\label{Ver-S2TG.3}
\left\{\lim\limits_{r\to 0^{+}}\frac{1}{{\mathscr{L}}^n(\Omega\cap B(x,r))}
\int_{\Omega\cap B(x,r)}\partial^\alpha u\,d{\mathscr{L}}^n\right\}
_{|\alpha|\leq k-1}
\end{eqnarray}
exists and equals $\big({\mathscr{R}}_{\Omega\to\partial\Omega}^{(k)}u\big)(x)$ 
at $\sigma$-a.e. $x\in\partial\Omega$. This is a consequence of 
\cite[Proposition~2, p.\,206]{JoWa84}. In turn, the applicability
of the latter result in the present context is ensured by Theorem~\ref{cjeg}
and the observation that the set $\overline{\Omega}$ is $n$-Ahlfors regular
(as seen from an inspection of the proof of \cite[Lemma~2.3, p.\,77]{Jon81}
which actually reveals that $\Omega$ has the interior corkscrew property,
in the sense of Jerison-Kenig \cite{JeKe82}).
\end{remark}

\begin{remark}
Theorem~\ref{Kance.745} and Corollary~\ref{NIceTRace.CC} deal with the class 
of $(\varepsilon,\delta)$-domains in ${\mathbb{R}}^n$ whose boundaries are 
$d$-Ahlfors regular for some $d\in(0,n)$. While there are many examples
of such domains when $d\in[n-1,n)$ (for example Lipschitz domains, in which
case $d=n-1$, and certain fractal sets like a multi-dimensional 
analogue of the von Koch snowflake, in which case matters can be arranged 
for $d$ to be any desired number in $(n-1,n)$) we wish to note that $(\varepsilon,\delta)$-domains having a $d$-Ahlfors regular boundary 
with $d\in(0,n-1)$ also occur naturally. For example, one may readily
verify that for any affine subspace $H$ of 
${\mathbb{R}}^n$ of dimension $d\leq n-2$ the set $\Omega:={\mathbb{R}}^n\setminus H$
is a $(\varepsilon,\infty)$-domain for some $\varepsilon>0$ whose boundary,
$H$, is $d$-Ahlfors regular (indeed, given any two points 
$x,y\in{\mathbb{R}}^n\setminus H$, the semi-circular path $\gamma$ joining them, 
having $|x-y|$ as diameter, and which is contained in a plane perpendicular 
on the affine variety spanned by $H$ and the line passing through $x,y$, satisfies
\eqref{TYDY-854} for some $c=c(n)>0$).
\end{remark}

In the second part of this section we shall employ the trace/extension theory on 
locally $(\varepsilon,\delta)$-domains onto/from Ahlfors regular subsets developed 
in Theorem~\ref{NIceTRace} and its corollaries in order to derive several important 
properties of the Sobolev spaces with partially vanishing traces considered in 
Definition~\ref{IUha-Tw}. First, we shall use the characterization \eqref{Uan-Taev75} 
as the key ingredient in the proof of the following theorem. 

\begin{theorem}[Hereditary property]\label{Pcsd5}
Let $\Omega$ be an $(\varepsilon,\delta)$-domain in ${\mathbb{R}}^n$ 
with ${\rm rad}\,(\Omega)>0$,
and consider a closed subset $D$ of $\overline\Omega$ which is $d$-Ahlfors regular,
for some $d\in(0,n)$. Then for each $k,m\in{\mathbb{N}}$ and $p$ such that
$\max\,\{1,n-d\}<p<\infty$ one has
\begin{eqnarray}\label{BSgh8.5Ws}
W^{k+m,p}_D(\Omega)=\Bigl\{u\in W^{k+m,p}(\Omega)
\cap W^{k,p}_D(\Omega):\,\partial^\gamma u\in W^{m,p}_D(\Omega),\,\,\,
\forall\,\gamma\in{\mathbb{N}}^n_0,\,\,\,|\gamma|=k\Bigr\}.
\end{eqnarray}
\end{theorem}

\begin{proof}
Let $u\in W^{k+m,p}_D(\Omega)$. Then clearly 
\begin{eqnarray}\label{NDfgb7}
u\in W^{k+m,p}(\Omega)\cap W^{k,p}_D(\Omega). 
\end{eqnarray}
In addition, using the definition of the space $W^{k+m,p}_D(\Omega)$ it follows 
that there exists a sequence of functions $\{\varphi_j\}_{j\in{\mathbb{N}}}
\subseteq{\mathscr{C}}^\infty_c({\mathbb{R}}^n)$ such that 
$D\cap{\rm supp}\,\varphi_j=\emptyset$ for each $j\in{\mathbb{N}}$, 
and $\varphi_j\big|_{\Omega}\rightarrow u$ in $W^{k+m,p}(\Omega)$ as $j\to\infty$. 
In particular, for every $\gamma\in{\mathbb{N}}_0^n$ with $|\gamma|=k$ there holds
\begin{eqnarray}\label{Pbhnk}
(\partial^\gamma\varphi_j)\big|_{\Omega}
=\partial^\gamma\big(\varphi_j\big|_{\Omega}\big)\longrightarrow\partial^\gamma u
\,\,\mbox{ in $W^{m,p}(\Omega)$ as }\,\,j\to\infty.
\end{eqnarray}
Since for each $j\in{\mathbb{N}}$ and each $\gamma\in{\mathbb{N}}_0^n$ 
we have $\partial^\gamma\varphi_j\in{\mathscr{C}}^\infty_c({\mathbb{R}}^n)$ 
and $D\cap{\rm supp}(\partial^\gamma\varphi_j)=\emptyset$, \eqref{Pbhnk} 
and the definition of $W^{m,p}_D(\Omega)$ guarantee that 
\begin{eqnarray}\label{Mdfcs}
\partial^\gamma u\in W^{m,p}_D(\Omega),
\qquad\forall\,\gamma\in{\mathbb{N}}^n_0\,\,\mbox{ such that }\,\,|\gamma|=k.
\end{eqnarray} 
Combining \eqref{Mdfcs} and \eqref{NDfgb7} we obtain that the left-to-right 
inclusion in \eqref{BSgh8.5Ws} holds. Parenthetically, we wish to note that 
this portion of the proof works for any nonempty open subset $\Omega$ of 
${\mathbb{R}}^n$ and any closed set $D\subseteq{\overline{\Omega}}$.

There remains to establish the right-to-left inclusion in \eqref{BSgh8.5Ws},
which makes full use of the assumptions on $\Omega$ and $D$ stipulated in the
statement of the theorem. To this end, pick a function $u$ such that
\begin{eqnarray}\label{Rbnhgf}
u\in W^{m+k,p}(\Omega)\cap W^{k,p}_D(\Omega)\,\,\mbox{ and }
\,\,\partial^\gamma u\in W^{m,p}_D(\Omega),\,\,\,
\forall\,\gamma\in{\mathbb{N}}^n_0,\,\,\,|\gamma|=k.
\end{eqnarray} 
Keeping in mind that $u\in W^{k+m,p}(\Omega)$, it follows from \eqref{Piba} and
\eqref{AUan-TLyat6g} that the membership of $u$ to $W^{k+m,p}_D(\Omega)$ is equivalent to 
\begin{eqnarray}\label{Jrfbv8}
{\mathscr{R}}_{\Omega\to D}^{(1)}
\bigl[\partial^\alpha u\bigr]=0\,\,\,\,{\mathcal{H}}^d\mbox{-a.e. on $D$, }\,\,
\forall\,\alpha\in{\mathbb{N}}^n_0\,\,
\mbox{ such that }\,\,|\alpha|\leq m+k-1.
\end{eqnarray}
With the goal of proving \eqref{Jrfbv8}, first notice that, on the one hand, 
the last condition in \eqref{Rbnhgf} implies (thanks to \eqref{Piba} and \eqref{AUan-TLyat6g}) that 
\begin{eqnarray}\label{Ndc5bp}
{\mathscr{R}}_{\Omega\to D}^{(1)}
\bigl[\partial^\beta(\partial^\gamma u)\bigr]=0
\,\,\,{\mathcal{H}}^d\mbox{-a.e. on $D$, }\,\,\forall\,\beta,
\gamma\in{\mathbb{N}}^n_0\,\,\mbox{ such that }\,\,
|\beta|\leq m-1\,\,\mbox{ and }\,\,|\gamma|=k,
\end{eqnarray}
whereupon
\begin{eqnarray}\label{JKmrn8}
{\mathscr{R}}_{\Omega\to D}^{(1)}\bigl[\partial^\alpha u\bigr]=0
\,\,\,{\mathcal{H}}^d\mbox{-a.e. on $D$, }\,\,
\forall\,\alpha\in{\mathbb{N}}^n_0\,\,\mbox{ such that }
\,\,|\alpha|\in\{k,\dots,m+k-1\}.
\end{eqnarray}
On the other hand, the first condition in \eqref{Rbnhgf} ensures that
$u\in W^{k,p}_D(\Omega)$, and thus, by once again appealing to 
\eqref{Piba} and \eqref{AUan-TLyat6g}, 
\begin{eqnarray}\label{Guyhnml}
{\mathscr{R}}_{\Omega\to D}^{(1)}\bigl[\partial^\alpha u\bigr]=0
\,\,\,{\mathcal{H}}^d\mbox{-a.e. on $D$, }\,\,\forall\,\alpha\in{\mathbb{N}}^n_0
\,\,\mbox{ such that }\,\,|\alpha|\in\{0,\dots,k-1\}.
\end{eqnarray}
Altogether, \eqref{JKmrn8} and \eqref{Guyhnml} prove that \eqref{Jrfbv8} holds, 
as desired. This shows that if $u$ is as in \eqref{Rbnhgf} then 
$u\in W^{k+m,p}_D(\Omega)$. Thus, the right-to-left inclusion in 
\eqref{BSgh8.5Ws} holds as well, finishing the proof of the theorem.
\end{proof}

The following consequence of Theorem~\ref{Pcsd5} answers a question posed to 
us by D.~Arnold (in the more specialized setting of 
Lipschitz domains a solution has been given in \cite{MiMi-HOPDE}). 

\begin{corollary}\label{Uafav976f4}
Suppose $\Omega$ is an $(\varepsilon,\delta)$-domain in ${\mathbb{R}}^n$ with 
${\rm rad}\,(\Omega)>0$ and such that $\partial\Omega$ is $d$-Ahlfors 
regular for some $d\in(0,n)$. Then 
\begin{eqnarray}\label{BSgh8}
{\mathring{W}}^{k+m,p}(\Omega)=\Bigl\{u\in W^{k+m,p}(\Omega)
\cap{\mathring{W}}^{k,p}(\Omega):\,
\partial^\gamma u\in{\mathring{W}}^{m,p}(\Omega),\,\,\,
\forall\,\gamma\in{\mathbb{N}}^n_0,\,\,\,|\gamma|=k\Bigr\},
\end{eqnarray} 
whenever $k,m\in{\mathbb{N}}$ and $\max\,\{1,n-d\}<p<\infty$. 
\end{corollary}

\begin{proof}
Formula \eqref{BSgh8} is a direct consequence 
of \eqref{BSgh8.5Ws} (with $D:=\partial\Omega$) and part $(5)$ in Lemma~\ref{TYRD-f4f}
\end{proof}

Moving on to a different, yet related topic, for every nonempty 
open subset $\Omega$ of ${\mathbb{R}}^n$ and any $k\in{\mathbb{N}}$, 
$p\in[1,\infty]$, let us now introduce the following brand of Sobolev space,
\begin{eqnarray}\label{YJBKL-uvv}
\widetilde{W}^{k,p}(\Omega):=\big\{u\in W^{k,p}(\Omega):\,\widetilde{u}\in
W^{k,p}({\mathbb{R}}^n)\big\},
\end{eqnarray}
where, as in the past, for any function $u$ defined in $\Omega$ we have set 
\begin{eqnarray}\label{YJBKL-uvv.2}
\widetilde{u}:=\left\{
\begin{array}{ll}
u & \mbox{ in }\,\,\Omega,
\\[4pt]
0 & \mbox{ in }\,\,\Omega^c:={\mathbb{R}}^n\setminus\Omega.
\end{array}
\right.
\end{eqnarray}
Also, equip the space $\widetilde{W}^{k,p}(\Omega)$ with the norm 
$\|\cdot\|_{W^{k,p}(\Omega)}$. For the time being, we note the following 
elementary lemma. 

\begin{lemma}\label{HtdYUI}
Suppose that $\Omega$ is an arbitrary nonempty open subset of ${\mathbb{R}}^n$,
and fix $k\in{\mathbb{N}}$ along with $p\in[1,\infty]$. Then 
\begin{eqnarray}\label{YJBKL-uvv.4}
\mathring{W}^{k,p}(\Omega)\hookrightarrow\widetilde{W}^{k,p}(\Omega)
\hookrightarrow W^{k,p}(\Omega)\,\,\,\mbox{ isometrically},
\end{eqnarray}
and 
\begin{eqnarray}\label{YJBKL-uvv.6}
{\mathscr{L}}^n(\partial\Omega)=0\,\Longrightarrow\,
\widetilde{W}^{k,p}(\Omega)=\big\{v\big|_{\Omega}:\,
v\in W^{k,p}({\mathbb{R}}^n)\,\,\mbox{ with }\,\,
{\rm supp}\,v\subseteq\overline{\Omega}\,\big\}.
\end{eqnarray}
Moreover, 
\begin{eqnarray}\label{YJBKL-uvv.6GBB}
\begin{array}{c}
\mathring{W}^{k,p}(\Omega)=\widetilde{W}^{k,p}(\Omega)=\big\{v\big|_{\Omega}:\,
v\in W^{k,p}({\mathbb{R}}^n)\,\,\mbox{ with }\,\,{\rm supp}\,v\subseteq\overline{\Omega}
\big\}
\\[8pt]
\mbox{whenever }\,\,\partial\Omega=\partial\big(\,\overline{\Omega}\,\big)
\,\,\mbox{ and }\,\,p>n.
\end{array}
\end{eqnarray}
\end{lemma}

\begin{proof}
The first inclusion in \eqref{YJBKL-uvv.4} follows from \eqref{YJBKL-uvv.5}
and \eqref{YJBKL-uvv}, whereas the second one is clear from definitions.
With the goal of proving \eqref{YJBKL-uvv.6}, first observe 
that if $u\in\widetilde{W}^{k,p}(\Omega)$ then, by design,  
$\widetilde{u}\in W^{k,p}({\mathbb{R}}^n)$,
$\widetilde{u}\big|_{\Omega}=u$, and ${\rm supp}\,\widetilde{u}
\subseteq\overline{\Omega}$. This shows that  
\begin{eqnarray}\label{YJBKL-uvv.6A}
\mbox{the inclusion }\,\,
\widetilde{W}^{k,p}(\Omega)\subseteq\big\{v\big|_{\Omega}:\,
v\in W^{k,p}({\mathbb{R}}^n)\,\,\mbox{ with }\,\,{\rm supp}\,v\subseteq\overline{\Omega}
\big\}\,\,\mbox{ always holds}.
\end{eqnarray}
For the opposite inclusion it is useful to have
${\mathscr{L}}^n(\partial\Omega)=0$, a condition we now assume. 
In this context, suppose that $v\in W^{k,p}({\mathbb{R}}^n)$ 
satisfies ${\rm supp}\,v\subseteq\overline{\Omega}$ and set 
$u:=v\big|_{\Omega}$. Then clearly $u\in W^{k,p}(\Omega)$ 
and $\widetilde{u}$ coincides with $v$ pointwise ${\mathscr{L}}^n$-a.e.
in ${\mathbb{R}}^n\setminus\partial\Omega$, thus ultimately
${\mathscr{L}}^n$-a.e. in ${\mathbb{R}}^n$, granted the assumption on $\partial\Omega$.
As a consequence, $\widetilde{u}$ also belongs to $W^{k,p}({\mathbb{R}}^n)$,
which puts $u=v\big|_{\Omega}$ in $\widetilde{W}^{k,p}(\Omega)$, as desired.

As far as \eqref{YJBKL-uvv.6GBB} is concerned, assume that 
$\partial\Omega=\partial\big(\,\overline{\Omega}\,\big)$ and $p>n$. 
Select an arbitrary $v\in W^{k,p}({\mathbb{R}}^n)$ with 
${\rm supp}\,v\subseteq\overline{\Omega}$ and note that 
$v\in{\mathscr{C}}^{k-1}({\mathbb{R}}^n)$ by standard embeddings. 
In addition, since $v\equiv 0$ on $\big(\Omega^c\big)^\circ$ it follows that
\begin{eqnarray}\label{YakBvvt6}
(\partial^\alpha v)\big|_{\overline{(\Omega^c)^\circ}}\,=0
\,\,\,\mbox{ everywhere on }\,\,\overline{(\Omega^c)^\circ}
\,\,\mbox{ for every }\,\,\alpha\in{\mathbb{N}}_0^n
\,\,\mbox{ with }\,\,|\alpha|\leq k-1.
\end{eqnarray}
Let us momentarily digress in order to note that the assumption 
$\partial\Omega=\partial\big(\,\overline{\Omega}\,\big)$ forces
\begin{eqnarray}\label{YF-YFT-85}
\big(\,\overline{\Omega}\,\big)^\circ
=\overline{\Omega}\setminus\partial\big(\,\overline{\Omega}\,\big)
=\overline{\Omega}\setminus\partial\Omega=\Omega.
\end{eqnarray}
As such, 
\begin{eqnarray}\label{YF-YFT-86}
\overline{(\Omega^c)^\circ}=\overline{(\,\overline{\Omega}\,)^c}
=\big(\big(\,\overline{\Omega}\,)^\circ\big)^c=\Omega^c
\end{eqnarray}
thus \eqref{YakBvvt6} becomes
\begin{eqnarray}\label{Yan-UIah-yT.y}
(\partial^\alpha v)\big|_{\Omega^c}=0\,\,\,\mbox{ everywhere on }\,\,\Omega^c
\,\,\mbox{ for each }\,\,\alpha\in{\mathbb{N}}_0^n
\,\,\mbox{ with }\,\,|\alpha|\leq k-1.
\end{eqnarray}
In particular (cf. \eqref{Piba.2}), 
\begin{eqnarray}\label{PiKhav}
{\mathscr{R}}^{(k)}_{\Omega^c}\,v=0\,\,\mbox{ quasi-everywhere on }\,\,\Omega^c,
\end{eqnarray}
therefore $v\in W^{k,p}_{\Omega^c}({\mathbb{R}}^n)$ by \eqref{kaan67}. 
With the help of $(9)$ and $(5)$ in Lemma~\ref{TYRD-f4f} we then further deduce 
from this that $v\big|_{\Omega}\in W^{k,p}_{\partial\Omega}(\Omega)
={\mathring{W}}^{k,p}(\Omega)$. All in all, this argument shows that in 
the current setting 
$\big\{v\big|_{\Omega}:\,v\in W^{k,p}({\mathbb{R}}^n)\,\,\mbox{ with }\,\,
{\rm supp}\,v\subseteq\overline{\Omega}\big\}\subseteq\mathring{W}^{k,p}(\Omega)$.
Based on this, the first inclusion in \eqref{YJBKL-uvv.4}, and \eqref{YJBKL-uvv.6A},
it follows that the double equality in \eqref{YJBKL-uvv.6GBB} holds, finishing the
proof of the lemma. 
\end{proof}

The issue of the coincidence of the spaces displayed in \eqref{YJBKL-uvv.6GBB}
for {\it all} values of $p\in(1,\infty)$ is addressed in the theorem below. This is
accomplished under the assumptions that $\Omega\subseteq{\mathbb{R}}^n$ is an open 
set which sits on only one side of its topological boundary and such that the 
interior of its complement is an $(\varepsilon,\delta)$-domain. In particular, 
these conditions are satisfied if $\big(\Omega^c\big)^\circ$ is an NTA domain 
in the sense of \cite{JeKe82}.

\begin{theorem}[Extension of Sobolev functions by zero]\label{Tfgg-75dS}
Let $\Omega$ be a nonempty proper open subset of ${\mathbb{R}}^n$ 
with the property that $\partial\Omega=\partial\big(\,\overline{\Omega}\,\big)$
and such that $\big(\Omega^c\big)^\circ$ is an $(\varepsilon,\delta)$-domain
with ${\rm rad}\big((\Omega^c)^\circ\big)>0$. 
Then for every $k\in{\mathbb{N}}$ and $p\in(1,\infty)$,
\begin{eqnarray}\label{YJBKL-uvv.7}
\widetilde{W}^{k,p}(\Omega)=\mathring{W}^{k,p}(\Omega).
\end{eqnarray}
As a corollary, in the current setting the following properties hold:
\begin{eqnarray}\label{YJBKL-uvv.3}
&& \Big(\widetilde{W}^{k,p}(\Omega)\,,\,\|\cdot\|_{W^{k,p}(\Omega)}\Big)
\,\,\,\mbox{ is a Banach space},
\\[4pt]
&& \widetilde{W}^{k,p}(\Omega)\ni u\longmapsto\widetilde{u}
\in W^{k,p}({\mathbb{R}}^n)\,\,\,\mbox{ isometrically, and}
\label{YJBKL-uvv.5.ii}
\\[4pt]
&& \widetilde{W}^{k,p}(\Omega)=\big\{v\big|_{\Omega}:\,
v\in W^{k,p}({\mathbb{R}}^n)\,\,\mbox{ with }\,\,
{\rm supp}\,v\subseteq\overline{\Omega}\big\}.
\label{YJBKL-uvv.6.iii}
\end{eqnarray}
\end{theorem}

\begin{proof}
The inclusion $\mathring{W}^{k,p}(\Omega)\subseteq\widetilde{W}^{k,p}(\Omega)$
is contained in \eqref{YJBKL-uvv.4}, so the crux of the matter is establishing
the opposite one. To this end, let $u\in\widetilde{W}^{k,p}(\Omega)$ be an 
arbitrary function. Then $\widetilde{u}\in W^{k,p}({\mathbb{R}}^n)$ and, hence, 
there exits a sequence
$\{v_j\}_{j\in{\mathbb{N}}}\subseteq{\mathscr{C}}^\infty_c({\mathbb{R}}^n)$
with the property that 
\begin{eqnarray}\label{YJBUHa0TG}
v_j\longrightarrow\widetilde{u}\,\,\mbox{ in }\,\,W^{k,p}({\mathbb{R}}^n)
\,\,\mbox{ as }\,\,j\to\infty. 
\end{eqnarray}
In particular, 
\begin{eqnarray}\label{YJBUHa0TG.222}
v_j\big|_{(\Omega^c)^\circ}\longrightarrow
\widetilde{u}\big|_{(\Omega^c)^\circ}=0\,\,\mbox{ in }\,\,
W^{k,p}\big((\Omega^c)^\circ\big)\,\,\mbox{ as }\,\,j\to\infty. 
\end{eqnarray}
To proceed, for each $j\in{\mathbb{N}}$ consider 
\begin{eqnarray}\label{YJBKL-uvv.8}
w_j:=v_j-\Lambda_k^c\Big(v_j\big|_{(\Omega^c)^\circ}\Big)\,\,\,\mbox{ in }\,\,
{\mathbb{R}}^n,
\end{eqnarray}
where $\Lambda_k^c$ denotes Jones' extension operator for the 
$(\varepsilon,\delta)$-domain $\big(\Omega^c\big)^\circ$, which, 
by assumption, satisfies ${\rm rad}\big((\Omega^c)^\circ\big)>0$.
Based on \eqref{YJBKL-uvv.8} and Theorem~\ref{cjeg}, for each $j\in{\mathbb{N}}$ 
we then have 
\begin{eqnarray}\label{aTGBOUH-9}
w_j\in W^{k,q}({\mathbb{R}}^n)\,\,\,\mbox{ for each }\,\,\,q\in[1,\infty].
\end{eqnarray}
Together with standard embedding results, this implies
\begin{eqnarray}\label{aTB-Tgab-7yG}
w_j\in{\mathscr{C}}^{k-1}({\mathbb{R}}^n)\,\,\,\mbox{ for each }\,\,\,j\in{\mathbb{N}}.
\end{eqnarray}
Furthermore, in light of \eqref{YJBKL-uvv.8} and Theorem~\ref{cjeg}, 
for every $j\in{\mathbb{N}}$ we may estimate
\begin{eqnarray}\label{ayYYH.6cTTT}
\big\|\widetilde{u}-w_j\big\|_{W^{k,p}({\mathbb{R}}^n)}
&\leq & \big\|\widetilde{u}-v_j\big\|_{W^{k,p}({\mathbb{R}}^n)}
+\Big\|\Lambda^c_k\Big(v_j\big|_{(\Omega^c)^\circ}\Big)\Big\|_{W^{k,p}({\mathbb{R}}^n)}
\nonumber\\[4pt]
&\leq & \big\|\widetilde{u}-v_j\big\|_{W^{k,p}({\mathbb{R}}^n)}
+C\big\|v_j\big|_{(\Omega^c)^\circ}\big\|_{W^{k,p}((\Omega^c)^\circ)}.
\end{eqnarray}
In turn, this forces
\begin{eqnarray}\label{ayTTga-Tgab}
w_j\longrightarrow\widetilde{u}\,\,\mbox{ in }\,\,W^{k,p}({\mathbb{R}}^n)\,\,\,
\mbox{ as }\,\,j\to\infty,
\end{eqnarray}
by \eqref{YJBUHa0TG} and \eqref{YJBUHa0TG.222}. In addition, from 
\eqref{YJBKL-uvv.8} and the analogue of \eqref{Pa-PL2} for the 
domain $\big(\Omega^c\big)^\circ$, we conclude that for every 
$j\in{\mathbb{N}}$ we have 
\begin{eqnarray}\label{ayTTga.8UGva}
w_j\big|_{(\Omega^c)^\circ}=0\,\,\,\,{\mathcal{L}}^n\mbox{-a.e. on }\,\,\,
(\Omega^c)^\circ. 
\end{eqnarray}
From \eqref{ayTTga.8UGva} and \eqref{aTB-Tgab-7yG} we may now conclude that, 
for each $j\in{\mathbb{N}}$,
\begin{eqnarray}\label{Yan-UIah-yT.i}
(\partial^\alpha w_j)\big|_{\overline{(\Omega^c)^\circ}}\,=0
\,\,\,\mbox{ everywhere on }\,\,\overline{(\Omega^c)^\circ}
\,\,\mbox{ for every }\,\,\alpha\in{\mathbb{N}}_0^n
\,\,\mbox{ with }\,\,|\alpha|\leq k-1.
\end{eqnarray}
Based on this and the fact that $\partial\Omega=\partial\big(\,\overline{\Omega}\,\big)$
implies \eqref{YF-YFT-86}, we may ultimately conclude that 
\begin{eqnarray}\label{Yan-UIah-yT}
(\partial^\alpha w_j)\big|_{\Omega^c}=0\,\,\,\mbox{ everywhere on }\,\,\Omega^c
\,\,\mbox{ for each }\,\,\alpha\in{\mathbb{N}}_0^n
\,\,\mbox{ with }\,\,|\alpha|\leq k-1.
\end{eqnarray}
When combined with \eqref{kaan67} and \eqref{aTGBOUH-9}, the everywhere vanishing trace 
condition from \eqref{Yan-UIah-yT} implies that 
\begin{eqnarray}\label{kaan67-YYG6ahN}
w_j\in W^{k,p}_{\Omega^c}({\mathbb{R}}^n)\,\,\,\mbox{ for every }\,\,j\in{\mathbb{N}}.
\end{eqnarray}
Finally, from \eqref{kaan67-YYG6ahN}, \eqref{ayTTga-Tgab} and the fact that 
$W^{k,p}_{\Omega^c}({\mathbb{R}}^n)$ is a closed subspace of $W^{k,p}({\mathbb{R}}^n)$
we deduce that 
\begin{eqnarray}\label{PIOHH88}
\widetilde{u}\in W^{k,p}_{\Omega^c}({\mathbb{R}}^n).
\end{eqnarray}
From this and the definition of $W^{k,p}_{\Omega^c}({\mathbb{R}}^n)$ it follows
that there exists a sequence $\{\varphi_j\}_{j\in{\mathbb{N}}}\subseteq
{\mathscr{C}}^\infty_c(\Omega)$ with the property that 
\begin{eqnarray}\label{IanbYG}
\widetilde{\varphi_j}
\longrightarrow\widetilde{u}\,\,\mbox{ in }\,\,W^{k,p}({\mathbb{R}}^n)
\,\,\mbox{ as }\,\,j\to\infty. 
\end{eqnarray}
Consequently,
\begin{eqnarray}\label{POJna0JB}
\varphi_j=\widetilde{\varphi_j}\big|_{\Omega}\longrightarrow
\widetilde{u}\big|_{\Omega}=u\,\,\mbox{ in }\,\,
W^{k,p}(\Omega)\,\,\mbox{ as }\,\,j\to\infty,
\end{eqnarray}
hence, further, $u\in\mathring{W}^{k,p}(\Omega)$ by \eqref{azTagbM}.
Since the function $u\in\widetilde{W}^{k,p}(\Omega)$ has been arbitrarily chosen, 
this shows that $\widetilde{W}^{k,p}(\Omega)\subseteq\mathring{W}^{k,p}(\Omega)$
and finishes the proof of \eqref{YJBKL-uvv.7}.

Moving on, \eqref{YJBKL-uvv.3} is a direct consequence of \eqref{YJBKL-uvv.7}
and \eqref{azTagbM}, whereas \eqref{YJBKL-uvv.5.ii} is immediate from 
\eqref{YJBKL-uvv.7} and \eqref{YJBKL-uvv.5}. Finally, as regards 
\eqref{YJBKL-uvv.6.iii}, first note that since $\big(\Omega^c\big)^\circ$ is 
assumed to be an $(\varepsilon,\delta)$-domain, \eqref{PJE-wcX} implies
that ${\mathscr{L}}^n\Big(\partial\big(\big(\Omega^c\big)^\circ\big)\Big)=0$.
On the other hand, 
\begin{eqnarray}\label{PiahhIYT}
\partial\big(\big(\Omega^c\big)^\circ\big)
=\partial\big(\big(\,\overline{\Omega}\,\big)^c\big)
=\partial\big(\,\overline{\Omega}\,\big)=\partial\Omega. 
\end{eqnarray}
Hence, ${\mathscr{L}}^n(\partial\Omega)=0$, so \eqref{YJBKL-uvv.6.iii} now follows
from \eqref{YJBKL-uvv.6}.
\end{proof}

We now propose to study the issue as to whether Sobolev functions defined
on either side of the boundary of a domain may be ``glued" together with 
preservation of smoothness. In order to introduce a natural geometrical 
setting for this type of question we make the following definition. 

\begin{definition}\label{TYD-Ubbb6B}
Call an open, nonempty, proper subset $\Omega$ of ${\mathbb{R}}^n$ a 
{\tt two-sided} $(\varepsilon,\delta)$-{\tt domain} provided
both $\Omega$ and $\big(\Omega^c\big)^\circ$ are $(\varepsilon,\delta)$-domains, 
${\rm rad}\,(\Omega)>0$, ${\rm rad}\big((\Omega^c)^\circ\big)>0$,
and $\partial\Omega=\partial\big(\,\overline{\Omega}\,\big)$.
\end{definition}

\noindent Note that, as seen from \eqref{PiahhIYT}, 
\begin{eqnarray}\label{PiahhIYT.2}
\mbox{any two-sided $(\varepsilon,\delta)$-domain $\Omega$ satisfies }\,\,
\partial\big(\big(\Omega^c\big)^\circ\big)=\partial\Omega. 
\end{eqnarray}
Examples of two-sided $(\varepsilon,\delta)$-domains include 
the class of bounded Lipschitz domains and, more generally, the class of 
two-sided NTA domains (by which we mean connected sets which are NTA and 
whose interior of their complement is also connected and NTA). 

Theorem~\ref{SSSg-g5dS} below states that gluing Sobolev functions 
defined inside and outside of a two-sided $(\varepsilon,\delta)$-domain
preserves Sobolev smoothness if and only if the functions in question 
have matching traces across the boundary. To facilitate the reading of its 
statement, the reader is advised to recall the definition and properties
of the higher-order restriction map ${\mathscr{R}}^{(k)}_{\Omega\to\partial\Omega}$ 
from Corollary~\ref{NIceTRace.CC}.

\begin{theorem}[Gluing Sobolev functions with matching traces]\label{SSSg-g5dS}
Let $\Omega$ be a two-sided $(\varepsilon,\delta)$-domain in ${\mathbb{R}}^n$ 
with the property that $\partial\Omega$ is $d$-Ahlfors regular for some 
$d\in[n-1,n)$. Also, fix $k\in{\mathbb{N}}$ along with some $p$ such that
$\max\,\{1,n-d\}<p<\infty$.
Then for any $u\in W^{k,p}(\Omega)$ and $v\in W^{k,p}\big((\Omega^c)^\circ\big)$ 
the following two conditions are equivalent:
\begin{enumerate}
\item[(i)] the functions $u,v$ have matching traces, i.e., 
\begin{eqnarray}\label{7reer}
{\mathscr{R}}^{(k)}_{\Omega\to\partial\Omega}\,u
={\mathscr{R}}^{(k)}_{(\Omega^c)^\circ\to\partial\Omega}\,v
\,\,\,\mbox{ at ${\mathcal{H}}^{d}$-a.e. point on }\,\,\partial\Omega;
\end{eqnarray}
\item[(ii)] the function 
\begin{eqnarray}\label{7reer.2}
w:=\left\{
\begin{array}{ll}
u &\mbox{ in }\,\Omega,
\\[4pt]
v &\mbox{ in }\,(\Omega^c)^\circ,
\end{array}
\right.
\end{eqnarray}
(which is defined ${\mathscr{L}}^n$-a.e. in ${\mathbb{R}}^n$) 
belongs to $W^{k,p}({\mathbb{R}}^n)$.
\end{enumerate}

Moreover, whenever condition $(i)$ holds and $w$ is defined as in \eqref{7reer.2}, 
one has
\begin{eqnarray}\label{7reer.3}
\|w\|_{W^{k,p}({\mathbb{R}}^n)}\leq C\Big(
\|u\|_{W^{k,p}(\Omega)}+\|v\|_{W^{k,p}((\Omega^c)^\circ)}\Big),
\end{eqnarray}
where $C>0$ is a finite constant depending only on $n,\varepsilon,\delta,k,p$. 
\end{theorem}

\begin{proof}
For starters, observe that thanks to property \eqref{PiahhIYT.2} it makes 
sense to talk about ${\mathscr{R}}^{(k)}_{(\Omega^c)^\circ\to\partial\Omega}$. 

Consider the implication $(i)\Rightarrow(ii)$. 
In this regard, pick two functions $u\in W^{k,p}(\Omega)$, 
$v\in W^{k,p}\big((\Omega^c)^\circ\big)$ satisfying \eqref{7reer}, and 
consider $w$ as in \eqref{7reer.2}. To begin with, this function is defined 
${\mathscr{L}}^n$-a.e. in 
\begin{eqnarray}\label{7reer.4}
\Omega\cup\big(\Omega^c\big)^\circ
=\Omega\cup\Big(\Omega^c\setminus\partial\big(\big(\Omega^c\big)^\circ\big)\Big)
=\Omega\cup\Big(\Omega^c\setminus\partial\Omega\Big)
={\mathbb{R}}^n\setminus\partial\Omega
\end{eqnarray}
thus, ultimately, ${\mathscr{L}}^n$-a.e. in ${\mathbb{R}}^n$ by \eqref{PJE-wcX}
and the fact that $\Omega$ is an $(\varepsilon,\delta)$-domain in ${\mathbb{R}}^n$.
Going further, define 
\begin{eqnarray}\label{7rYhab}
u_\ast:=\big(\Lambda_k^c v\big)\big|_{\Omega}\,\,\,\mbox{ in }\,\,\Omega,
\end{eqnarray}
where $\Lambda_k^c$ denotes Jones' extension operator for the 
$(\varepsilon,\delta)$-domain $\big(\Omega^c\big)^\circ$. Then 
there exists a finite constant $C>0$ depending only on $n,\varepsilon,\delta,k,p$
such that
\begin{eqnarray}\label{7reer.5}
u_\ast\in W^{k,p}(\Omega)\,\,\,\mbox{ and }\,\,\,
\|u_\ast\|_{W^{k,p}(\Omega)}\leq C\|v\|_{W^{k,p}((\Omega^c)^\circ)},
\end{eqnarray}
thanks to \eqref{PJE-3UYGav} and Theorem~\ref{cjeg}. Furthermore, 
at ${\mathcal{H}}^{d}$-a.e. $x\in\partial\Omega
=\partial\big(\big(\Omega^c\big)^\circ\big)$ we have 
\begin{eqnarray}\label{7reer.6}
\big({\mathscr{R}}^{(k)}_{\Omega\to\partial\Omega}u_\ast\big)(x)
&=& \left\{\lim\limits_{r\to 0^{+}}\meanint_{B(x,r)}
\partial^\alpha\big(\Lambda_k^c v\big)\,d{\mathscr{L}}^n\right\}_{|\alpha|\leq k-1}
\nonumber\\[4pt]
&=& \Big({\mathscr{R}}^{(k)}_{(\Omega^c)^\circ\to\partial\Omega}
\Big(\big(\Lambda_k^c v\big)\big|_{(\Omega^c)^\circ}\Big)\Big)(x)
\nonumber\\[4pt]
&=& \big({\mathscr{R}}^{(k)}_{(\Omega^c)^\circ\to\partial\Omega}v\big)(x)
\nonumber\\[4pt]
&=& \big({\mathscr{R}}^{(k)}_{\Omega\to\partial\Omega}u\big)(x),
\end{eqnarray}
by \eqref{Ver-S2TG.2} in Corollary~\ref{NIceTRace.CC} (used twice),
the analogue of \eqref{Pa-PL2} for $\big(\Omega^c\big)^\circ$, and \eqref{7reer}.
As a result, the function $u-u_\ast\in W^{k,p}(\Omega)$ satisfies
${\mathscr{R}}^{(k)}_{\Omega\to\partial\Omega}(u-u_\ast)=0$ 
at ${\mathcal{H}}^{d}$-a.e. point in $\partial\Omega$. Hence, 
$u-u_\ast\in\mathring{W}^{k,p}(\Omega)$ by \eqref{Uan-Taev75}.
Consequently, from this, \eqref{YJBKL-uvv.7} and \eqref{7reer.5} we deduce that
\begin{eqnarray}\label{7reer.7}
\widetilde{u-u_\ast}\in W^{k,p}({\mathbb{R}}^n)\,\,\,\mbox{ and }\,\,\,
\big\|\widetilde{u-u_\ast}\big\|_{W^{k,p}({\mathbb{R}}^n)}
\leq C\Big(\|u\|_{W^{k,p}(\Omega)}+\|v\|_{W^{k,p}((\Omega^c)^\circ)}\Big),
\end{eqnarray}
where $C>0$ is a finite constant depending only on $n,\varepsilon,\delta,k,p$.
Thus, if we now introduce
\begin{eqnarray}\label{IahnbF}
w_\ast:=\widetilde{u-u_\ast}+\Lambda_k^c v\,\,\,\mbox{ in }\,\,{\mathbb{R}}^n,
\end{eqnarray}
it follows from \eqref{IahnbF}, \eqref{7reer.7} and Theorem~\ref{cjeg} that
\begin{eqnarray}\label{IahnbF.2}
w_\ast\in W^{k,p}({\mathbb{R}}^n)\,\,\,\mbox{ and }\,\,\,
\|w_\ast\|_{W^{k,p}({\mathbb{R}}^n)}
\leq C\Big(\|u\|_{W^{k,p}(\Omega)}+\|v\|_{W^{k,p}((\Omega^c)^\circ)}\Big),
\end{eqnarray}
where $C>0$ is a finite constant depending only on $n,\varepsilon,\delta,k,p$.
Moreover, from \eqref{IahnbF} and \eqref{7rYhab} we have 
\begin{eqnarray}\label{IahnbF.3}
w_\ast\big|_{\Omega}=u-u_\ast+\big(\Lambda_k^c v\big)\big|_{\Omega}
=u-u_\ast+u_\ast=u\,\,\,\mbox{ ${\mathscr{L}}^n$-a.e. in }\,\,\Omega,
\end{eqnarray}
whereas \eqref{IahnbF} and the analogue of \eqref{Pa-PL2} for 
$\big(\Omega^c\big)^\circ$ we obtain 
\begin{eqnarray}\label{IahnbF.4}
w_\ast\big|_{(\Omega^c)^\circ}=0+\big(\Lambda_k^c v\big)\big|_{(\Omega^c)^\circ}
=v\,\,\,\mbox{ ${\mathscr{L}}^n$-a.e. in }\,\,\big(\Omega^c\big)^\circ.
\end{eqnarray}
Thus, $w_\ast=w$ ${\mathscr{L}}^n$-a.e. in ${\mathbb{R}}^n$, 
by \eqref{IahnbF.3}-\eqref{IahnbF.4}, \eqref{7reer.2}, \eqref{7reer.4}, 
and \eqref{PJE-wcX}. With this in hand, the fact that $w$ belongs to 
$W^{k,p}({\mathbb{R}}^n)$ and satisfies \eqref{7reer.3} follows from 
\eqref{IahnbF.2}. This concludes the proof of the implication $(i)\Rightarrow(ii)$
and also justifies the last claim in the statement of the theorem. 

There remains to show that the implication $(ii)\Rightarrow(i)$ also holds.
To this end, suppose that $u\in W^{k,p}(\Omega)$, 
$v\in W^{k,p}\big((\Omega^c)^\circ\big)$ are such that the function
$w$ defined as in \eqref{7reer.2} belongs to $W^{k,p}({\mathbb{R}}^n)$.
Since, by design, $w\big|_{\Omega}=u$ and $w\big|_{(\Omega^c)^\circ}=v$,
condition \eqref{7reer} follows by writing 
\begin{eqnarray}\label{7reer.Tfavf}
\big({\mathscr{R}}^{(k)}_{\Omega\to\partial\Omega}\,u\big)(x)
=\left\{\lim\limits_{r\to 0^{+}}\meanint_{B(x,r)}
\partial^\alpha w\,d{\mathscr{L}}^n\right\}_{|\alpha|\leq k-1}
=\big({\mathscr{R}}^{(k)}_{(\Omega^c)^\circ\to\partial\Omega}\,v\big)(x)
\end{eqnarray}
at ${\mathcal{H}}^{d}$-a.e. point $x$ in 
$\partial\Omega=\partial\big(\big(\Omega^c\big)^\circ\big)$, 
thanks to (a two-fold application of) Corollary~\ref{NIceTRace.CC}.
Hence $(ii)\Rightarrow(i)$ also holds, completing the proof of the theorem.  
\end{proof}

A word of clarification regarding the statement of the above theorem is in order. 
A cursory inspection of the proof of Theorem~\ref{SSSg-g5dS} reveals that, 
in principle, the argument carries through under the less stringent demand 
that the Ahlfors regularity dimension $d$ of $\partial\Omega$ belongs to 
the interval $(0,n)$. However, if $d\in(0,n-1)$ then the $d$-Ahlfors regular 
set $\partial\Omega$ has $(n-1)$-dimensional Hausdorff measure zero and, as such, 
well-known removability results (cf., e.g., \cite[Lemma~9.1.10, p.\,237]{AH96}) 
give that the function $w$ defined as in \eqref{7reer.2} automatically belongs 
to $W^{k,p}({\mathbb{R}}^n)$ irrespective of the choice of the functions 
$u\in W^{k,p}(\Omega)$ and $v\in W^{k,p}\big((\Omega^c)^\circ\big)$. 
Now, on the one hand, such functions may be constructed with arbitrary traces on 
$B^{p,p}_{k-(n-d)/p}(\partial\Omega)$ (by Corollary~\ref{NIceTRace.CC}),
and this Besov space is nontrivial. On the other hand, as already mentioned, 
the equivalence $(i)\Leftrightarrow(ii)$ in Theorem~\ref{SSSg-g5dS} continues 
to hold for $d\in(0,n-1)$ as well. This contradiction shows that there is 
no two-sided $(\varepsilon,\delta)$-domain $\Omega$ in ${\mathbb{R}}^n$ 
whose boundary is $d$-Ahlfors regular for some $d\in(0,n-1)$.

\vskip 0.10in

The last result in this section requires certain notions from geometric 
measure theory (for standard terminology and basic results in this area 
the reader is referred to the informative discussion in \cite{EG}). 
Specifically, assume now that $\Omega$ is an open subset of ${\mathbb{R}}^n$ 
which is of locally finite perimeter. Recall that the measure-theoretic boundary
$\partial_*\Omega$ of the set $\Omega$ is defined by
\begin{eqnarray}\label{2.1.10}
\partial_*\Omega:=\Bigl\{x\in\partial\Omega:\,
\limsup\limits_{r\rightarrow 0}\,\frac{{\mathcal{L}}^n(B(x,r)\cap\Omega)}{r^n}>0
\,\,\mbox{ and }\,\,\limsup\limits_{r\rightarrow 0}\,
\frac{{\mathcal{L}}^n(B(x,r)\setminus\Omega)}{r^n}>0\Bigr\}.
\end{eqnarray}
The condition that the set $\Omega$ has locally finite perimeter allows us to define 
an outward unit normal $\nu=(\nu_j)_{1\leq j\leq n}$ at ${\mathcal{H}}^{n-1}$-a.e. 
point on $\partial_*\Omega$, in the sense of H.~Federer. In particular, if
${\mathcal{H}}^{n-1}(\partial\Omega\setminus\partial_*\Omega)=0$ 
then $\nu$ is defined ${\mathcal{H}}^{n-1}$-a.e. on $\partial\Omega$.
In such a context, given $m\in{\mathbb{N}}$ it is natural to consider a related 
version of the higher-order restriction operator \eqref{Ver-S2TG.3}, namely
the higher-order Dirichlet trace
\begin{equation}\label{newtrace}
{\rm TR}^{(m)}\,u:=\Bigl\{\frac{\partial^k u}{\partial\nu^k}\Bigr\}
_{0\leq k\leq m-1},
\end{equation}
which has been traditionally employed in the formulation of the classical Dirichlet 
boundary value problem for higher-order operators. A word of caution is in 
order here. Specifically, in general the unit normal $\nu$ has only bounded, 
measurable components, hence taking iterated normal derivatives requires attention. 
Concretely, we define for each $k\in\{0,...,m-1\}$
\begin{eqnarray}\label{nuk-a6h}
\frac{\partial^k}{\partial\nu^k}
:=\Bigl(\sum_{j=1}^n\xi_j\partial/\partial x_j\Bigr)^k\Big|_{\xi=\nu}
=\sum_{|\alpha|=k}\frac{k!}{\alpha!}\,\nu^\alpha\partial^\alpha,
\end{eqnarray}
which suggests setting (in an appropriate context)
\begin{equation}\label{nuk}
\frac{\partial^ku}{\partial\nu^k}
:=\sum_{|\alpha|=k}\frac{k!}{\alpha!}\,
\nu^\alpha\,{\mathscr{R}}^{(1)}_{\Omega\to\partial\Omega}
\,[\partial^\alpha u],\qquad\forall\,k\in\{0,1,...,m-1\},
\end{equation}
where ${\mathscr{R}}^{(1)}_{\Omega\to\partial\Omega}$ is the boundary trace 
operator of order one from Theorem~\ref{NIceTRace}.

Compared to \eqref{Ver-S2TG.3}, a distinguished feature of \eqref{newtrace} 
is that the latter has fewer components. More specifically, while 
${\mathscr{R}}^{(m)}_{\Omega\to\partial\Omega}\,u$ has
\begin{eqnarray}\label{comb-form}
\sum_{k=0}^{m-1}\left(\!\!\!\begin{array}{c} n+k-1\\ n-1\end{array}\!\!\!\right)
\end{eqnarray}
components, ${\rm TR}^{(m)}\,u$ has only $m$ components. It is then 
remarkable that the two trace mappings have the same null-space. 
This is made precise in the theorem stated below, which answers the question 
raised by J.\,Ne\v{c}as in \cite[Problem~4.1, p.\,91]{Nec}, 
\cite[Problem~4.1, p.\,86]{Nec2}, in a considerably more general setting
than the class of Lipschitz domains, as originally asked
(for the latter setting see also \cite{MaMiSh}, \cite{MiMi-HOPDE}). 

\begin{theorem}[The null-space of the higher-order Dirichlet trace operator]\label{NIcBBB}
Let $\Omega$ be an $(\varepsilon,\delta)$-domain in ${\mathbb{R}}^n$ with 
${\rm rad}\,(\Omega)>0$, and such that $\partial\Omega$ is $(n-1)$-Ahlfors 
regular and satisfies
\begin{eqnarray}\label{Tay-1}
{\mathcal{H}}^{n-1}(\partial\Omega\setminus\partial_*\Omega)=0.
\end{eqnarray}
Denote by $\nu$ the geometric measure theoretic outward unit normal to $\Omega$. 

Then for every $m\in{\mathbb{N}}$ and $p\in(1,\infty)$ one has
\begin{eqnarray}\label{Necas}
\mathring{W}^{m,p}(\Omega)=
\Big\{u\in W^{m,p}(\Omega):\,\frac{\partial^k u}{\partial\nu^k}=0\,\,
\mbox{ ${\mathcal{H}}^{n-1}$-a.e. on }\partial\Omega\,\,\mbox{ for }\,\,0\leq k\leq m-1
\Big\}.
\end{eqnarray}
\end{theorem}

\begin{proof}
To get started, set $\sigma:={\mathcal{H}}^{n-1}\lfloor\partial\Omega$. 
The long term goal is to show that the assignment
\begin{equation}\label{ass}
B^{p,p}_{m-1/p}(\partial\Omega)\ni\dot{f}=\{f_\alpha\}_{|\alpha|\leq m-1}
\longmapsto 
\Bigl\{\sum_{|\alpha|=k}\frac{k!}{\alpha!}\,\nu^\alpha\,f_\alpha\Bigr\}
_{0\leq k\leq m-1}\in L^p(\partial\Omega,\sigma)
\end{equation}
is one-to-one. To justify this, we shall make use of the fact that 
for each multi-index $\alpha\in{\mathbb{N}}_0^n$ there exist polynomial functions
$\big\{p^{\alpha\beta}_{jk}\big\}_{\stackrel{1\leq j,k\leq n}{|\beta|=|\alpha|-1}}$ 
of $n$ variables with the property that 
\begin{eqnarray}\label{m-xxx}
\partial^\alpha=\nu^\alpha\,\frac{\partial^{|\alpha|}}{\partial\nu^{|\alpha|}}
+\sum_{|\beta|=|\alpha|-1}\sum_{j,k=1}^n p^{\alpha\beta}_{jk}(\nu)
(\nu_j\partial_k-\nu_k\partial_j)\partial^\beta.
\end{eqnarray}
To prove identity 
\eqref{m-xxx}, assume that some $\alpha=(\alpha_1,...,\alpha_n)\in{\mathbb{N}}_0^n$ 
has been fixed and write
\begin{eqnarray}\label{m-xxx2}
\partial^\alpha & = & \prod_{j=1}^n\Bigl(\frac{\partial}{\partial x_j}
\Bigr)^{\alpha_j}
=\prod_{j=1}^n\Bigl[\sum_{k=1}^n\xi_k\Bigl(\xi_k
\frac{\partial}{\partial x_j}-\xi_j\frac{\partial}{\partial x_k}\Bigr)
+\sum_{k=1}^n\xi_j\xi_k\frac{\partial}{\partial x_k}\Bigr]^{\alpha_j}
\Bigl|_{\xi=\nu}
\nonumber\\[6pt]
& = & \prod_{j=1}^n\Bigl[\sum_{\ell=0}^{\alpha_j}\frac{\alpha_j!}{\ell!(\alpha_j-\ell)!}
\Bigl(\sum_{k=1}^n\xi_k\Bigl(\xi_k
\frac{\partial}{\partial x_j}-\xi_j\frac{\partial}{\partial x_k}\Bigr)\Bigr)
^{\alpha_j-\ell}\Big(\sum_{k=1}^n\xi_j\xi_k\frac{\partial}{\partial x_k}\Big)^\ell
\Bigr]\Bigl|_{\xi=\nu}
\\[6pt]
& = & \prod_{j=1}^n\Bigl[\Big(\sum_{k=1}^n\xi_j\xi_k\frac{\partial}{\partial x_k}\Big)
^{\alpha_j}+\sum_{\ell=0}^{\alpha_j-1}\frac{\alpha_j!}{\ell!(\alpha_j-\ell)!}
\Bigl(\sum_{k=1}^n\xi_k\Bigl(\xi_k
\frac{\partial}{\partial x_j}-\xi_j\frac{\partial}{\partial x_k}\Bigr)\Bigr)
^{\alpha_j-\ell}\Big(\sum_{k=1}^n\xi_j\xi_k\frac{\partial}{\partial x_k}\Big)^\ell
\Bigr]\Bigl|_{\xi=\nu}.
\nonumber
\end{eqnarray}
Upon noticing that 
\begin{eqnarray}\label{YFSZM895}
\prod_{j=1}^n\Bigl[\Big(\sum_{k=1}^n\xi_j\xi_k\frac{\partial}{\partial x_k}\Big)
^{\alpha_j}\Bigr]\Bigl|_{\xi=\nu}
=\prod_{j=1}^n\nu_j^{\alpha_j}\frac{\partial^{\alpha_j}}{\partial\nu^{\alpha_j}}
=\nu^\alpha\frac{\partial^{|\alpha|}}{\partial\nu^{|\alpha|}},
\end{eqnarray}
and $(\xi_k\partial/\partial x_j-\xi_j\partial/\partial x_k)\big|_{\xi=\nu}
=-(\nu_j\partial_k-\nu_k\partial_j)$, formula \eqref{m-xxx} follows.

Assume next that $\dot{f}\in B^{p,p}_{m-1/p}(\partial\Omega)$ is 
mapped to zero by the assignment \eqref{ass} and consider the function
$u:={\mathscr{E}}^{(m)}_{\partial\Omega\to\Omega}\dot{f}$ in $\Omega$,
with ${\mathscr{E}}^{(m)}_{\partial\Omega\to\Omega}$ as in Theorem~\ref{NIceTRace}.
Then Theorem~\ref{NIceTRace} ensures that
\begin{eqnarray}\label{TYgghcgh}
\begin{array}{c}
u\in W^{m,p}(\Omega)\,\,\mbox{ and }\,\,
f_\alpha={\mathscr{R}}_{\Omega\to\partial\Omega}^{(1)}(\partial^\alpha u)
\,\,\,\sigma\mbox{-a.e. on }\,\,\partial\Omega
\\[6pt]
\mbox{for each multi-index }\,\alpha\in{\mathbb{N}}_0^n\,\,
\mbox{ with }\,\,|\alpha|\leq m-1.
\end{array}
\end{eqnarray}
Also, granted the current assumptions on $\dot{f}$, 
\begin{eqnarray}\label{TYggh569}
\frac{\partial^k u}{\partial\nu^k}=0\,\,\,\sigma\mbox{-a.e. on }\,\,\partial\Omega
\,\,\mbox{ for }\,\,k=0,1,...,m-1. 
\end{eqnarray}
To proceed, observe that since ${\mathscr{C}}^\infty_c({\mathbb{R}}^n)
\hookrightarrow W^{m,p}({\mathbb{R}}^n)$ densely and since the restriction 
operator $W^{m,p}({\mathbb{R}}^n)\ni v\mapsto v\big|_{\Omega}\in W^{m,p}(\Omega)$
is well-defined, linear, continuous and surjective (as seen from Theorem~\ref{cjeg}),
it follows that
\begin{eqnarray}\label{TYGY87f}
\bigl\{\varphi\big|_{\Omega}:\,\varphi\in{\mathscr{C}}^\infty_c({\mathbb{R}}^n)\big\}
\hookrightarrow W^{m,p}(\Omega)\,\,\,\mbox{ densely}.
\end{eqnarray}
Consequently, it is possible to select a sequence 
$\{\varphi_i\}_{i\in{\mathbb{N}}}\subseteq{\mathscr{C}}^\infty_c({\mathbb{R}}^n)$ 
with the property that if $v_i:=\varphi_i\big|_{\Omega}$ for each $i\in{\mathbb{N}}$
then $v_i\to u$ in $W^{m,p}(\Omega)$ as $i\to\infty$. In view of 
Theorem~\ref{NIceTRace} and \eqref{TYgghcgh}, this implies
\begin{eqnarray}\label{m-xIHab}
\begin{array}{c}
(\partial^\alpha v_i)\big|_{\partial\Omega}\longrightarrow f_\alpha
\,\,\mbox{ in }\,\,L^p(\partial\Omega,\sigma)\,\,\mbox{ as }\,\,i\to\infty,
\\[6pt]
\mbox{for every multi-index }\,\,\alpha\in{\mathbb{N}}_0^n
\,\,\mbox{ with }\,\,|\alpha|\leq m-1.
\end{array}
\end{eqnarray}
In particular, from \eqref{m-xIHab}, \eqref{nuk}, and \eqref{TYggh569}, we obtain 
\begin{eqnarray}\label{m-xIHab.5tB}
\frac{\partial^{\ell} v_i}{\partial\nu^{\ell}}\longrightarrow 0
\,\,\mbox{ in }\,\,L^p(\partial\Omega,\sigma)\,\,\mbox{ as }\,\,i\to\infty,
\,\,\mbox{ for every }\,\,\ell\in\{0,1,...,m-1\}.
\end{eqnarray}

We shall now prove by induction on the number $\ell\in\{0,1,...,m-1\}$ that
the sequence 
\begin{eqnarray}\label{m-xiii}
\begin{array}{c}
\big\{(\partial^\alpha v_i)\big|_{\partial\Omega}\big\}_{i\in{\mathbb{N}}}
\,\,\,\mbox{ converges to zero weakly in }\,\,L^p(\partial\Omega,\sigma)
\\[6pt]
\mbox{for every multi-index }\,\,\alpha\in{\mathbb{N}}_0^n
\,\,\mbox{ with }\,\,|\alpha|=\ell,
\end{array}
\end{eqnarray}
and that the sequence 
\begin{eqnarray}\label{m-xxx.64DD}
\begin{array}{c}
\Big\{\nu_j(\partial_k\partial^\beta v_i)\big|_{\partial\Omega}
-\nu_k(\partial_j\partial^\beta v_i)\big|_{\partial\Omega}\Big\}_{i\in{\mathbb{N}}}
\,\,\,\mbox{ converges to zero weakly in }\,\,L^p(\partial\Omega,\sigma)
\\[8pt]
\mbox{for every multi-index $\beta\in{\mathbb{N}}_0^n$
with $|\beta|=\ell-1$ and any $j,k\in\{1,...,n\}$}.
\end{array}
\end{eqnarray}
When $\ell=0$ the second condition is void, whereas \eqref{m-xiii} follows
from by combining \eqref{m-xIHab} and 
\eqref{TYgghcgh} (written for $\alpha=(0,...,0)\in{\mathbb{N}}_0^n$)
with \eqref{TYggh569} (used with $k=0$; cf. also \eqref{nuk} in this regard). 
Assume next that both \eqref{m-xiii} and \eqref{m-xxx.64DD} hold for some 
number $\ell\in\{0,1,...,m-2\}$ and fix 
$\alpha\in{\mathbb{N}}_0^n$ with $|\alpha|=\ell+1$, 
$\beta\in{\mathbb{N}}_0^n$ with $|\beta|=\ell$, as well as some 
$j,k\in\{1,...,n\}$. Now, \eqref{m-xIHab} 
implies that $\Big\{\nu_j(\partial_k\partial^\beta v_i)\big|_{\partial\Omega}
\Big\}_{i\in{\mathbb{N}}}$ and $\Big\{\nu_k(\partial_j\partial^\beta v_i)
\big|_{\partial\Omega}\Big\}_{i\in{\mathbb{N}}}$ are bounded sequences in 
$L^p(\partial\Omega,\sigma)$. Granted this and given that 
\begin{eqnarray}\label{m-xIOH.gv}
\big\{\psi\big|_{\partial\Omega}:\,\psi\in{\mathscr{C}}^\infty_c({\mathbb{R}}^n)\big\}
\hookrightarrow L^p(\partial\Omega,\sigma)\,\,\,\mbox{ densely},
\end{eqnarray}
the fact that the sequence 
$\Big\{\nu_j(\partial_k\partial^\beta v_i)\big|_{\partial\Omega}
-\nu_k(\partial_j\partial^\beta v_i)\big|_{\partial\Omega}\Big\}_{i\in{\mathbb{N}}}$
converges to zero weakly in $L^p(\partial\Omega,\sigma)$ will follows as soon as we 
show that for every $\psi\in{\mathscr{C}}^\infty_c({\mathbb{R}}^n)$ there holds
\begin{eqnarray}\label{m-xxx.64DLGS}
\int_{\partial\Omega}
\Big(\nu_j(\partial_k\partial^\beta v_i)\big|_{\partial\Omega}
-\nu_k(\partial_j\partial^\beta v_i)\big|_{\partial\Omega}\Big)
\big(\psi\big|_{\partial\Omega}\big)\,d\sigma\longrightarrow 0
\,\,\mbox{ as }\,\,i\to\infty.
\end{eqnarray}
To see that this is the case, denote by $\{{\bf e}_r\}_{1\leq r\leq n}$ 
the standard orthonormal basis in ${\mathbb{R}}^n$ and, for each 
fixed $i\in{\mathbb{N}}$, consider the vector field 
\begin{eqnarray}\label{m-IVcc75}
\vec{F}:=\Big(\psi\partial_k\partial^\beta v_i+\partial_k\psi\partial^\beta v_i\Big)
{\bf e}_j-\Big(\psi\partial_j\partial^\beta v_i+\partial_j\psi\partial^\beta v_i\Big)
{\bf e}_k.
\end{eqnarray}
Note that all components of $\vec{F}$ are Lipschitz functions in $\Omega$ with 
bounded support, and 
\begin{eqnarray}\label{m-IVcc7522}
{\rm div}\vec{F} &=& 
\partial_j\Big(\psi\partial_k\partial^\beta v_i+\partial_k\psi\partial^\beta v_i\Big)
-\partial_k\Big(\psi\partial_j\partial^\beta v_i+\partial_j\psi\partial^\beta v_i\Big)
\nonumber\\[4pt]
&=& 
\partial_j\psi\partial_k\partial^\beta v_i
+\psi\partial_j\partial_k\partial^\beta v_i
+\partial_j\partial_k\psi\partial^\beta v_i
+\partial_k\psi\partial_j\partial^\beta v_i
\nonumber\\[4pt]
&& 
-\partial_k\psi\partial_j\partial^\beta v_i
-\psi\partial_k\partial_j\partial^\beta v_i
-\partial_k\partial_j\psi\partial^\beta v_i
-\partial_j\psi\partial_k\partial^\beta v_i
\nonumber\\[4pt]
&=& 0\,\,\mbox{ in }\,\,\Omega.
\end{eqnarray}
Based on this, the De Giorgi-Federer Divergence Theorem
(cf., e.g., \cite{EG}), and \eqref{Tay-1}, we deduce that 
\begin{eqnarray}\label{N-5gvvAsdu}
0=\int_{\partial\Omega}\nu\cdot\big(\vec{F}\big|_{\partial\Omega}\big)\,d\sigma
&=& \int_{\partial\Omega}
\Big(\nu_j(\partial_k\partial^\beta v_i)\big|_{\partial\Omega}
-\nu_k(\partial_j\partial^\beta v_i)\big|_{\partial\Omega}\Big)
\big(\psi\big|_{\partial\Omega}\big)\,d\sigma
\nonumber\\[4pt]
&&+\int_{\partial\Omega}
\Big(\nu_j(\partial_k\psi)\big|_{\partial\Omega}
-\nu_k(\partial_j\psi)\big|_{\partial\Omega}\Big)
\big(\partial^\beta v_i)\big|_{\partial\Omega}\,d\sigma.
\end{eqnarray}
The bottom line is that for each $i\in{\mathbb{N}}$ we have 
\begin{eqnarray}\label{Mmamg}
&&\hskip -0.50in
\int_{\partial\Omega}
\Big(\nu_j(\partial_k\partial^\beta v_i)\big|_{\partial\Omega}
-\nu_k(\partial_j\partial^\beta v_i)\big|_{\partial\Omega}\Big)
\big(\psi\big|_{\partial\Omega}\big)\,d\sigma
\nonumber\\[4pt]
&&\hskip 1.50in
=-\int_{\partial\Omega}
\Big(\nu_j(\partial_k\psi)\big|_{\partial\Omega}
-\nu_k(\partial_j\psi)\big|_{\partial\Omega}\Big)
\big(\partial^\beta v_i)\big|_{\partial\Omega}\,d\sigma.
\end{eqnarray}
Since the multi-index $\beta$ has length $\ell$, hypothesis \eqref{m-xiii} ensures 
that the sequence 
$\big\{(\partial^\beta v_i)\big|_{\partial\Omega}\big\}_{i\in{\mathbb{N}}}$ 
converges to zero weakly in $L^p(\partial\Omega,\sigma)$. Now \eqref{m-xxx.64DLGS}
follows from this and \eqref{Mmamg}.  This takes care of the version of 
\eqref{m-xxx.64DD} with $\ell+1$ in place of $\ell$. As regards the version of 
\eqref{m-xiii} with $\ell+1$ in place of $\ell$, pick a multi-index
$\alpha\in{\mathbb{N}}_0^n$ of length $\ell+1$ and observe that  
identity \eqref{m-xxx} gives
\begin{eqnarray}\label{m-xxx.22}
\partial^\alpha v_i=\nu^\alpha\,\frac{\partial^{\ell+1} v_i}{\partial\nu^{\ell+1}}
+\sum_{|\beta|=|\alpha|-1}\sum_{j,k=1}^n p^{\alpha\beta}_{jk}(\nu)
(\nu_j\partial_k-\nu_k\partial_j)\partial^\beta v_i\,\,\mbox{ $\sigma$-a.e. on }\,\,
\partial\Omega,\quad\forall\,i\in{\mathbb{N}}.
\end{eqnarray}
From what we have just proved, the double sum in the right-hand side of 
\eqref{m-xxx.22} converges to zero weakly in $L^p(\partial\Omega,\sigma)$
as $i\to\infty$ (since the length of all multi-indices $\beta$ involved is $\ell$, 
and each $p^{\alpha\beta}_{jk}(\nu)$ is a bounded function). With this in hand and
recalling \eqref{m-xIHab.5tB}, it readily follows from \eqref{m-xxx.22} that 
$\big\{(\partial^\alpha v_i)\big|_{\partial\Omega}\big\}_{i\in{\mathbb{N}}}$
converges to zero weakly in $L^p(\partial\Omega,\sigma)$. This completes
the induction scheme, hence \eqref{m-xiii} and \eqref{m-xxx.64DD} hold for every
$\ell\in\{0,1,...,m-1\}$. 

At this stage, by combining \eqref{m-xIHab} with \eqref{m-xiii} we deduce that
$f_\alpha=0$ for every multi-index $\alpha\in{\mathbb{N}}_0^n$ with $|\alpha|\leq m-1$, 
finishing the proof of the fact that the assignment 
\begin{equation}\label{ass.777}
\Psi:B^{p,p}_{m-1/p}(\partial\Omega)\longrightarrow L^p(\partial\Omega,\sigma),\quad
\Psi\dot{f}:=\Bigl\{\sum_{|\alpha|=k}\frac{k!}{\alpha!}\,\nu^\alpha\,f_\alpha\Bigr\}
_{0\leq k\leq m-1},\quad\forall\,\dot{f}=\{f_\alpha\}_{|\alpha|\leq m-1},
\end{equation}
is one-to-one. Let us also note that thanks to \eqref{nuk}, \eqref{ass.777},
and \eqref{Ver-S2TG.2}, we have
\begin{eqnarray}\label{nuk.XXX}
\Big\{\frac{\partial^k u}{\partial\nu^k}\Big\}_{0\leq k\leq m-1}
&=& \Bigg\{\sum_{|\alpha|=k}\frac{k!}{\alpha!}\,
\nu^\alpha\,{\mathscr{R}}^{(1)}_{\Omega\to\partial\Omega}\,[\partial^\alpha u]
\Bigg\}_{0\leq k\leq m-1}
\nonumber\\[4pt]
&=& \Psi\Bigl({\mathscr{R}}^{(m)}_{\Omega\to\partial\Omega}\,u\Big),\qquad
\forall\,u\in W^{m,p}(\Omega).
\end{eqnarray}
All things considered, formula \eqref{Necas} now follows from \eqref{nuk.XXX}, 
the fact that \eqref{GGG-89.UUU} is a well-defined operator (used here 
with $k:=m$), the injectivity of the map \eqref{ass.777}, and \eqref{Uan-Taev75}. 
\end{proof}

\section{Interpolation results for Sobolev spaces with partially vanishing traces}
\label{Sect:6}
\setcounter{equation}{0}

Here we take up the task of proving that, in an appropriate geometrical context, 
the scale of spaces introduced in \eqref{uig-adjb-2} is stable under 
both the complex and the real method of interpolation.
In order to facilitate the subsequent discussion, call a family of Banach spaces
$(X_p)_{p\in I}$ indexed by an open interval $I\subseteq(1,\infty)$ a 
{\tt complex interpolation scale} provided
\begin{eqnarray}\label{auig-1}
\big[X_{p_0},X_{p_1}\big]_{\theta}=X_{p}
\end{eqnarray}
whenever $p_0,p_1\in I$, $0<\theta<1$, and $1/p=(1-\theta)/p_0+\theta/p_1$, 
and call $(X_p)_{p\in I}$ a {\tt real interpolation scale} if 
\begin{eqnarray}\label{auig-1345}
\big(X_{p_0},X_{p_1}\big)_{\theta,p}=X_{p}
\end{eqnarray}
whenever $p_0,p_1\in I$, $0<\theta<1$, and $1/p=(1-\theta)/p_0+\theta/p_1$.
Above, $[\cdot,\cdot]_\theta$ and $(\cdot,\cdot)_{\theta,p}$ denote, respectively,
the standard complex and real brackets of interpolation (as defined in, e.g., 
\cite{BeLo76}). The starting point is the following consequence of 
Theorems~\ref{GCC-67}-\ref{YTah-YYHa8}.

We continue by recording an useful result which is essentially folklore
(a proof may be found in, e.g., \cite{KMM}). First, we make a definition.
Let $X_0, X_1$ and $Y_0,Y_1$ be two compatible pairs of Banach spaces.
Call $\{Y_0,Y_1\}$ a {\tt retract} of $\{X_0,X_1\}$ if there exist two
bounded, linear operators $E:Y_i\to X_i$, $R:X_i\to Y_i$, $i=0,1$,
such that $R\circ E=I$, the identity map, on each $Y_i$, $i=0,1$.

\begin{lemma}\label{l3.3}
Assume that $X_0,X_1$ and $Y_0,Y_1$ are two compatible pairs of
Banach spaces such that $\{Y_0,Y_1\}$ is a retract of $\{X_0,X_1\}$
(as before, the ``extension-restriction'' operators are denoted by $E$ and
$R$, respectively). Then for each $\theta\in(0,1)$ and $0<q\leq\infty$,
\begin{equation}\label{retract}
[Y_0,Y_1]_\theta=R\Bigl([X_0,X_1]_\theta\Bigr)\,\,
\mbox{ and }\,\,(Y_0,Y_1)_{\theta,q}
=R\Bigl((X_0,X_1)_{\theta,q}\Bigr).
\end{equation}

As a corollary, the following also holds. Assume that $X_0, X_1$ is a
compatible pair of Banach spaces and that $P$ is a common projection
(i.e., a linear, bounded operator on $X_i$, $i=0,1$, such that $P^2=P$).
Then the real and complex interpolation brackets commute with the action
of $P$, i.e.,
\begin{equation}\label{proj-int}
[PX_0,PX_1]_\theta=P\Bigl([X_0,X_1]_\theta\Bigr)\,\,
\mbox{ and }\,\,(PX_0,PX_1)_{\theta,q}
=P\Bigl((X_0,X_1)_{\theta,q}\Bigr),
\end{equation}
for each $\theta\in(0,1)$ and $0<q\leq\infty$.
\end{lemma}

A word of clarification is in order here. Specifically, generally speaking, given 
two normed spaces $X$, $Y$ and a linear, bounded operator $T:X\to Y$, by $TX$ we 
shall denote its range equipped with the graph-norm
\begin{equation}\label{eqTX}
\|y\|_{TX}:=\inf\{\|x\|_X:\,x\in X \mbox{ such that }y=Tx\},\qquad y\in TX.
\end{equation}
In particular, this is the sense in which \eqref{retract}
and \eqref{proj-int} should be understood.

\begin{proposition}\label{YTah-aP3G}
Assume that $D\subseteq{\mathbb{R}}^n$ is a closed set which is $d$-Ahlfors regular
for some $d\in(0,n)$, and fix a number $k\in{\mathbb{N}}$.
Then the operator 
\begin{eqnarray}\label{Tgb-aj}
{\mathscr{P}}^{(k)}_D:=I-{\mathscr{E}}^{(k)}_D\circ{\mathscr{R}}^{(k)}_D
\end{eqnarray}
is a continuous and linear projection from the classical Sobolev space
$W^{k,p}({\mathbb{R}}^n)$ onto the space $W^{k,p}_D({\mathbb{R}}^n)$
whenever $\max\{1,n-d\}<p<\infty$.
\end{proposition}

\begin{proof}
The claims about the operator \eqref{Tgb-aj} are readily seen 
from Theorem~\ref{YTah-YYHa8} and Theorem~\ref{GCC-67}.
\end{proof}

In combination with the abstract results from Lemma~\ref{l3.3}, 
Proposition~\ref{YTah-aP3G} permits us to prove the following interpolation
result in the context of ${\mathbb{R}}^n$.

\begin{proposition}\label{Haah8uN}
Suppose $D\subseteq{\mathbb{R}}^n$ is a closed set which is $d$-Ahlfors regular
for some $d\in(0,n)$, and fix a number $k\in{\mathbb{N}}$.
Then, in a natural sense, 
\begin{eqnarray}\label{aJN}
\begin{array}{c}
\big\{W^{k,p}_D({\mathbb{R}}^n)\big\}_{\max\{1,n-d\}<p<\infty}\,\,
\,\mbox{ is an interpolation scale}
\\[6pt]
\mbox{both for the complex and the real method}.
\end{array}
\end{eqnarray}
\end{proposition}

\begin{proof}
This follows from Proposition~\ref{YTah-aP3G}, the fact that 
$\big\{W^{k,p}({\mathbb{R}}^n)\big\}_{1<p<\infty}$ is an interpolation scale
both for the complex and the real method for each fixed $k\in{\mathbb{N}}$, 
and the second part in Lemma~\ref{l3.3}.
\end{proof}

With this in hand, we are ready to prove our main interpolation result, 
formulated below. 

\begin{theorem}[Interpolation Theorem]\label{HanmNB8}
Assume that $\Omega\subseteq{\mathbb{R}}^n$ and $D\subseteq\overline{\Omega}$ are 
such that $D$ is closed and $d$-Ahlfors regular for some $d\in(0,n)$, and $\Omega$ 
is locally an $(\varepsilon,\delta)$-domain near $\partial\Omega\setminus D$.
In addition, fix a number $k\in{\mathbb{N}}$. Then 
\begin{eqnarray}\label{GYaYn16H}
\begin{array}{c}
\big\{W^{k,p}_D(\Omega)\big\}_{\max\{1,n-d\}<p<\infty}\,\,
\,\mbox{ is an interpolation scale}
\\[6pt]
\mbox{both for the complex and the real method},
\end{array}
\end{eqnarray}
and, with the convention that $1/p+1/p'=1$,  
\begin{eqnarray}\label{GYaYn16H.Re}
\begin{array}{c}
\Big\{\big(W^{k,p'}_D(\Omega)\big)^*\Big\}_{\max\{1,n-d\}<p<\infty}\,\,
\,\mbox{ is an interpolation scale}
\\[6pt]
\mbox{both for the complex and the real method}.
\end{array}
\end{eqnarray}

In particular, if $\Omega$ is a nonempty open subset of ${\mathbb{R}}^n$ 
with the property that $\partial\Omega$ is $d$-Ahlfors regular
for some $d\in(0,n)$, then 
\begin{eqnarray}\label{GYaYn16H.Ts}
\begin{array}{c}
\big\{\mathring{W}^{k,p}(\Omega)\big\}_{\max\{1,n-d\}<p<\infty}\,\,
\,\mbox{ is an interpolation scale}
\\[6pt]
\mbox{both for the complex and the real method},
\end{array}
\end{eqnarray}
and
\begin{eqnarray}\label{GYaYn16H.Ts345}
\begin{array}{c}
\big\{W^{-k,p}(\Omega)\big\}_{\max\{1,n-d\}<p<\infty}\,\,
\,\mbox{ is an interpolation scale}
\\[6pt]
\mbox{both for the complex and the real method}.
\end{array}
\end{eqnarray}
\end{theorem}

\begin{proof}
Collectively, Theorem~\ref{YTah-YYHa9}, Theorem~\ref{cjeg.5}, and 
Theorem~\ref{cjeg.5AD} prove that, under assumptions made in the statement 
of the current corollary, 
\begin{eqnarray}\label{GYaTba}
\big\{W^{k,p}_D(\Omega)\big\}_{\max\{1,n-d\}<p<\infty}\,\,
\mbox{ is a retract of }\,\,\,
\big\{W^{k,p}_D({\mathbb{R}}^n)\big\}_{\max\{1,n-d\}<p<\infty}.
\end{eqnarray}
Based on this, \eqref{aJN}, the first part in Lemma~\ref{l3.3}, 
and Corollary~\ref{cjPPa}, the claim made in \eqref{GYaYn16H} follows.
Then the claim in \eqref{GYaYn16H.Re} becomes a consequence of 
\eqref{GYaYn16H}, duality theorems for the complex and real methods 
of interpolation (cf., e.g., \cite[Corollary~4.5.2, p.\,98]{BeLo76} and 
\cite[Theorem~3.7.1, p.\,54]{BeLo76}), and part $(3)$ in Lemma~\ref{TYRD-f4f}.
Finally, the claims in \eqref{GYaYn16H.Ts}-\eqref{GYaYn16H.Ts345}
are implied by \eqref{GYaYn16H}-\eqref{GYaYn16H.Re}, 
part $(5)$ in Lemma~\ref{TYRD-f4f}, and \eqref{azTagbM-W}.
\end{proof}

\section{Applications to mixed boundary value problems for higher-order systems}
\label{Sect:7}
\setcounter{equation}{0}

The aim in this section is to illustrate how the functional analytic results proved 
so far may be used to establish some very general solvability results for mixed boundary 
value problems. In this endeavor, we shall work with higher-order systems of PDE's
in divergence form, with bounded measurable coefficients. This requires a number 
of preliminaries which we first dispense with. The reader is alerted to the fact that
while here we shall frequently work with vector-valued functions, our notation 
does not necessarily reflects that (though matters should always be clear from the
context). 

To start the build-up in earnest, let ${\mathcal{L}}$ be the differential operator 
of order $2m$, where $m\in{\mathbb{N}}$, in divergence form given by 
\begin{eqnarray}\label{coHGB2}
{\mathcal{L}}u:=\sum_{|\alpha|=|\beta|=m} 
\partial^\alpha\big(A_{\alpha\beta}\,\partial^\beta u\big),
\end{eqnarray}
whose tensor coefficient $A=(A_{\alpha\beta})_{|\alpha|=|\beta|=m}$
consists of $M\times M$ matrices $A_{\alpha\beta}$ with bounded, measurable, 
complex-valued entries, i.e., 
\begin{eqnarray}\label{coHGB2.B}
A_{\alpha\beta}=\big(a^{\alpha\beta}_{ij}\big)_{1\leq i,j\leq M}
\,\,\,\mbox{ with each }\,\,a^{\alpha\beta}_{ij}\in L^\infty(\Omega,{\mathscr{L}}^n),
\end{eqnarray}
and where the function $u=(u_1,...,u_M)$ is ${\mathbb{C}}^M$-valued. The first 
order of business is to associate a conormal derivative for ${\mathcal{L}}$. 

In order to motivate the subsequent
discussion, consider first the situation when $\Omega$ is a bounded
${\mathscr{C}}^\infty$ domain and denote by $\nu=(\nu_j)_{1\leq j\leq n}$ 
its outward unit normal. Also, set $\sigma:={\mathcal{H}}^{n-1}\lfloor\partial\Omega$.
In the case when the entries of each $A_{\alpha\beta}$ are functions from  
${\mathscr{C}}^\infty(\overline{\Omega})$, given any ${\mathbb{C}}^M$-valued 
functions $u,v$ with components from ${\mathscr{C}}^\infty(\overline{\Omega})$, 
repeated integrations by parts yield
\begin{eqnarray}\label{con-prop2}
&&\sum_{|\alpha|=|\beta|=m}
\int_{\Omega}\Bigl\langle A_{\alpha\beta}\,\partial^\beta u,
\partial^\alpha v\Bigr\rangle\,d{\mathscr{L}}^n
\\[4pt]
&&\qquad\qquad
=(-1)^{m+1}\int_{\partial\Omega}\Bigl\langle\partial^A_{\nu}u,
{\mathscr{R}}^{(m)}_{\Omega\to\partial\Omega}v\Bigr\rangle\,d\sigma
+(-1)^m\int_{\Omega}\bigl\langle{\mathcal{L}}u,v\bigr\rangle\,d{\mathscr{L}}^n,
\nonumber
\end{eqnarray}
where $\langle\cdot,\cdot\rangle$ is the usual (real) pointwise inner product of 
vector-valued functions, and the vector-valued function $\partial^A_{\nu}u$ 
is given by
\begin{eqnarray}\label{con-comp}
\begin{array}{l}
\partial^A_\nu u
:=\Bigl\{\big(\partial_\nu^A u\big)_\delta\Bigr\}_{|\delta|\leq m-1}
\mbox{ with the $\delta$-component given by the formula} 
\\[16pt]
{\displaystyle
\big(\partial^A_\nu u\big)_\delta:=\sum_{|\alpha|=|\beta|=m}\sum_{j=1}^n(-1)^{|\delta|}
\,\frac{\alpha!|\delta|!(m-|\delta|-1)!}{m!\delta!(\alpha-\delta-e_j)!}
\nu_jA_{\alpha\beta}\Bigl(\partial^{\alpha+\beta-\delta-e_j}u}\Bigr)
\Bigl|_{\partial\Omega},
\end{array}
\end{eqnarray}
with the convention that the sum in $\alpha$ and $j$ is only performed 
over those $\alpha$'s and $j$'s such that $\alpha-\delta-e_j$ does not 
have any negative components (and with $e_j:=(0,...,0,1,0,...,0)\in{\mathbb{N}}_0^n$ 
with the only nonzero component on the $j$-th slot, $j\in\{1,...,n\}$).
While retaining the same context as above, suppose we are interested in the
case in which the function $\partial^A_\nu u$ vanishes on $\partial\Omega\setminus D$
where $D$ is a given closed subset of $\partial\Omega$. 
Note that if $v\in{\mathscr{C}}_D^\infty(\Omega)$, then \eqref{con-prop2} becomes
\begin{eqnarray}\label{con-pTi}
&&\sum_{|\alpha|=|\beta|=m}
\int_{\Omega}\Bigl\langle A_{\alpha\beta}\,\partial^\beta u,
\partial^\alpha v\Bigr\rangle\,d{\mathscr{L}}^n
\\[4pt]
&&\qquad\qquad
=(-1)^{m+1}\int_{\partial\Omega\setminus D}\Bigl\langle\partial^A_{\nu}u,
{\mathscr{R}}^{(m)}_{\Omega\to\partial\Omega}v\Bigr\rangle\,d\sigma
+(-1)^m\int_{\Omega}\bigl\langle{\mathcal{L}}u,v\bigr\rangle\,d{\mathscr{L}}^n,
\nonumber
\end{eqnarray}
since ${\mathscr{R}}^{(m)}_{\Omega\to\partial\Omega}v=0$ near $D$. 
In fact, for every $p\in(1,\infty)$ we have 
${\mathscr{R}}^{(m)}_{\Omega\to\partial\Omega}v
\in B^{p,p}_{m-1/p}(\partial\Omega)$, with compact support contained in the 
relatively open subset $\partial\Omega\setminus D$ of $\partial\Omega$. 
Based on this and \eqref{con-pTi}, it is not difficult to see that under the
current smoothness hypotheses on the objects involved, the fact that  
$\partial^A_\nu u$ vanishes on $\partial\Omega\setminus D$ is equivalent to having 
\begin{eqnarray}\label{con-pTi.2}
\sum_{|\alpha|=|\beta|=m}
\int_{\Omega}\Bigl\langle A_{\alpha\beta}\,\partial^\beta u,
\partial^\alpha v\Bigr\rangle\,d{\mathscr{L}}^n
=(-1)^m\int_{\Omega}\bigl\langle{\mathcal{L}}u,v\bigr\rangle\,d{\mathscr{L}}^n,
\qquad\forall\,v\in{\mathscr{C}}_D^\infty(\Omega).
\end{eqnarray}
We continue to retain the same setting as above and, in addition, fix some
$p'\in(1,\infty)$. Then thanks to \eqref{uig-adjb-2} and \eqref{con-pTi.2} 
the condition that $\partial^A_\nu u$ vanishes on $\partial\Omega\setminus D$
may be further equivalently expressed as 
\begin{eqnarray}\label{con-pTi.3}
\sum_{|\alpha|=|\beta|=m}
\int_{\Omega}\Bigl\langle A_{\alpha\beta}\,\partial^\beta u,
\partial^\alpha v\Bigr\rangle\,d{\mathscr{L}}^n
=(-1)^m\int_{\Omega}\bigl\langle{\mathcal{L}}u,v\bigr\rangle\,d{\mathscr{L}}^n,
\qquad\forall\,v\in W^{m,p'}_D(\Omega).
\end{eqnarray}

At this stage in the build-up, we wish to assign a meaning of the condition that 
\begin{eqnarray}\label{con-pTi.4}
\partial^A_\nu u\,\,\,\mbox{ vanishes on }\,\,\,\partial\Omega\setminus D
\end{eqnarray}
in a much more general setting than the smooth one considered so far. 
To be specific, suppose that $\Omega$ is an open subset of ${\mathbb{R}}^n$ 
and that ${\mathcal{L}}$ is as in \eqref{coHGB2}-\eqref{coHGB2.B}, for 
some $m\in{\mathbb{N}}$. Also, assume that $D$ is a closed subset of 
$\partial\Omega$ and that $u\in W^{m,p}_D(\Omega)$ for some $p\in(1,\infty)$. 
Of course, this renders the pointwise definition of $\partial^A_\nu u$ from
\eqref{con-comp} utterly inadequate so, instead, we shall attempt to make 
sense of \eqref{con-pTi.3} (which, as indicated above, is an equivalent 
way of expressing \eqref{con-pTi.4} in the smooth case). With this goal in mind, 
note that if $p,p'$ are H\"older conjugate exponents (something we will assume
from now on), then the left-hand side of \eqref{con-pTi.3} is well-defined
for any $v\in W^{m,p'}_D(\Omega)$, given that we are assuming that  
the tensor coefficient has bounded entries. This being said, making sense
of the right-hand side could be problematic since one cannot expect 
$f:={\mathcal{L}}u$ to be more regular than a generic distribution 
in $W^{-m,p}(\Omega)$.

Having identified this issue, the remedy we propose is as follows. 
As a preamble, let us agree that given a topological vector space $X$,
with dual $X^*$, the symbol ${}_{X^\ast}\langle\cdot,\cdot\rangle_{X}$
indicates the pairing between functionals from $X^*$ and vectors from $X$. 
Also, we shall associate to any functional $f\in\big(W^{m,p'}_D(\Omega)\big)^*$ 
the distribution $f\lfloor_{\,\Omega}\,\in{\mathcal{D}}'(\Omega)$ defined by  
\begin{eqnarray}\label{Tgb-1}
{}_{{\mathcal{D}}'(\Omega)}
\big\langle f\lfloor_{\,\Omega}\,,\varphi\big\rangle_{{\mathcal{D}}(\Omega)}:=
{}_{\big(W^{m,p'}_D(\Omega)\bigr)^*}
\big\langle f,\varphi\big\rangle_{W^{m,p'}_D(\Omega)},\qquad
\forall\,\varphi\in{\mathscr{C}}^\infty_c(\Omega).
\end{eqnarray}
The reader is alerted to the fact that, while linear and continuous,  
\begin{eqnarray}\label{IYHN0987}
\mbox{the assignment }\,\,\big(W^{m,p'}_D(\Omega)\big)^*\ni f\longmapsto 
f\lfloor_{\,\Omega}\,\in{\mathcal{D}}'(\Omega)\,\,\mbox{ is not injective},
\end{eqnarray}
generally speaking (a remarkable exception is when $D=\partial\Omega$; cf.
$(5)$ in Lemma~\ref{TYRD-f4f}, \eqref{azTagbM-W2} and \eqref{azTagbM-W}). 
To see this, assume for a moment that $\Omega$ is an $(\varepsilon,\delta)$-domain in 
${\mathbb{R}}^n$ with the property that its boundary is $(n-1)$-Ahlfors 
regular, and set $\sigma:={\mathcal{H}}^{n-1}\lfloor\partial\Omega$. Also, suppose 
that $D$ is a proper closed subset of $\partial\Omega$. Then there exists a 
nontrivial function $g\in L^p(\partial\Omega,\sigma)$ with support in 
$\partial\Omega\setminus D$, and we define the functional
$f\in\big(W^{m,p'}_D(\Omega)\big)^*$ by setting
\begin{eqnarray}\label{IYavvTGq}
{}_{\big(W^{m,p'}_D(\Omega)\bigr)^*}
\big\langle f,v\big\rangle_{W^{m,p'}_D(\Omega)}
:=\int_{\partial\Omega}g\,{\mathscr{R}}_{\Omega\to\partial\Omega}^{(1)}v\,d\sigma,
\qquad\forall\,v\in W^{m,p'}_D(\Omega).
\end{eqnarray}
By Theorem~\ref{NIceTRace}, the above functional is then well-defined and nonzero, 
but it is clear that $f\lfloor_{\,\Omega}$ is zero as a distribution. Hence,
\eqref{IYHN0987} holds, as claimed. A heuristic, yet suggestive, way of expressing 
this is by saying that arbitrary distributions in $\Omega$ may have multiple
extensions to functionals in $\big(W^{m,p'}_D(\Omega)\bigr)^*$.

This discussion sets the stage for making the following definition.

\begin{definition}\label{Ian-Yab679}
Let $\Omega$ be an arbitrary open subset of ${\mathbb{R}}^n$ 
and assume that ${\mathcal{L}}$ is as in \eqref{coHGB2}-\eqref{coHGB2.B}, for 
some $m\in{\mathbb{N}}$. Also, suppose that $D$ is a closed subset of 
$\partial\Omega$ and fix $p,p'\in(1,\infty)$ with $1/p+1/p'=1$. 
Then, given $u\in W^{m,p}_D(\Omega)$ and $f\in\big(W^{m,p'}_D(\Omega)\big)^*$ 
satisfying the (necessary) compatibility condition
\begin{eqnarray}\label{Tgb-2}
{\mathcal{L}}u=f\lfloor_{\,\Omega}\,\,\,\mbox{ in }\,\,\,{\mathcal{D}}'(\Omega),
\end{eqnarray}
abbreviate by 
\begin{eqnarray}\label{Tgb-3}
\partial^A_\nu (u,f)=0\,\,\mbox{ on }\,\,\partial\Omega\setminus D
\end{eqnarray}
the demand that 
\begin{eqnarray}\label{Tgb-4}
(-1)^m\sum_{|\alpha|=|\beta|=m}
\int_{\Omega}\Bigl\langle A_{\alpha\beta}\,\partial^\beta u,
\partial^\alpha v\Bigr\rangle\,d{\mathscr{L}}^n
={}_{\big(W^{m,p'}_D(\Omega)\bigr)^*}
\big\langle f,v\big\rangle_{W^{m,p'}_D(\Omega)}
\qquad\forall\,v\in{\mathscr{C}}^\infty_D(\Omega).
\end{eqnarray}
\end{definition}

\noindent We wish to stress that, in \eqref{Tgb-3}, the symbol 
``$\partial^A_\nu (u,f)$" is not defined individually, but rather we 
assign a meaning to the condition 
``$\partial^A_\nu (u,f)=0\,\,\mbox{ on }\,\,\partial\Omega\setminus D$"
as a block, through \eqref{Tgb-4}.

A few comments pertaining to the nature of Definition~\ref{Ian-Yab679} are
in order. First, the fact that the compatibility condition \eqref{Tgb-2} is
necessary if \eqref{Tgb-4} is to hold, is readily seen by specializing 
\eqref{Tgb-4} to the case when $v\in{\mathscr{C}}^\infty_c(\Omega)$ and 
keeping in mind \eqref{Tgb-1}. Second, by density, \eqref{Tgb-4} is equivalent
to the condition that 
\begin{eqnarray}\label{Tgb-5}
(-1)^m\sum_{|\alpha|=|\beta|=m}
\int_{\Omega}\Bigl\langle A_{\alpha\beta}\,\partial^\beta u,
\partial^\alpha v\Bigr\rangle\,d{\mathscr{L}}^n
={}_{\big(W^{m,p'}_D(\Omega)\bigr)^*}
\big\langle f,v\big\rangle_{W^{m,p'}_D(\Omega)}
\qquad\forall\,v\in W^{m,p'}_D(\Omega).
\end{eqnarray}
Third, if $f$ is more regular than originally assumed, say if 
$f\in L^p(\Omega,{\mathscr{L}}^n)$, then \eqref{Tgb-5} further becomes 
equivalent to \eqref{con-pTi.3}. In particular, \eqref{Tgb-3} reduces to 
condition \eqref{con-pTi.4}, interpreted poitwise, in the case when all
objects involved are regular enough (as in the earlier discussion, pertaining 
to \eqref{con-prop2}-\eqref{con-pTi.3}).

However, in general, condition \eqref{Tgb-3} is not an ordinary 
generalization of the demand that \eqref{con-pTi.4} holds in a pointwise sense
(with $\partial^A_\nu u$ as in \eqref{con-comp}). In fact, it is more appropriate 
to regard the former as a ``renormalization'' of the latter, in a fashion that 
depends strongly on the choice of an extension of the distribution
${\mathcal{L}}u\in{\mathcal{D}}'(\Omega)$ to a functional 
$f\in\big(W^{m,p'}_D(\Omega)\big)^*$ (a phenomenon which may be   
traced back to \eqref{IYHN0987} and the subsequent discussion). 
Indeed, if $f_i\in\big(W^{m,p'}_D(\Omega)\big)^*$, $i=1,2$ are two 
such extensions of ${\mathcal{L}}u\in{\mathcal{D}}'(\Omega)$, in the sense that
\begin{eqnarray}\label{Tgb-2Eac}
{\mathcal{L}}u=f_i\lfloor_{\,\Omega}\,\,\,\mbox{ in }\,\,\,{\mathcal{D}}'(\Omega),
\,\,\,\mbox{ for }\,\,\,i=1,2.
\end{eqnarray}
then the validity of \eqref{Tgb-4} with $f=f_1$ does not necessarily imply 
the validity of \eqref{Tgb-4} with $f=f_2$. It is precisely for this reason
that, in contrast to \eqref{con-pTi.4} interpreted as a pointwise condition
in the sense of \eqref{con-comp}, the notation in \eqref{Tgb-3} reflects the 
fact that the functional $f$ plays a basic role in this case.

\vskip 0.10in

Definition~\ref{Ian-Yab679} is now employed in the formulation of the mixed 
boundary value problem in the next proposition.

\begin{proposition}\label{Tabb-Than}
Let $\Omega$ be an arbitrary open subset of ${\mathbb{R}}^n$ 
and assume that ${\mathcal{L}}$ is as in \eqref{coHGB2}-\eqref{coHGB2.B}, for 
some $m\in{\mathbb{N}}$. Also, suppose that $D$ is a closed subset of 
$\partial\Omega$ and fix $p,p'\in(1,\infty)$ with $1/p+1/p'=1$. In this context, 
for an arbitrary $f\in\big(W^{m,p'}_D(\Omega)\big)^*$ consider the mixed 
boundary value problem 
\begin{eqnarray}\label{Tgb-5aa}
\left\{
\begin{array}{l}
{\mathcal{L}}u=f\lfloor_{\,\Omega}\,\,\mbox{ in }\,\,{\mathcal{D}}'(\Omega),
\\[6pt]
u\in W^{m,p}_D(\Omega),
\\[6pt]
\partial^A_\nu(u,f)=0\,\,\mbox{ on }\,\,\partial\Omega\setminus D,
\end{array}
\right.
\end{eqnarray}
where the last condition is understood in the sense of Definition~\ref{Ian-Yab679}.
Finally, define the (linear and bounded) operator 
\begin{eqnarray}\label{IOH-aci}
T_{\mathcal{L}}:W^{m,p}_D(\Omega)\longrightarrow\big(W^{m,p'}_D(\Omega)\big)^*
\end{eqnarray}
by setting 
\begin{eqnarray}\label{IOH-aciF}
\begin{array}{c}
\displaystyle
{}_{\big(W^{m,p'}_D(\Omega)\big)^*}
\big\langle T_{\mathcal{L}}u,v\bigr\rangle_{W^{m,p}_D(\Omega)}
:=(-1)^m\sum\limits_{|\alpha|=|\beta|=m}
\int_{\Omega}\Bigl\langle A_{\alpha\beta}\,\partial^\beta u,
\partial^\alpha v\Bigr\rangle\,d{\mathscr{L}}^n,
\\[16pt]
\mbox{for every }\,\,\,u\in W^{m,p}_D(\Omega)\,\,\,\mbox{ and }\,\,\,
v\in W^{m,p'}_D(\Omega).
\end{array}
\end{eqnarray}

Then the mixed boundary value problem \eqref{Tgb-5aa} has a solution 
if and only if $f\in{\rm Im}\,T_{\mathcal{L}}$, and the solution is 
unique up to functions in ${\rm Ker}\,T_{\mathcal{L}}$. 

In particular, the mixed boundary value problem \eqref{Tgb-5aa} is well-posed 
(for arbitrary data in the space 
$\big(W^{m,p'}_D(\Omega)\big)^*$) if and only if the operator 
$T_{\mathcal{L}}$ in \eqref{IOH-aci}-\eqref{IOH-aciF} is invertible.
\end{proposition}

\begin{proof}
After unraveling notation, given an arbitrary functional 
$f\in\big(W^{m,p'}_D(\Omega)\big)^*$ the existence of a function 
$u\in W^{m,p}_D(\Omega)$ solving the mixed boundary value problem 
\eqref{Tgb-5aa} comes down to finding some $u\in W^{m,p}_D(\Omega)$ such that 
\begin{eqnarray}\label{Tgb-7}
{}_{\big(W^{m,p'}_D(\Omega)\bigr)^*}
\big\langle T_{\mathcal{L}}u,v\big\rangle_{W^{m,p'}_D(\Omega)}
={}_{\big(W^{m,p'}_D(\Omega)\bigr)^*}
\big\langle f,v\big\rangle_{W^{m,p'}_D(\Omega)}
\qquad\forall\,v\in{\mathscr{C}}^\infty_D(\Omega).
\end{eqnarray}
Since, by design, ${\mathscr{C}}^\infty_D(\Omega)$ is dense in 
$W^{m,p'}_D(\Omega)$ and since the operator $T_{\mathcal{L}}$ in 
\eqref{IOH-aci}-\eqref{IOH-aciF} is continuous, this is further 
equivalent to finding $u\in W^{m,p}_D(\Omega)$ solving 
\begin{eqnarray}\label{Tgb-8}
T_{\mathcal{L}}u=f\,\,\mbox{ in }\,\,\big(W^{m,p'}_D(\Omega)\big)^*.
\end{eqnarray}
From this, all desired conclusions follow.
\end{proof}

Proposition~\ref{Tabb-Than} highlights the relevance of the functional analytic 
properties of $T_{\mathcal{L}}$ in \eqref{IOH-aci}-\eqref{IOH-aciF} 
from the perspective of the solvability of the mixed boundary value problem
\eqref{Tgb-5aa}. In order to state our main Fredholm solvability result for the
mixed boundary value problem \eqref{Tgb-5aa}, let us agree that, 
for each open set $\Omega\subseteq{\mathbb{R}}^n$ and each $m\in{\mathbb{N}}$,
\begin{eqnarray}\label{autYYY}
{\mathcal{P}}_m(\Omega):=\big\{P\big|_{\Omega}:\,
P\,\,\mbox{ complex polynomial of degree }\leq m\,\,\mbox{ in }\,\,{\mathbb{R}}^n\big\}.
\end{eqnarray}
Here is the main well-posedness result in this paper. 

\begin{theorem}[Well-posedness of the higher-order mixed boundary problem]\label{yaUNDJ}
Let $\Omega$ be a bounded, connected, open, nonempty, subset of ${\mathbb{R}}^n$, 
$n\geq 2$, and suppose that $D$ is a closed subset of $\partial\Omega$ which is 
$d$-Ahlfors regular for some $d\in(n-2,n)$. In addition, assume that $\Omega$ is 
locally an $(\varepsilon,\delta)$-domain near $\partial\Omega\setminus D$. 
Next, consider an $M\times M$ divergence-form system ${\mathcal{L}}$ of order 
$2m$, as in \eqref{coHGB2}-\eqref{coHGB2.B}, for some $m\in{\mathbb{N}}$, and 
suppose that ${\mathcal{L}}$ satisfies the strong ellipticity condition 
asserting that there exists $\kappa>0$ such that
\begin{eqnarray}\label{S-Ellip}
\begin{array}{c}
\displaystyle{
{\rm Re}\left[\sum_{|\alpha|=|\beta|=m}\sum_{i,j=1}^M a^{\alpha\beta}_{ij}(x)
\zeta^\alpha_i\overline{\zeta^\beta_j}\right]
\geq\kappa\sum_{|\alpha|=m}\sum_{i=1}^M\frac{\alpha!}{m!}\,|\zeta^\alpha_i|^2}
\,\,\mbox{ for ${\mathscr{L}}^n$-a.e. $x\in\Omega$}, 
\\[20pt]
\displaystyle{
\mbox{and for all families of complex numbers }\,\,\big(\zeta^\alpha_i\big)
_{\stackrel{|\alpha|=m}{1\leq i\leq M}}}.
\end{array}
\end{eqnarray}

Then there exists $p_\ast\in(2,\infty)$, depending only on $n,m,M,d,\Omega,D,A,\kappa$, 
with the property that, whenever 
\begin{eqnarray}\label{S-EllPka}
\frac{p_\ast}{p_\ast-1}<p<p_\ast,
\end{eqnarray}
for each functional
\begin{eqnarray}\label{S-EllP877}
f\in\big(W^{m,p'}_D(\Omega)\big)^*\,\,\,\mbox{ where }\,\,1/p+1/p'=1,
\end{eqnarray}
the mixed boundary value problem 
\begin{eqnarray}\label{Tgb-5bb}
\left\{
\begin{array}{l}
{\mathcal{L}}u=f\lfloor_{\,\Omega}\,\,\mbox{ in }\,\,{\mathcal{D}}'(\Omega),
\\[6pt]
u\in W^{m,p}_D(\Omega),
\\[6pt]
\partial^A_\nu(u,f)=0\,\,\mbox{ on }\,\,\partial\Omega\setminus D,
\end{array}
\right.
\end{eqnarray}
(with the last condition understood in the sense of Definition~\ref{Ian-Yab679})
is well-posed when $D\not=\emptyset$. 

Moreover, in the case when $D=\emptyset$, the problem \eqref{Tgb-5bb}
has a solution if and only if 
\begin{eqnarray}\label{S-EYnn}
{}_{\big(W^{m,p'}_D(\Omega)\big)^*}\langle f,v\rangle_{W^{m,p'}_D(\Omega)}=0,
\qquad\forall\,v\in{\mathcal{P}}_{m-1}(\Omega),
\end{eqnarray}
in which scenario solutions of \eqref{Tgb-5bb} are unique up to 
functions from ${\mathcal{P}}_{m-1}(\Omega)$.
\end{theorem}

Note that the membership $u\in W^{m,p}_D(\Omega)$ entails 
${\mathscr{R}}^{(m)}_{\Omega\to D}\,u=0$ 
at ${\mathcal{H}}^d$-a.e. point on $D$, by Corollary~\ref{Kae.Yab7b6} 
(and, in fact, the latter is equivalent to $u\in W^{m,p}_D(\Omega)$
under the background assumption that $u$ is in $W^{k,p}(\Omega)$ to being with, 
in the context specified in Theorem~\ref{Kance.745}). As such, problem 
\eqref{Tgb-5bb} imposes the homogeneous Dirichlet boundary 
condition ${\mathscr{R}}^{(m)}_{\Omega\to D}\,u=0$ on $D$ and the homogeneous
Neumann boundary condition $\partial^A_\nu(u,f)=0$ on $\partial\Omega\setminus D$, 
thus justifying the terminology ``mixed boundary value problem".

\vskip 0.08in
\begin{proof}[Proof of Theorem~\ref{yaUNDJ}]
As a preamble, we first claim that if $\Omega,D$ are as in the statement of the theorem
then whenever $m\in{\mathbb{N}}$ and $\max\,\{1,n-d\}<p<\infty$, we have 
\begin{eqnarray}\label{S-EllPPP.5tG}
{\mathcal{P}}_{m-1}(\Omega)\cap W^{m,p}_D(\Omega) 
=\left\{
\begin{array}{ll}
\{0\} & \mbox{ if }\,D\not=\emptyset,
\\[4pt]
{\mathcal{P}}_{m-1}(\Omega)& \mbox{ if }\,D=\emptyset.
\end{array}
\right.
\end{eqnarray}
To justify this, observe that in the case when $D\not=\emptyset$ 
we necessarily have ${\mathcal{H}}^d(D)>0$, given that $D$ is assumed to be 
$d$-Ahlfors regular. On the other hand, from \eqref{Ver-S2TG.2U}-\eqref{Ver-S2TG.2}
and \eqref{amZbYJnab} we deduce that
\begin{eqnarray}\label{amZbYJnab.6YHN}
\begin{array}{c}
\big\{\partial^\alpha P\big|_D\big\}_{|\alpha|\leq m-1}
={\mathscr{R}}_{\Omega\to D}^{(k)}v=(0,...,0)
\,\,\mbox{ at ${\mathcal{H}}^d$-a.e. point on }\,\,D,
\\[8pt]
\mbox{if }\,\,v\in W^{m,p}_D(\Omega)\,\,\mbox{ is of the form }\,\,
v=P\big|_{\Omega}\,\,\mbox{ for some }\,\,P\in{\mathcal{P}}_{m-1}({\mathbb{R}}^n).
\end{array}
\end{eqnarray}
All together, the above analysis show that ${\mathcal{P}}_{m-1}(\Omega)
\cap W^{m,p}_D(\Omega)=\{0\}$ if $D\not=\emptyset$ (given that $\Omega$ is 
assumed to be connected). Finally, if $D=\emptyset$
then the fact that ${\mathcal{P}}_{m-1}(\Omega)\cap W^{m,p}_D(\Omega)
={\mathcal{P}}_{m-1}(\Omega)$ is clear from \eqref{Itebn09.bFD}, 
Lemma~\ref{LLDEnse}, and the assumption that $\Omega$ is bounded. 

To proceed in earnest with the proof of the well-posedness of the mixed boundary 
problem \eqref{Tgb-5bb}, recall first the operator $T_{\mathcal{L}}$
introduced in \eqref{IOH-aci}-\eqref{IOH-aciF}. When considered in the context
\begin{eqnarray}\label{IanUHrr}
T_{\mathcal{L}}:W^{m,2}_D(\Omega)\longrightarrow\big(W^{m,2}_D(\Omega)\big)^*,
\end{eqnarray}
the ellipticity condition \eqref{S-Ellip} implies that this mapping has the 
property that, for every function $u=(u_j)_{1\leq j\leq M}\in W^{m,2}_D(\Omega)$,
\begin{eqnarray}\label{UaiGEh}
\kappa\sum_{|\alpha|=m}\sum_{i=1}^M\frac{\alpha!}{m!}
\int_{\Omega}|\partial^\alpha u_i|^2\,d{\mathscr{L}}^n
&\leq & {\rm Re}\left[\sum\limits_{|\alpha|=|\beta|=m}\sum\limits_{i,j=1}^M 
\int_{\Omega}a^{\alpha\beta}_{ij}\,\partial^\beta u_j
\overline{\partial^\alpha u_i}\,d{\mathscr{L}}^n\right]
\nonumber\\[4pt]
&\leq &\Big|
{}_{\big(W^{m,2}_D(\Omega)\big)^*}
\big\langle T_{\mathcal{L}}u,\overline{u}\bigr\rangle_{W^{m,2}_D(\Omega)}\Big|
\nonumber\\[4pt]
&\leq & \|T_{\mathcal{L}}u\|_{\big(W^{m,2}_D(\Omega)\big)^*}
\|u\|_{W^{m,2}(\Omega)}
\nonumber\\[4pt]
&\leq & (4\theta)^{-1}\|T_{\mathcal{L}}u\|^2_{\big(W^{m,2}_D(\Omega)\big)^*}
+\theta\|u\|^2_{W^{m,2}(\Omega)},
\end{eqnarray}
for each $\theta\in(0,\infty)$. Choosing $\theta$ small, \eqref{UaiGEh} ultimately
shows that there exists some finite constant $C>0$, independent of $u$, such that
\begin{eqnarray}\label{UaiGEh.7}
\|u\|_{W^{m,2}(\Omega)}
\leq C\|T_{\mathcal{L}}u\|_{\big(W^{m,2}_D(\Omega)\big)^*}
+C\|u\|_{W^{m-1,2}(\Omega)},\qquad\forall\,u\in W^{m,2}_D(\Omega).
\end{eqnarray}
Given that the embedding $W^{m,2}_D(\Omega)\hookrightarrow W^{m-1,2}(\Omega)$
is compact, estimate \eqref{UaiGEh.7} implies that $T_{\mathcal{L}}$ 
in \eqref{IanUHrr} is bounded from below modulo compact operators.
Granted this, standard functional analysis gives that $T_{\mathcal{L}}$
in \eqref{IanUHrr} has closed range and finite dimensional kernel. Since 
\begin{eqnarray}\label{OI-Unn}
\big(T_{\mathcal{L}}\big)^*=T_{\mathcal{L}^*}
\end{eqnarray}
and the adjoint ${\mathcal{L}}^*$ of ${\mathcal{L}}$ also satisfies the ellipticity 
condition \eqref{S-Ellip} (written for its tensor coefficient), it follows 
that the adjoint of $T_{\mathcal{L}}$ from \eqref{IanUHrr} also enjoys the 
aforementioned properties. That is, 
\begin{eqnarray}\label{IanUHrr.6Y}
\big(T_{\mathcal{L}}\big)^\ast:W^{m,2}_D(\Omega)
\longrightarrow\big(W^{m,2}_D(\Omega)\big)^*
\,\,\mbox{ has closed range and finite dimensional kernel},
\end{eqnarray}
since the spaces \eqref{uig-adjb-2} are reflexive for every $p\in(1,\infty)$.
All together, this analysis proves that 
\begin{eqnarray}\label{IanUHrr.7H}
T_{\mathcal{L}}:W^{m,2}_D(\Omega)
\longrightarrow\big(W^{m,2}_D(\Omega)\big)^*
\,\,\mbox{ is a Fredholm operator}.
\end{eqnarray}
It is also implicit in estimate \eqref{UaiGEh} (cf. the third inequality there) that 
if $u\in W^{m,2}_D(\Omega)$ is such that 
$T_{\mathcal{L}}u=0\in\big(W^{m,2}_D(\Omega)\big)^*$ then necessarily 
$u$ is a polynomial of degree $\leq m-1$ in $\Omega$.
Conversely, as seen from \eqref{IOH-aciF}, any such function is mapped 
by $T_{\mathcal{L}}$ to zero. Hence, for the operator \eqref{IanUHrr},
\begin{eqnarray}\label{autYiiR}
{\rm Ker}\,T_{\mathcal{L}}={\mathcal{P}}_{m-1}(\Omega)\cap W^{m,2}_D(\Omega).
\end{eqnarray}
From this and \eqref{OI-Unn} we also deduce that 
\begin{eqnarray}\label{autYiiR.2}
{\rm Ker}\,\big(T_{\mathcal{L}}\big)^\ast
={\mathcal{P}}_{m-1}(\Omega)\cap W^{m,2}_D(\Omega).
\end{eqnarray}
In concert, \eqref{IanUHrr.7H}-\eqref{autYiiR.2} show that
\begin{eqnarray}\label{IanUHrr.7H.2}
T_{\mathcal{L}}:W^{m,2}_D(\Omega)
\longrightarrow\big(W^{m,2}_D(\Omega)\big)^*
\,\,\mbox{ is Fredholm with index zero}.
\end{eqnarray}

At this stage, from \eqref{IanUHrr.7H.2}, the fact that the operator \eqref{IOH-aci}
is linear and bounded, \eqref{GYaYn16H}-\eqref{GYaYn16H.Re}, and the stability of 
the quality of being Fredholm with index zero, as well as the stability of null-space 
on complex interpolation scales which are nested (with respect to the scale parameter),
and the assumption that $n-d<2$ (cf., \cite{CaSa}, \cite{KMM}, \cite{KaMi}, 
\cite{ViVi}), we deduce that there exists $p_\ast\in(2,\infty)$,
depending only on $n,m,M,d,\Omega,D,A,\kappa$, with the property that whenever 
\eqref{S-EllPka} holds and $1/p+1/p'=1$, 
\begin{eqnarray}\label{UIhan9HB}
\begin{array}{c}
T_{\mathcal{L}}:W^{m,p}_D(\Omega)
\longrightarrow\big(W^{m,p'}_D(\Omega)\big)^*
\,\,\mbox{ is a Fredholm operator with index zero, and}
\\[4pt]
\mbox{both its kernel and the kernel of its adjoint coincide with }\,\,
{\mathcal{P}}_{m-1}(\Omega)\cap W^{m,p}_D(\Omega).
\end{array}
\end{eqnarray}
Based on this, \eqref{S-EllPPP.5tG}, and Proposition~\ref{Tabb-Than}, 
all claims in the statement of the theorem now readily follow.
\end{proof}

The case $D=\partial\Omega$ of Theorem~\ref{yaUNDJ} deserves to be stated 
separately since this corresponds to the well-posedness of the 
inhomogeneous Dirichlet boundary value problem for higher-order systems 
in a very general analytic and geometric measure theoretic setting. 
An artifact of the choice $D=\partial\Omega$ which deserves special mention 
is the fact that condition $d\in(n-2,n)$, which would normally carry over from 
the formulation of Theorem~\ref{yaUNDJ}, naturally readjusts to $d\in[n-1,n)$. 
To see this, recall a general version of the classical isoperimetric inequality 
proved by H.\,Federer in \cite[3.2.43-3.2.44, p.\,278]{Fe}, according to which
\begin{eqnarray}\label{Fed-H1}
{\mathscr{L}}^n(\overline{E})
\leq\frac{1}{n\omega_{n-1}^{n-1}}\,{\mathcal{H}}^{n-1}(\partial E)^{n/(n-1)},
\qquad\forall\,E\subset{\mathbb{R}}^n
\,\,\mbox{ with }\,\,{\mathscr{L}}^n(\overline{E})<+\infty,
\end{eqnarray}
where $\omega_{n-1}$ stands for the surface area of the unit sphere 
in ${\mathbb{R}}^n$. Hence, in the case of a bounded, open, nonempty 
set $\Omega\subseteq{\mathbb{R}}^n$ such that $\partial\Omega$ is 
$d$-Ahlfors regular, we simultaneously have ${\mathcal{H}}^{n-1}(\partial\Omega)>0$
and ${\mathcal{H}}^d(\partial\Omega)<+\infty$. Together, these two conditions
imply $d\geq n-1$, which accounts for the adjustment mentioned earlier. 

\begin{theorem}[Well-posedness of the higher-order inhomogeneous Dirichlet
problem]\label{yTganDJ}
Assume that $\Omega$ is a bounded, open, nonempty subset of ${\mathbb{R}}^n$, 
$n\geq 2$, whose boundary is $d$-Ahlfors regular for some $d\in[n-1,n)$. 
Also, consider an $M\times M$ divergence-form system ${\mathcal{L}}$ of order 
$2m$ with bounded measurable coefficients, as in \eqref{coHGB2}-\eqref{coHGB2.B}, 
for some $m\in{\mathbb{N}}$, and suppose that ${\mathcal{L}}$ satisfies the 
strong ellipticity condition \eqref{S-Ellip} for some $\kappa>0$.

Then there exists $p_\ast\in(2,\infty)$, depending only on $n,m,M,d,\Omega,A,\kappa$ 
with the property that the inhomogeneous Dirichlet boundary value problem 
\begin{eqnarray}\label{Tgb-5Ohab}
\left\{
\begin{array}{l}
{\mathcal{L}}u=f\in W^{-m,p}(\Omega),
\\[6pt]
u\in\mathring{W}^{m,p}(\Omega),
\end{array}
\right.
\end{eqnarray}
is well-posed whenever $\frac{p_\ast}{p_\ast-1}<p<p_\ast$.
\end{theorem}

\begin{proof}
Essentially, this is a consequence of Theorem~\ref{yaUNDJ} employed 
with $D=\partial\Omega$, and part $(5)$ in Lemma~\ref{TYRD-f4f}.
The only matter which requires further clarification, given that we are not 
assuming in the current case that $\Omega$ is necessarily connected, is the fact that
\begin{eqnarray}\label{FYU-UIG}
\big\{u\in\mathring{W}^{k,p}(\Omega):\,
u\,\,\mbox{ locally a polynomial of degree }\leq m-1\big\}=\{0\}.
\end{eqnarray}
This, however, is readily seen from an $(m-1)$-fold application of Poincar\'e's 
inequality (which is valid in the space $\mathring{W}^{k,p}(\Omega)$, given its
definition in \eqref{azTagbM}).  
\end{proof}

Both Theorem~\ref{yaUNDJ} and Theorem~\ref{yTganDJ} are sharp, in the sense that 
the membership of $p$ to a small neighborhood of $2$ is a necessary condition, 
even when $\Omega\subseteq{\mathbb{R}}^n$ is a bounded ${\mathscr{C}}^\infty$ domain, 
if the coefficients $A_{\alpha\beta}$ are merely bounded and measurable.
To treat both cases simultaneously, consider the (permissible) scenario in 
which $D=\partial\Omega$. In the case $n=M=2$ and $m=1$ a relevant counterexample 
has been given by N.\,Meyers in \S\,{5} of \cite{Mey}. Specifically, take 
$\Omega:=\{(x_1,x_2)\in{\mathbb{R}}^2:\,x_1^2+x_2^2<1\}$ and consider 
the tensor coefficient given by   
\begin{eqnarray}\label{TE-YGHC.1}
\begin{array}{l}
A_{11}(x_1,x_2)=1-(1-\mu^2)x_2^2(x_1^2+x_2^2)^{-1},
\\[6pt]
A_{12}(x_1,x_2)=A_{21}(x_1,x_2)=(1-\mu^2)x_1x_2(x_1^2+x_2^2)^{-1},
\\[6pt]
A_{22}(x_1,x_2)=1-(1-\mu^2)x_1^2(x_1^2+x_2^2)^{-1},
\end{array}
\qquad\forall\,(x,y)\in\Omega\setminus\{(0,0)\},
\end{eqnarray}
where $\mu\in(0,1)$ is a fixed parameter. Define the scalar operator
${\mathcal{L}}u:=\sum\limits_{\alpha,\beta=1,2}
\partial_\alpha\big(A_{\alpha\beta}(x_1,x_2)\partial_\beta u\big)$ in $\Omega$. 
Note that the $A_{\alpha\beta}$'s belong to $L^\infty(\Omega,{\mathscr{L}}^2)$
and a direct calculation shows that
\begin{eqnarray}\label{S-EHna}
\sum_{\alpha,\beta=1,2}A_{\alpha\beta}(x_1,x_2)
\zeta^\alpha\zeta^\beta=|\zeta|^2-(1-\mu^2)
\frac{(x_1\zeta^2-x_2\zeta^1)^2}{x_1^2+x_2^2}
\geq\mu^2|\zeta|^2,
\end{eqnarray}
for each $\zeta=\big(\zeta^1,\zeta^2\big)\in{\mathbb{R}}^2$
and $(x_1,x_2)\in\Omega\setminus\{0\}$. Hence, ${\mathscr{L}}$ satisfies
the strong ellipticity condition \eqref{S-Ellip}. To proceed, introduce the function 
\begin{eqnarray}\label{TE-Yah}
v(x_1,x_2):=x_1(x_2^2+x_2^2)^{(\mu-1)/2}\in L^\infty(\Omega,{\mathscr{L}}^2)
\cap{\mathscr{C}}^\infty\big(\overline{\Omega}\setminus\{0\}\big).
\end{eqnarray}
A straightforward calculation shows that ${\mathcal{L}}v=0$ near origin. 
Also, fix $\phi\in{\mathscr{C}}^\infty_c(\Omega)$ so that $\phi\equiv 1$ near 
origin, and set $u:=\phi\,v$. It follows that 
\begin{eqnarray}\label{TE-YGHC.2}
u\in\mathring{W}^{1,2}(\Omega),
\quad f:={\mathcal{L}}u\in{\mathscr{C}}_c^\infty(\Omega),\quad
\big|(\nabla u)(x_1,x_2)\big|\approx(x_1^2+x_2^2)^{(\mu-1)/2}\,\,\mbox{ near }\,\,(0,0).
\end{eqnarray}
Consequently,
\begin{eqnarray}\label{TE-YGHC.3}
u\in W^{1,p}(\Omega)\Longleftrightarrow p<\frac{2}{1-\mu}.
\end{eqnarray}
In particular, the fact that $2/(1-\mu)\searrow 2$ as $\mu\searrow 0$ shows that 
that for each $p>2$ there exists $\mu\in(0,1)$ with the property that
the operator ${\mathcal{L}}:\mathring{W}^{1,p}(\Omega)\to W^{-1,p}(\Omega)$
fails to be an isomorphism. By duality, (${\mathcal{L}}$ is formally self-adjoint), 
the same type of conclusion holds for $p<2$. 

When $n\geq 3$, $m=1$, $N=n$, a counterexample may be produced by altering  
(in the spirit of \cite{NeSt72}) a construction of E.\,De Giorgi from \cite{DG}. 
Specifically, consider $\Omega:=\{x\in\mathbb{R}^n:\,|x|<1\}$ and, for 
each $\gamma\in\big[0,\frac{n}{2}\big)$ and $\alpha,\beta\in\{1,...,n\}$, 
let $A_{\alpha\beta}$ be the $n\times n$ matrix whose $(i,j)$-entry is 
\begin{eqnarray}\label{TYba97H}
a^{\alpha\beta}_{ij}(x):=\delta_{\alpha\beta}\delta_{ij}
+\frac{\gamma(n-\gamma)(n-2)^2}{(n-2\gamma)^2(n-1)^2}
\left[\delta_{i\alpha}+\frac{n}{n-2}\frac{x_ix_\alpha}{|x|^2}\right]
\left[\delta_{j\beta}+\frac{n}{n-2}\frac{x_jx_\beta}{|x|^2}\right],
\end{eqnarray}
for each $x\in\Omega\setminus\{0\}$.
Obviously, $a^{\alpha\beta}_{ij}\in L^\infty(\Omega,{\mathscr{L}}^n)$ and 
a straightforward calculation shows that
\begin{eqnarray}\label{S-ElYha}
\sum_{\alpha,\beta=1}^n\sum_{i,j=1}^n a^{\alpha\beta}_{ij}(x)
\zeta^\alpha_i\zeta^\beta_j
=|\zeta|^2+\frac{\gamma(n-\gamma)(n-2)^2}{(n-2\gamma)^2(n-1)^2}
\Bigl(\sum_{i=1}^n\zeta_i^i+\frac{n}{n-2}
\sum_{i,\alpha=1}^n\zeta_i^\alpha\frac{x_ix_\alpha}{|x|^2}\Bigr)^2
\end{eqnarray}
for each $\zeta=\big(\zeta^\alpha_i\big)_{1\leq \alpha,i\leq n}\in{\mathbb{R}}^{n^2}$
and $x\in\Omega\setminus\{0\}$. Given our assumptions on $\gamma$, 
it follows that the strong ellipticity condition holds:
\begin{eqnarray}\label{S-ElYhaii}
\sum_{\alpha,\beta=1}^n\sum_{i,j=1}^n a^{\alpha\beta}_{ij}(x)
\zeta^\alpha_i\zeta^\beta_j\geq|\zeta|^2\qquad
{\mathscr{L}}^n\mbox{-a.e. in }\Omega,\qquad\forall\,
\zeta=\big(\zeta^\alpha_i\big)_{1\leq \alpha,i\leq n}\in{\mathbb{R}}^{n^2}.
\end{eqnarray}
Now, the fact that $\gamma<n/2$ ensure that the function 
\begin{eqnarray}\label{S-El333}
u(x):=\frac{x}{|x|^{\gamma}}-x,\qquad\forall\,x\in\Omega\setminus\{0\},
\end{eqnarray}
belongs to $W^{1,2}(\Omega)$. Since by design $u\big|_{\partial\Omega}=0$, 
we deduce that actually $u\in\mathring{W}^{1,2}(\Omega)$. Furthermore, if
\begin{eqnarray}\label{S-El4v}
f:=(f_1,...,f_n)\,\,\,\mbox{ with }\,\,\,
f_i:=-\sum\limits_{\alpha=1}^n\sum\limits_{j=1}^n\partial_{\alpha}
a^{\alpha j}_{ij}\,\,\,\mbox{ for }\,\,\,1\leq i\leq n, 
\end{eqnarray}
then clearly 
\begin{eqnarray}\label{S-El6tf}
f\in\bigcap_{1<p<\infty} W^{-1,p}(\Omega),
\end{eqnarray}
while a direct computation shows that  
\begin{eqnarray}\label{S-ayYBB65}
\sum\limits_{\alpha,\beta=1}^n\partial_{\alpha}
\Bigl(A_{\alpha\beta}(x)\partial_\beta u\Bigr)=f\,\,\,\mbox{ in }\,\,\,
{\mathcal{D}}'(\Omega).
\end{eqnarray}
However, on the one hand $u\in W^{1,p}(\Omega)$ if and only if $p<n/\gamma$, 
while on the other hand $n/\gamma\searrow 2$ as $\gamma\nearrow n/2$. 

For $n\geq 2$ and higher-order operators we make use of an example originally 
due to V.G.~Maz'ya (cf. \cite{Maz-Ctr}). Specifically, when $m\in\mathbb{N}$ 
is even, consider the divergence-form operator of order $2m$ 
\begin{equation}\label{Maz-Op}
{\mathcal{L}}:=\Delta^{\frac{1}{2}m-1}{\mathcal L}_4\,\Delta^{\frac{1}{2}m-1}
\quad\mbox{in }\,\,\Omega:=\{x\in\mathbb{R}^n:\,|x|<1\},
\end{equation}
where ${\mathcal L}_4$ is the fourth-order operator  
\begin{eqnarray}\label{Maz-Op2}
{\mathcal L}_4 u & \!\!\!\!:=\!\!\!\! & a\,\Delta^2 u
+b\sum_{i,j=1}^n\Delta\Bigl(\frac{x_ix_j}{|x|^2}
\partial_i\partial_j\,u\Bigr)
+b\sum_{i,j=1}^n\partial_i\partial_j\Bigl(\frac{x_ix_j}{|x|^2}\,
\Delta\,u\Bigr)
\nonumber\\[6pt]
&&+c\sum_{i,j,k,l=1}^n\partial_k\partial_l\Bigl(\frac{x_ix_jx_kx_l}{|x|^4}
\partial_i\partial_j\,u\Bigr).
\end{eqnarray}
Obviously, the coefficients of ${\mathcal L}_4$ are bounded, and if the parameters 
$a,b,c\in\mathbb{R}$, $a>0$, are chosen such that $b^2<ac$ then ${\mathcal L}_4$ 
along with ${\mathcal L}=\Delta^{\frac{1}{2}m-1}{\mathcal L}_4\,\Delta^{\frac{1}{2}m-1}$ 
are strongly elliptic. Now, it has been observed in \cite{Maz-Ctr} that if 
\begin{equation}\label{Maz-Op3}
\theta:=2-\frac{n}{2}+\sqrt{\frac{n^2}{4}-\frac{(n-1)(bn+c)}{a+2b+c}},
\end{equation}
then the function $v(x):=|x|^{\theta+m-2}$ for each $x\in\Omega\setminus\{0\}$
belongs to $W^{m,2}(\Omega)$ and satisfies ${\mathcal{L}}v=0$ in
${\mathcal{D}}'(\Omega)$. Furthermore, $v$ is ${\mathscr{C}}^\infty$ 
in a neighborhood of $\partial\Omega$ and, as such, there exists 
a function $w\in{\mathscr{C}}^\infty(\overline{\Omega})$ with the property that
\begin{equation}\label{MaUan}
u:=v-w\in{\mathring{W}}^{m,2}(\Omega).
\end{equation}
Note that, by deign, ${\mathcal{L}}u=f$ in ${\mathcal{D}}'(\Omega)$, where 
\begin{eqnarray}\label{S-El6tftt}
f:=-{\mathcal{L}}w\in\bigcap_{1<p<\infty} W^{-m,p}(\Omega),
\end{eqnarray}
and $u\in W^{m,p}(\Omega)$ if and only if $v\in W^{m,p}(\Omega)$.
In order to focus on the veracity of the latter condition, we find it convenient
to specialize matters by taking $a:=(n-2)^2+\varepsilon$, $b:=n(n-2)$, $c:=n^2$, 
for some small $\varepsilon>0$. The strong ellipticity condition is satisfied, 
and the parameter $\theta$ from \eqref{Maz-Op3} becomes 
\begin{eqnarray}\label{ajBBBV}
\theta(\varepsilon)=2-\frac{n}{2}+
\frac{n\,\varepsilon^{1/2}}{2\sqrt{4(n-1)^2+\varepsilon}}. 
\end{eqnarray}
However, $v\in W^{m,p}(\Omega)$ if and only if 
$p<n/(2-\theta(\varepsilon))$, and the bound $n/(2-\theta(\varepsilon))$ 
approaches $2$ when $\varepsilon\to 0^{+}$. The bottom line is that range 
of $p$'s in the interval $(2,\infty)$ for which $u\in W^{m,p}(\Omega)$ 
shrinks to $2$ as $\varepsilon\to 0^{+}$. 

In \cite{Maz-Ctr} an analogous example was also constructed when $m>1$ 
is odd, starting with the sixth order operator 
\begin{eqnarray}\label{Maz-Op4}
{\mathcal L}_6\,u &\!\!\!:=\!\!\!& a\,\Delta^3u
+b\sum_{i,j,k,l=1}^n\partial_i\partial_j\partial_k\Bigl(\frac{x_ix_jx_kx_l}{|x|^4}
\Delta\partial_l\,u\Bigr)
+b\sum_{i,j,k,l=1}^n\Delta\partial_l\Bigl(\frac{x_ix_jx_kx_l}{|x|^4}
\partial_i\partial_j\partial_k\,u\Bigr)
\nonumber\\[6pt]
&&+c\sum_{i,j,k,l,r,t=1}^n\partial_i\partial_j\partial_k
\Bigl(\frac{x_ix_jx_kx_lx_rx_t}{|x|^6}\partial_l\partial_r\partial_t\,u\Bigr)
\end{eqnarray}
and then considering 
\begin{equation}\label{Maz-Op5}
{\mathcal L}:=\Delta^{\frac{m-3}{2}}{\mathcal L}_6\,\Delta^{\frac{m-3}{2}}
\quad\mbox{ in }\,\,\,\Omega:=\{x\in\mathbb{R}^n:\,|x|<1\}.
\end{equation}
For the choice $a:=(n-4)^2+\varepsilon$, $b:=(n-4)(n+2)$, $c:=(n+2)^2$, 
$\varepsilon>0$, the operator \eqref{Maz-Op5} is strongly elliptic and 
the function $v(x):=|x|^{\mu+m-3}$ for each $x\in\Omega\setminus\{0\}$
belongs to $W^{m,2}(\Omega)$ and satisfies ${\mathcal{L}}v=0$ in
${\mathcal{D}}'(\Omega)$ if $\mu=\mu(\varepsilon)$ given by 
\begin{equation}\label{Maz-Op6}
\mu:=3-\frac{n}{2}+\frac{(n+2)(n-4)}{2}
\sqrt{\frac{\varepsilon}{4(n-1)^2+\varepsilon}}.
\end{equation}
Moreover, $v\in W^{m,p}(\Omega)$ if and only if 
$p<n/(3-\mu(\varepsilon))$, and the bound $n/(3-\mu(\varepsilon))$ 
approaches $2$ when $\varepsilon\to 0^{+}$. With this in hand and proceeding 
as in the previous case, the same type of conclusion may be derived 
in the current setting as well. 

\vskip 0.10in

Moving on, we wish to extend the well-posedness result for the higher-order
inhomogeneous Dirichlet problem from Theorem~\ref{yTganDJ} by treating its 
fully inhomogeneous version, albeit in a more resourceful geometrical setting. 

\begin{theorem}[Well-posedness of the fully inhomogeneous higher-order Poisson 
problem]\label{yTganDJ.Yam}
Let $\Omega$ be a bounded $(\varepsilon,\delta)$-domain in ${\mathbb{R}}^n$, 
$n\geq 2$, with ${\rm rad}\,(\Omega)>0$, and whose boundary is $d$-Ahlfors regular 
for some $d\in[n-1,n)$. In addition, consider an $M\times M$ divergence-form system
${\mathcal{L}}$ of order $2m$ with bounded measurable coefficients, as 
in \eqref{coHGB2}-\eqref{coHGB2.B}, for some $m\in{\mathbb{N}}$, and 
suppose that ${\mathcal{L}}$ satisfies the strong ellipticity condition 
\eqref{S-Ellip} for some $\kappa>0$.

Then there exists $p_\ast\in(2,\infty)$ sufficiently close to $2$, 
which depends only on $n,m,M,d,\Omega,A,\kappa$, and with the property 
that the fully inhomogeneous Poisson problem 
\begin{eqnarray}\label{Tgb-5Ohab.LK}
\left\{
\begin{array}{l}
{\mathcal{L}}u=f\in W^{-m,p}(\Omega),
\\[6pt]
u\in W^{m,p}(\Omega),
\\[6pt]
{\mathscr{R}}^{(m)}_{\Omega\to\partial\Omega}u
=\dot{g}\in B^{p,p}_{m-(n-d)/p}(\partial\Omega),
\end{array}
\right.
\end{eqnarray}
is well-posed whenever $\frac{p_\ast}{p_\ast-1}<p<p_\ast$.
\end{theorem}

\begin{proof}
Let $p_\ast>2$ be as in Theorem~\ref{yTganDJ} (without loss of generality, 
it may be assumed that $p_\ast$ is sufficiently close to $2$), and fix some 
$p$ such that $\frac{p_\ast}{p_\ast-1}<p<p_\ast$. Suppose now that an arbitrary 
$f\in W^{-m,p}(\Omega)$ and $\dot{g}\in B^{p,p}_{m-(n-d)/p}(\partial\Omega)$ have
been given. Since in the current context Corollary~\ref{NIceTRace.CC} guarantees that 
the restriction operator ${\mathscr{R}}_{\Omega\to\partial\Omega}^{(m)}$ 
maps $W^{m,p}(\Omega)$ onto $B^{p,p}_{m-(n-d)/p}(\partial\Omega)$, it follows
that there exists $v\in W^{m,p}(\Omega)$ such that 
${\mathscr{R}}_{\Omega\to\partial\Omega}^{(m)}v=\dot{g}$. 
Moreover, as a consequence of the Open Mapping Theorem, it may be assumed 
that $\|v\|_{W^{m,p}(\Omega)}\leq C\|\dot{g}\|_{B^{p,p}_{m-(n-d)/p}(\partial\Omega)}$
for some finite constant $C>0$ independent of $\dot{g}$. If we now use   
the well-posedness statement in Theorem~\ref{yTganDJ} in order to 
solve the inhomogeneous Dirichlet boundary value problem 
\begin{eqnarray}\label{Tgb-5Ohab.A}
\left\{
\begin{array}{l}
{\mathcal{L}}w=f-{\mathcal{L}}v\in W^{-m,p}(\Omega),
\\[6pt]
w\in\mathring{W}^{m,p}(\Omega),
\end{array}
\right.
\end{eqnarray}
it follows that $u:=v+w$ solves the original problem \eqref{Tgb-5Ohab.LK}
(keeping in mind \eqref{Uan-Taev75}), and also obeys natural estimates.
Finally, uniqueness for \eqref{Tgb-5Ohab.LK} is a consequence of the uniqueness part 
in Theorem~\ref{yTganDJ} and \eqref{Uan-Taev75}.
\end{proof}

The well-posedness of the fully inhomogeneous Poisson problem \eqref{Tgb-5Ohab.LK}
has been established in \cite{MaMiSh} in the context of weighted Sobolev
spaces in Lipschitz domains, for strongly elliptic higher-order systems with 
bounded measurable coefficients, when the integrability parameter $p$ belongs to
a small neighborhood of $2$. In \cite{MaMiSh}, the authors have also proved
that problem \eqref{Tgb-5Ohab.LK} continues to be well-posed if, additionally,
the outward unit normal $\nu$ to the Lipschitz domain $\Omega$ belongs to 
${\rm VMO}(\partial\Omega)$, the Sarason space of functions with vanishing 
mean oscillations and the coefficients of the operator ${\mathcal{L}}$ are 
in ${\rm VMO}(\Omega)$.

We conclude this section with the following solvability result for the 
inhomogeneous Neumann problem in $(\varepsilon,\delta)$-domains.

\begin{theorem}[Well-posedness of the higher-order Neumann boundary
problem]\label{yaUNDJ.taF}
Let $\Omega$ be a bounded, connected $(\varepsilon,\delta)$-domain in ${\mathbb{R}}^n$, 
$n\geq 2$, with ${\rm rad}\,(\Omega)>0$, and suppose that ${\mathcal{L}}$ is an 
$M\times M$ divergence-form system of order $2m$, as in 
\eqref{coHGB2}-\eqref{coHGB2.B}, for some $m\in{\mathbb{N}}$, 
which satisfies the strong ellipticity condition \eqref{S-Ellip} for some constant
$\kappa>0$.

Then there exists $p_\ast\in(2,\infty)$, depending only on $n,m,M,\Omega,D,A,\kappa$, 
with the following significance. If $\frac{p_\ast}{p_\ast-1}<p<p_\ast$ 
then for each functional $f\in\big(W^{m,p'}(\Omega)\big)^*$ (where $1/p+1/p'=1$), 
the inhomogeneous Neumann boundary value problem 
\begin{eqnarray}\label{Tgb-5bb.Kan}
\left\{
\begin{array}{l}
{\mathcal{L}}u=f\lfloor_{\,\Omega}\,\,\mbox{ in }\,\,{\mathcal{D}}'(\Omega),
\\[6pt]
u\in W^{m,p}(\Omega),
\\[6pt]
\partial^A_\nu(u,f)=0\,\,\mbox{ on }\,\,\partial\Omega,
\end{array}
\right.
\end{eqnarray}
(with the last condition understood in the sense of Definition~\ref{Ian-Yab679}
specialized to the case when $D=\emptyset$) has a solution if and only if 
\begin{eqnarray}\label{S-EYnn.Lan}
{}_{\big(W^{m,p'}(\Omega)\big)^*}\langle f,v\rangle_{W^{m,p'}(\Omega)}=0,
\qquad\forall\,v\in{\mathcal{P}}_{m-1}(\Omega),
\end{eqnarray}
in which scenario solutions of \eqref{Tgb-5bb.Kan} are unique up to 
functions from ${\mathcal{P}}_{m-1}(\Omega)$.
\end{theorem}

\begin{proof}
This is an immediate consequence of Theorem~\ref{yaUNDJ} specialized to the
case in which $D:=\emptyset$ (cf. also Lemma~\ref{LLDEnse} in this regard). 
\end{proof}

\vspace{0.08in}
\noindent --------------------------------------
\vspace{0.08in}

\noindent\begin{minipage}[t]{7.5cm}
\noindent {\tt Kevin Brewster}

\noindent Department of Mathematics

\noindent University of Missouri at Columbia

\noindent Columbia, MO 65211, USA

\vskip 0.08in

\noindent {\tt e-mail}: {\it kjb886\@@mail.mizzou.edu}

\vskip 0.10in

\noindent {\tt Dorina Mitrea}

\noindent Department of Mathematics

\noindent University of Missouri at Columbia

\noindent Columbia, MO 65211, USA

\vskip 0.08in

\noindent {\tt e-mail}: {\it mitread\@@missouri.edu}

\end{minipage}

\vskip 0.15in

\noindent\begin{minipage}[t]{7.5cm}
\noindent {\tt Irina Mitrea}

\noindent Department of Mathematics

\noindent Temple University

\noindent Philadelphia, PA 19122, USA

\vskip 0.08in

\noindent {\tt e-mail}: {\it imitrea\@@temple.edu}

\vskip 0.10in

\noindent {\tt Marius Mitrea}

\noindent Department of Mathematics

\noindent University of Missouri at Columbia

\noindent Columbia, MO 65211, USA

\vskip 0.08in

\noindent {\tt e-mail}: {\it mitream\@@missouri.edu}
\end{minipage}
\end{document}